%% file: Finalversionanyp.tex
\documentclass[a4paper,oneside,11pt]{article}
\usepackage{geometry}                % See geometry.pdf to learn the layout options. There are lots.
\usepackage[active]{srcltx}
\usepackage{hyperref}

\usepackage{graphicx}
\usepackage{amssymb}
\usepackage{epstopdf}
\usepackage{amsmath}
\usepackage{amsthm}
\usepackage{amssymb}
\usepackage{latexsym}
\usepackage{mathtools}
\usepackage{mathrsfs}
\usepackage{graphics}
\usepackage{latexsym}
\usepackage{psfrag}
\usepackage{import}
\usepackage{verbatim}
\usepackage{enumerate}
\usepackage{enumitem}
\usepackage{graphicx}
\usepackage[usenames]{color}

\usepackage{xcolor}

\newcommand{\as}{a,s}
\newcommand{\awos}{ a}

\newcommand{\R}{\mathbb{R}}

\newcommand{\NN}{\mathcal{N}}

\newcommand{\sgn}{\text{\rm sgn}}

\newcommand{\Eps}{\mathcal{E}}
\newcommand{\sing}{\text{sing}}

\newcommand{\ve}{\varepsilon}

\newcommand{\f}{\varphi}

\newcommand{\T}{\mathcal{T}}

\renewcommand{\L}{\mathcal{L}}

\newcommand{\RCD}{\mathsf{RCD}}

\newcommand{\CD}{\mathsf{CD}}

\newcommand{\Geo}{{\rm Geo}}
\newcommand{\MCP}{\mathsf{MCP}}

\newcommand{\OptGeo}{{\rm OptGeo}}

\newcommand{\abs}[1]{\left\vert#1\right\vert}
\newcommand{\set}[1]{\left\{#1\right\}}
\newcommand{\brac}[1]{\left(#1\right)}

\newcommand{\Real}{\mathbb{R}}

\newcommand{\eps}{\varepsilon}

\renewcommand{\L}{\mathcal{L}}
\newcommand{\m}{\mathfrak{m}}
\newcommand{\q}{\mathfrak{q}}

\renewcommand{\P}{\mathbb P}
\newcommand{\ellc}{{\bar{\ell}}}
\newcommand{\varphic}{{\bar{\varphi}}}
\newcommand{\Phic}{{\bar{\Phi}}}

\renewcommand{\P}{\mathcal{P}}
\newcommand{\len}{\ell}

\newcommand{\spt}{{\rm spt}}
\newcommand{\mm}{\mathfrak m}
\newcommand{\qq}{\mathfrak q}

\newcommand{\ee}{{\rm e}}
\newcommand{\QQ}{\mathfrak Q}

\newcommand{\sfd}{\mathsf d}

\newcommand{\Opt}{\mathrm{Opt}}

\theoremstyle{plain}
\newtheorem{lemma}{Lemma}[section]
\newtheorem{theorem}[lemma]{Theorem}

\newtheorem{proposition}[lemma]{Proposition}
\newtheorem{corollary}[lemma]{Corollary}
\newtheorem*{theorem*}{Theorem}
\newtheorem*{maintheorem*}{Main Theorem}

\theoremstyle{definition}

\newtheorem{definition}[lemma]{Definition}
\newtheorem*{definition*}{Definition}
\newtheorem{remark}[lemma]{Remark}
\newtheorem*{example*}{Example}
\newtheorem{example}[lemma]{Example}

\numberwithin{equation}{section}

\newcounter{mycounter}
\begin{document}
\title{Independence of synthetic Curvature Dimension conditions on transport distance exponent\thanks{RM's research is supported in part by Natural Sciences and Engineering Research Council of Canada
Discovery Grants RGPIN--2015--04383 and 2020--04162.
}
}
\author{Afiny Akdemir\thanks{Department of Mathematics, 
University of Toronto, 
%40 St George St,
Toronto Ontario, 
Canada M5S 2E4
\tt afiny@math.toronto.edu, andrew.colinet@mail.utoronto.ca, mccann@math.toronto.edu}, 
Fabio Cavalletti\thanks{Mathematics Area, SISSA, Trieste (Italy) {\tt cavallet@sissa.it, fsantarc@sissa.it}}, 
Andrew Colinet$^\dagger$, 
Robert McCann$^\dagger$, 
Flavia Santarcangelo$^\ddagger$} 

\maketitle

\begin{abstract}
The celebrated Lott-Sturm-Villani theory of metric measure spaces furnishes
synthetic notions of  a Ricci curvature lower bound $ K$ joint with an upper bound $ N$ on the dimension.  
Their condition, called  the Curvature-Dimension condition and denoted by $\CD(K,N)$, 
is formulated in terms of a modified displacement convexity of an entropy functional 
along $W_{2}$-Wasserstein geodesics.
We  show that the choice of the squared-distance function as transport cost does not 
influence the theory.   By denoting  with $\CD_{p}(K,N)$ the analogous condition but with the cost as the $p^{th}$ power of the distance, we show that  $\CD_{p}(K,N)$ are all equivalent conditions for any $p>1$  --- at least in spaces whose
geodesics do not branch.

Following \cite{CMi},  we show that the trait d'union between all the seemingly unrelated $\CD_{p}(K,N)$ conditions  is the needle decomposition or  localization technique associated to the
$L^{1}$-optimal transport problem. We also establish  the local-to-global property of $\CD_{p}(K,N)$ spaces.
\end{abstract}

\tableofcontents

\section{Introduction}
The theory of optimal transport has been successfully used to  study
geometric and analytic properties of possibly singular spaces.  In their seminal works, Lott--Villani \cite{lottvillani:metric} and Sturm \cite{sturm:I, sturm:II} have established
for metric measure spaces $(X,\sfd,\m)$  consisting of a complete separable metric space $(X,\sfd)$
endowed with 
a Radon reference measure $\mm$,   a synthetic condition which encodes, in a generalized sense, a 
 combined lower bound  $K\in \R$ on the Ricci curvature  and upper bound  $N\in [1,\infty)$ on the dimension. Their condition is called  the Curvature-Dimension condition $\CD(K,N)$;
a general account 
on its history, huge developments and impacts  
goes far beyond the scope of this introduction. 

For our purposes, the cornerstone of the  Curvature-Dimension condition is the equivalence between 
a lower bound on the Ricci curvature  combined with an upper bound on the dimension
 and a certain convexity property of an entropy functional along $W_2$-Wasserstein geodesics in the setting of weighted manifolds
 \cite{CorderoMcCannSchmuckenschlager01} \cite{SturmvonRenesse05}.
The strength of the optimal transport approach permitted Lott--Villani and Sturm 
to obtain a stable notion  of convergence, with 
stability intended with respect to a suitable distance over the class of metric measure spaces.
  We refer to Section \ref{Ss:CD} for precise definitions.

 As the $\CD(K,N)$ condition for smooth manifolds is 
equivalent to a joint lower bound on the Ricci curvature and 
an upper bound on the dimension, it is natural to consider whether the squared-distance  cost function plays a special role in the theory.  Among  the possible transport cost functions, 
the power distance costs, namely $\sfd^{p}$ with $p>1$, are related to the geometry of the underlying space.
The  power distance costs have already appeared in the literature  in  the definition of the $p$-Wasserstein distance
$W_{p}$ that turns the space of probability measures with finite $p^{th}$-moments into a 
complete and separable metric space $(\mathcal{P}_{p}(X),W_{p})$.   Another natural setting for such spaces can also be seen in the case of  doubly-degenerate diffusion dyanamics \cite{Otto90s} \cite{Agueh03}. 
Accordingly, the modified displacement convexity of the entropy functional can be considered 
with respect to $W_{p}$-geodesics  -- and this in turn  furnishes  a straightforward and legitimate extension of the definition of $\CD(K,N)$ condition  proposed by Kell \cite{KellWp} and  denoted by $\CD_{p}(K,N)$.
The notation $\CD(K,N)$ will be reserved for the classical case $p = 2$. While  Kell established the equivalence of all $\CD_{p}(K,N)$ in the smooth setting  via the use of
Ricci curvature,  no previous results are known in the context of nonsmooth metric measure spaces.

 Our approach to obtaining this equivalence in the nonsmooth setting utilizes  two closely related $L^1$ optimal transportation curvature dimension conditions introduced by  Cavalletti and Milman \cite{CMi}, which we denote by $\CD^{1}(K,N)$ and $\CD^1_{Lip}(K,N)$.
The $\CD^{1}(K,N)$ condition has been successfully used in \cite{CMi} to establish the local-to-global property of $\CD(K,N)$ spaces.  Cavalletti and Milman's formulation is partially based on the  needle or localization paradigm 
introduced by Klartag \cite{klartag} in the smooth setting,  which was later generalized to the metric setting 
by Cavalletti and Mondino \cite{CM17a}.

Cavalletti and Milman established the local-to-global property by demonstrating the equivalence of the local version of $\CD(K,N)$ condition (namely $\CD_{loc}(K,N)$, Definition \ref{def:CDKNloc})  to the $\CD^1 (K,N)$ condition.  
In particular, the trait d'union between all  of the  curvature-dimension conditions 
 is in the behaviour of the  gradient flow lines of signed-distance functions, also known as transport rays.

 In this paper we will use this point of view to link two different curvature dimension conditions:  we will demonstrate the equivalence of $\CD_{p}(K,N)$ and $\CD_ {q}(K,N)$ for a 
general  metric measure space $(X,\sfd,\mm)$, for $p,q>1$ and $K,N\in\mathbb{R}$ with $N>1$,   under the requirement that  $(X,\sfd,\mm)$  is either non-branching or at least satisfies
appropriate versions of the
essential non-branching  condition of Definition \ref{D:PENB}.
More specifically, we obtain the following results:
\begin{theorem}[Equivalence of $\CD_{p}$ on $p>1$]\label{thm:main1}
Let $(X,\sfd,\m)$ be such that $\mm(X) = 1$. 
Assume  $(X,\sfd,\m)$  is  $p$-essentially non-branching and verifies $\CD_p(K,N)$ for some $p >1$. 
If $(X,\sfd,\m)$ is also $q$-essentially non-branching  for some $q>1$, then
it verifies $\CD_q(K,N)$.                                                                                                   
\end{theorem}
 Recently, Cavalletti, Gigli, and Santarcangelo  \cite{CGS20}  have characterized $\CD^{1}_{Lip}(K,N)$  in terms of a modified displacement convexity of an entropy functional along 
a certain family of $W_{1}$ Wasserstein geodesics.
Hence, Theorem \ref{thm:main1} completes the picture  by showing that for any $p \geq 1$, 
all  of the $\CD_{p}(K,N)$ conditions,  when expressed in terms of displacement convexity, are equivalent,
provided the space $X$ satisfies the appropriate essentially non-branching condition.

Since we employ the strategy used in \cite{CMi} to distance costs with powers other  than $p =2$, 
we also establish the local-to-global property for $\CD_{p}(K,N)$ spaces.

\begin{theorem}[Local-to-Global]\label{thm:main2}
Fix any $p>1$ and $K,N \in \R$ with $N>1$. 
Let $(X,\sfd,\mm)$ be a $p$-essentially non-branching metric measure space verifying $\CD_{p,loc}(K,N)$ 
 from Definition \ref{def:CDKNloc} and such that 
$(X,\sfd)$ is a length space with $\spt (\mm) = X$ and $\mm(X) = 1$. 
Then $(X,\sfd,\mm)$ verifies $\CD_{p}(K,N)$.
\end{theorem}

 In Theorem \ref{thm:main1} and Theorem \ref{thm:main2} we are assuming 
$\mm(X) = 1$.  This assumption is also used  in \cite{CMi} but  we believe that it is most likely  a purely technical assumption. At the moment, the main obstacle to the case of a general Radon measure 
$\mm$ is the lack of a canonical disintegration theorem once a ``measurable'' 
partition is given. For some preliminary results in this direction we refer to 
\cite{CM18a}.

 Another motivation to studying distance costs with powers other than $p=2$ comes from the recent works of McCann \cite{McCann} and Mondino-Suhr \cite{MoSu},  where the authors analyze the relation between optimal transportation and timelike Ricci curvature bounds in the smooth Lorentzian  setting. Analogously to the Riemannian setting, timelike Ricci curvature lower bounds can be equivalently characterised in terms of convexity properties of the Bolzmann-Shannon entropy functional along $\ell_{p}$-geodesics of probability measures, where 
$\ell_{p}$ denotes the causal  transport  distance  with exponent $p \in (0,1]$.
This point of view has been pushed forward in \cite{CM20a} 
and \cite{McCann20p} where   the authors proposed a synthetic formulation of the  
Strong Energy condition, denoted by $\mathsf{TCD}_{p}(K,N)$, which is valid for non-smooth Lorentzian spaces.  Unlike the Riemmannian case, the Lorentzian setting does not have a  distinguished $p$;  and one of the next steps of the theory will be to address whether 
$TCD_p(K,N)$ depends on $p$ or not.
%%%%%%%%%%%%%%%%%%%%%%%%%%%%%%%%%%%%%%%%%%%%%%%%%%%%%%%%%%%%%%%%%%%%%%%%%%%%%%%%%%%%%
%%%%%%%%%%%%%%%%%%%%%%%%%%%%%%%%%%%%%%%%%%%%%%%%%%%%%%%%%%%%%%%%%%%%%%%%%%%%%%%%%%%%%

\subsection{Structure of the paper}\label{Ss:structure}
We start this note by recalling basic definitions of Optimal Transport  as
well as the notions of synthetic lower curvature bounds as introduced by 
Lott-Sturm-Villani in Section \ref{Ss:preprequisites}.

Section \ref{S:HopfLax} is devoted to a careful analysis of Kantorovich 
potentials and their evolution via the Hopf-Lax semigroup with a general exponent 
$p > 1$. In particular, we will obtain second and third order information                                    
on the time behaviour of $\varphi_{t}$ leading to the fundamental Theorem 
\ref{teo:zz} where a new third order inequality is obtained that plays
a crucial role in the rest of the paper.

In Section \ref{S:firstimplication} we show that a local version of 
$\CD_{p}(K,N)$ 
implies $\CD^{1}(K,N)$ in the version reported in Theorem \ref{teo:cm}.
Finally, in Section \ref{S:anyp} we obtain a complete equivalence 
of all $\CD_{p}(K,N)$ (Theorem \ref{thm:main1}) and 
each of them also enjoys the local-to-global property (Theorem \ref{thm:main2}).
%%%%%%%%%%%%%%%%%%%%%%%%%%%%%%%%%%%%%%%%%%%%%%%%%%%%%%%%%%%%%%%%%%%%%%%%%%%%%%%%%%%%%
%%%%%%%%%%%%%%%%%%%%%%%%%%%%%%%%%%%%%%%%%%%%%%%%%%%%%%%%%%%%%%%%%%%%%%%%%%%%%%%%%%%%%
\subsection{Brief Overview}\label{Ss:overview}
Throughout this  overview we will be working on a metric measure space $(X,\sfd,\mm)$ 
satisfying suitable hypotheses.
We will also be considering the transport of a measures $\mu_{0}$ to $\mu_{1}$ where both measures are absolutely continuous with
respect to $\mm$.
We denote the interpolant measure by $\mu_{t}$ and we write $\rho_{t}$ for their densities with respect to $\mm$.

In Section $\ref{S:HopfLax}$ the goal is to obtain information about the time derivative of the $t$-propagated $s$-Kantorovich potential $\Phi_{s}^{t}$ as
defined in Section $\ref{Ss:IntermediateKantorovichPotential}$.
This quantity is  crucial for the Jacobian factor that appears when comparing interpolant measures, $\mu_{t}$, between measures
$\mu_{0}$ and $\mu_{1}$ along a transport geodesic at two times.
To achieve this goal, Section $\ref{Ss:HopfLaxDef}$ - $\ref{Ss:FirstAndSecondRegularity}$ are dedicated to a detailed study of the regularity properties of
the Hopf-Lax transform.
In particular  we establish second order regularity  for the Hopf-Lax transform of a Kantorovich potential as well as a few identities related to the positional
information stored in a Kantorovich potential.
From here, Section $\ref{Ss:ThirdOrder}$ demonstrates, through a delicate argument, third order temporal regularity of time propagated
Kantorovich potentials along transport geodesics.

In Section $\ref{S:firstimplication}$, we remind the reader of the standard definitions of $L^{1}$-optimal transport.
 We show in Section $\ref{TransportSets}$ that the non-branched transport set partitions a space into transport rays.
 This partition allows us to decompose measures into a collection of  one-dimensional conditional measures concentrated on transport rays.
This disintegration also gives the advantage of passing curvature information from the total space down to the $L^1$-transport rays 
no matter from which $\CD_{p}(K,N)$ we are starting,  as we show in Section $\ref{Ss:curvature}$.
This is highlighted in Theorem $\ref{teo:cm}$ where 
we demonstrate that any 
$p$-essentially non-branching metric measure space verifying $\CD_{p}(K,N)$ 
also verifies $\CD^{1}_{Lip}(K,N)$.
This will be useful in Section $\ref{S:anyp}$ when we want to compare the restriction of measure to a Kantorovich geodesic at two different times.
To propagate a measure from one time to another we will use the time propagated Kantorovich potential  from Section $\ref{S:HopfLax}$.

In Section $\ref{S:anyp}$, the goal is to transfer the curvature properties along transport geodesics back to the total space through $q$-Wasserstein
geodesics and hence proving that an enhanced version of $\CD^{1}_{Lip}(K,N)$
implies $\CD_{q}(K,N)$.
This will be done by proving, in the terminology of \cite{CMi},
an ``LY''-decomposition for the densities $\rho_{t}$ of the $q$-Wasserstein geodesic $\mu_{t}$ (see Theorem $\ref{T:decomposition}$). 
More precisely, this ``LY''-decomposition provides a factorization of the ratio 
$\rho_{t}/\rho_{s}$ into two factors: the first one --- denoted by $L$  --- is a concave function taking 
into account only the one dimensional distortion due to the volume stretching in the direction 
of the geodesic. The second factor is denoted by $Y$ and contains the 
volume distortion in the transversal directions.

To achieve this goal we first use the Disintegration theorem from 
Section $\ref{S:firstimplication}$ to represent $\mm$ as an average of measures that live on
$L^{1}$-transport geodesics for the signed distance to any given level set of a $p$-Kantorovich potential.
In this disintegration of $\mm$ we follow the evolution of a specific collection of Kantorovich geodesics.
More specifically, we fix $a\in\mathbb{R}$ and $s\in(0,1)$, and consider $q$-Kantorovich geodesics $\gamma$ which satisfy $\varphi_{s}(\gamma_s)=a$, 
where $\varphi_{s}$ is the  evolved % interpolation of a 
Kantorovich potential for the $q$-Wasserstein geodesic.
We denote such geodesics by $G_{a,s}$ and we disintegrate $\mm$  %on $\ee_{[0,1]}(G_{\as})$  
over $\{\gamma_t \colon \gamma \in G_{a,s}\}_{t \in [0,1]}$ to obtain
\begin{equation*}
\mm\llcorner_{\ee_{[0,1]}(G_{a,s})}= \int_{[0,1]} \mm_t^{a,s}\, \L^1(dt),
\end{equation*}
Then we compare this to a disintegration of $\mm$  over $\{\varphi_s^{-1}(a)\}_{a \in \R}$ on the time $t$ evaluation of a 
 sufficiently large set of Kantorovich geodesics denoted by $G$.
Specifically, we obtain
\begin{equation*}
\mm \llcorner_{\ee_{t}(G)}= \int_{\varphi_s(\ee_s(G))} \mm^t_{a,s}\mathcal{L}^1(da)
\end{equation*}
This leads to two measures,  $\mm^{a,s}_t$  and $\mm^t_{a,s}$, that live on $\ee_{t}(G_{a,s})$.
In Section $\ref{Ss:comparison}$ we compare these two disintegrations to deduce that $\mm_t^{a,s}$ and $\mm^t_{a,s}$ differ only by
$\partial_{t}\Phi_{s}^{t}$.
This information is used in Section $\ref{Ss:change}$ to deduce the Jacobian factor between $\rho_{t}(\gamma_{t})$ and $\rho_{s}(\gamma_s)$.
This formula for the Jacobian factor allows us to conclude the desired ``LY'' decomposition. Once the ``LY'' decomposition is at  our disposal, we can invoke \cite{CMi} to conclude 
that the space satisfies $\CD_{q}(K,N)$.
%%%%%%%%%%%%%%%%%%%%%%%%%%%%%%%%%%%%%%%%%%%%%%%%%%%%%%%%%%%%%%%%%%%%%%%%%%%%%%%%%%%%%
%%%%%%%%%%%%%%%%%%%%%%%%%%%%%%%%%%%%%%%%%%%%%%%%%%%%%%%%%%%%%%%%%%%%%%%%%%%%%%%%%%%%%
%%%%%%%%%%%%%%%%%%%%%%%%%%%%%%%%%%%%%%%%%%%%%%%%%%%%%%%%%%%%%%%%%%%%%%%%%%%%%%%%%%%%%
%%%%%%%%%%%%%%%%%%%%%%%%%%%%%%%%%%%%%%%%%%%%%%%%%%%%%%%%%%%%%%%%%%%%%%%%%%%%%%%%%%%%%
\section{Prerequisites}\label{Ss:preprequisites}
\subsection{Geodesics and Measures}
\label{Ss:geomeas}

Let $(X,\sfd)$ be a complete and separable metric space.
A map $\gamma : [0,1] \to X$ satisfying 
$\sfd(\gamma_{t},\gamma_s) = |t - s| \sfd(\gamma_{0},\gamma_{1})$  for all $s,t \in [0,1]$
is called a geodesic connecting $\gamma_{0}$ to $\gamma_{1}$. 
We regard $\Geo(X)$ as a subset of all Lipschitz  curves $\text{Lip}([0,1], X)$ endowed with the uniform topology. 

We say the metric space $(X, \sfd)$ is a geodesic metric space if  for each $x, y \in
X$ there is a geodesic connecting $x$ and  $y$.

 For any $t\in [0,1]$, we denote the continuous evaluation map $\ee_{t} : \Geo(X) \to X$ 
as $\ee_{t}(\gamma) = \gamma_{t}$.
We will also adopt the following abbreviations: 
given $I \subset [0,1]$ and $G \subset \Geo(X)$
\begin{align*}
\ee_t(G) = G(t) & =  \set{ \gamma_t \; ; \; \gamma \in G }, 
\quad \ee_I(G) := \cup_{t \in I} \ee_t(G).
\end{align*}

The space of all Borel probability measures over $X$  is denoted by $\mathcal{P}(X)$, and  
$\mathcal{P}_{p}(X)$  is the subspace of $\mathcal{P}(X)$ consisting of measures with finite $p^{th}$-moment.   Given a non-negative Radon measure $\mm$, we call the space
$(X,\sfd,\mm)$ a metric measure space, and   $\mathcal{P}_{p}(X,\sfd,\mm)$ 
will denote the subspace of $\mathcal{P}_{p}(X)$ consisting of probability measures  that are
absolutely continuous with respect to $\mm$.   Unless otherwise noted, we assume $\mm(X)=1$ to permit disintegration of $\mm$ into conditional measures as needed.
For any $p\geq 1$,
the $L^{p}$-Wasserstein distance $W_{p}$ is defined  for any  $\mu_0,\mu_1 \in \mathcal{P}(X)$ as
\begin{equation}\label{eq:Wdef}
  W_p^p(\mu_0,\mu_1) := \inf_{ \pi\in \Pi(\mu_{0},\mu_{1})} \int_{X\times X} \sfd^p(x,y) \, \pi(dx , dy),
\end{equation}
where $\Pi(\mu_{0},\mu_{1})$ is the set of $\pi \in \mathcal{P}(X\times X)$ with 
$(P_{1})_{\sharp}\pi=\mu_0$ and $(P_{2})_{\sharp}\pi=\mu_1$.

It is known that the infimum in (\ref{eq:Wdef}) is always attained for any $\mu_0,\mu_1 \in \mathcal{P}(X)$; the set of optimal plans will be denoted by $\Opt_{p}(\mu_{0},\mu_{1})$.

As $(X,\sfd)$ is a complete and separable metric space, so is $(\mathcal{P}_p(X), W_p)$. 
 It is also known that $(X,\sfd)$ is geodesic if and only if $(\mathcal{P}_p(X), W_p)$ is geodesic.
Moreover, if $(X,\sfd)$ is a geodesic space, then the following two statements are equivalent (see for instance 
\cite[Theorem 3.10]{ambro:userguide}):
\begin{itemize}
\item $[0,1] \ni t \mapsto \mu_{t} \in\mathcal{P}_{p}(X)$ is a $W_{p}$-geodesic;
\item there exists $\nu \in \mathcal{P}(\Geo(X))$ such that $(\ee_{0},\ee_{1})_{\sharp}\nu \in \Opt_{p}(\mu_{0},\mu_{1})$ and $\mu_{t} = (\ee_{t})_{\sharp} \nu$.
\end{itemize}
The set of $\nu \in \mathcal{P}(\Geo(X))$ verifying the last point are 
called dynamical optimal plans and are denoted by $\OptGeo_{p}(\mu_{0},\mu_{1})$. Notice that 
if $\nu \in \OptGeo_{p}(\mu_{0},\mu_{1})$, then also $(\ee_{t},\ee_{s})_{\sharp}\nu$ is $p$-optimal between its marginals.

\begin{definition}[\emph{p}-Essentially Non-Branching]\label{D:PENB}
A subset $G \subset \Geo(X)$ of geodesics is called non-branching if for any $\gamma^{1}, \gamma^{2} \in G$ the following holds:
$$
\gamma_{0}^{1} = \gamma_{0}^{2}, \ \gamma_{\bar t}^{1} = \gamma_{\bar t}^{2}, \ \bar t\in (0,1)  
\quad 
\Longrightarrow 
\quad 
\gamma^{1}_{s} = \gamma^{2}_{s}, \quad \forall s \in [0,1].
$$
 The space $(X,\sfd)$ is called \emph{non-branching} if $\Geo(X)$ is non-branching; 
$(X,\sfd, \mm)$ is called \emph{p-essentially non-branching}  if for all $\mu_{0},\mu_{1} \in \mathcal{P}_{p}(X,\sfd,m)$, any $\nu \in \OptGeo_{p}(\mu_{0},\mu_{1})$ is concentrated on a Borel non-branching set $G\subset \Geo(X)$, in agreement with the terminology of \cite{RS2014} when $p=2$.
\end{definition}
We remark that examples of branched spaces which are essentially non-branching may be found in Ohta \cite{Ohta2}.
%%%%%%%%%%%%%%%%%%%%%%%%%%%%%%%%%%%%%%%%%%%%%%%%%%%%%%%%%%%%%%%%%%%%%%%%%%%%%%%%%%%%%
%%%%%%%%%%%%%%%%%%%%%%%%%%%%%%%%%%%%%%%%%%%%%%%%%%%%%%%%%%%%%%%%%%%%%%%%%%%%%%%%%%%%%
\subsection{Curvature-Dimension conditions}\label{Ss:CD}
We recall the definition of volume distortion coefficients.
\begin{definition}[$\sigma_{K,\NN}$-coefficients] \label{def:sigma}
Given $K \in \Real$ and $\NN \in (0,\infty]$, define:
\[
D_{K,\NN} := \begin{cases}  \frac{\pi}{\sqrt{K/\NN}}  & K > 0 \;,\; \NN < \infty,\\ +\infty & \text{otherwise}.\end{cases}
\]
In addition, given $t \in [0,1]$ and $0 \leq \theta < D_{K,\NN}$, define:
\[
\sigma^{(t)}_{K,\NN}(\theta) := \frac{\sin(t \theta \sqrt{\frac{K}{\NN}})}{\sin(\theta \sqrt{\frac{K}{\NN}})} = 
\begin{cases}   
\frac{\sin(t \theta \sqrt{\frac{K}{\NN}})}{\sin(\theta \sqrt{\frac{K}{\NN}})}  & K > 0 \;,\; \NN < \infty \\
t & K = 0 \text{ or } \NN = \infty \\
 \frac{\sinh(t \theta \sqrt{\frac{-K}{\NN}})}{\sinh(\theta \sqrt{\frac{-K}{\NN}})} & K < 0 \;,\; \NN < \infty 
\end{cases}                                                                                                                                                        
\]
and set $\sigma^{(t)}_{K,\NN}(0) = t$ and $\sigma^{(t)}_{K,\NN}(\theta) = +\infty$ for $\theta \geq D_{K,\NN}$. 
\end{definition}

\begin{definition}[$\tau_{K,N}$-coefficients]
Given $K \in \Real$ and $N=\mathcal{N}+1 \in (1,\infty]$, define:
\[
\tau_{K,N}^{(t)}(\theta) := t^{\frac{1}{N}} \sigma_{K,N-1}^{(t)}(\theta)^{1 - \frac{1}{N}} .
\]
When $N=1$, set $\tau^{(t)}_{K,1}(\theta) = t$ if $K \leq 0$ and $\tau^{(t)}_{K,1}(\theta) = +\infty$ if $K > 0$. 
\end{definition}

We will use the following definition introduced in \cite{sturm:II} for the case $p=2$. 
Recall that given $N \in [1,\infty)$, the $N$-R\'enyi relative-entropy functional 
$\mathcal{E}_N : \P(X) \rightarrow [0,\infty]$ is defined as:
$$
\Eps_N(\mu) := \int \rho^{1 - \frac{1}{N}} d\mm ,
$$
where $\mu = \rho \mm + \mu^{\sing}$ is the Lebesgue decomposition of $\mu$ with $\mu^\sing \perp \mm$. It is known \cite{sturm:II} that $\mathcal{E}_N$ is upper semi-continuous with respect to the weak topology on $\P(X)$. 

\begin{definition}[$\CD_{p}(K,N)$] \label{def:CDKN}
Given $K,N \in \R$ with $N\geq 1$, $(X,\sfd,\mm)$ is said to satisfy $\CD_{p}(K,N)$ if for all 
$\mu_0,\mu_1 \in \P_p(X,\sfd,\mm)$, there exists $\nu \in \OptGeo_{p}(\mu_0,\mu_1)$ so that for all $t\in[0,1]$, $\mu_t := (\ee_t)_{\sharp} \nu \ll \mm$, and for all $N' \geq N$:
\begin{equation} \label{eq:CDKN-def}
\Eps_{N'}(\mu_t) \geq \int_{X \times X} \brac{\tau^{(1-t)}_{K,N'}(\sfd(x_0,x_1)) \rho_0^{-1/N'}(x_0) + \tau^{(t)}_{K,N'}(\sfd(x_0,x_1)) \rho_1^{-1/N'}(x_1)} \pi(dx_0,dx_1) ,
\end{equation}
where $\pi = (\ee_0,\ee_1)_{\sharp}(\nu)$ and $\mu_i = \rho_i \mm$, $i=0,1$. 
\end{definition}

When we omit the subscript $p$ from $\CD_{p}(K,N)$, we tacitly mean                                                                                                         
the classical $p=2$, as introduced independently by Lott-Villani in \cite{lottvillani:metric} and Sturm in \cite{sturm:I,sturm:II}.

As a natural curvature notion, $\CD_{p}(K,N)$ has a local version that is denoted by 
$\CD_{p,loc}(K,N)$.

\begin{definition}[$\CD_{p,loc}(K,N)$] \label{def:CDKNloc}
Given $K,N \in \R$ with $N\geq 1$, $(X,\sfd,\mm)$ is said to satisfy $\CD_{p,loc}(K,N)$ 
if for any $o \in \spt(\mm)$, there exists a neighborhood $X_o \subset X$ of $o$, so that for all $\mu_0,\mu_1 \in \P_p(X,\sfd,\mm)$ supported in $X_o$, there exists $\nu \in \OptGeo_{p}(\mu_0,\mu_1)$ so that for all $t\in[0,1]$, $\mu_t := (\ee_t)_{\sharp} \nu \ll \mm$, and for all $N' \geq N$, (\ref{eq:CDKN-def}) holds. 
\end{definition}
\noindent
Note that $(\ee_t)_{\sharp} \nu$  from the definition of $\CD^{p}_{loc}(K,N)$  is not required to be supported in $X_o$ for intermediate times $t \in(0,1)$.

Requiring the $\CD(K,N)$ condition to hold whenever $\mu_1$ degenerates to $\delta_o$, a delta-measure at $o \in \spt(\mm)$, goes by the name of  Measure Contraction Property and is denoted by $\MCP(K,N)$.
 This property was introduced independently by Ohta in \cite{Ohta1} and Sturm in \cite{sturm:II}. 
Since $\OptGeo_{p}(\mu_0, \delta_{o} )$ does not depend on $p$, whenever $p > 1$, 
the  superscript will be omitted. We now record the version of the definition of $\MCP(K,N)$ found in \cite{Ohta1}.

\begin{definition}[$\MCP(K,N)$] \label{D:Ohta1}
The space $(X,\sfd,\mm)$ is said to satisfy $\MCP(K,N)$ if for any $o \in \spt(\mm)$ and  $\mu_0 \in \P_2(X,\sfd,\mm)$ of the form $\mu_0 = \frac{1}{\mm(A)} \mm\llcorner_{A}$ for some Borel set $A \subset X$ with $0 < \mm(A) < \infty$ (and with $A \subset B(o, \pi \sqrt{(N-1)/K})$ if $K>0$), there exists $\nu \in \OptGeo_{2}(\mu_0, \delta_{o} )$ such that:
\begin{equation} \label{eq:MCP-def}
\frac{d}{d\mm} \left[(\ee_{t})_{\sharp} \big( \tau_{K,N}^{(1-t)}(\sfd(\gamma_{0},\gamma_{1}))^{N} \nu(d \gamma) \big)\right]  
\le \frac{1}{\mm(A)}  
\;\;\; \forall t \in [0,1] .
\end{equation}
\end{definition}

As one would expect, $\CD_{p}(K,N)$ implies $\MCP(K,N)$ 
(see \cite[Lemma 6.11]{CMi} for the case $p =2$; the proof works the same 
for any $p > 1$), without any type of essential non-branching. When coupled with the $p$-essentially non-branching condition, 
$\MCP$ yields nice properties for $W_{p}$-geodesics. 
 A weaker contraction property 
\cite{CHues} of $(X,\sfd,\mm)$ is called 
{\em qualitative non-degeneracy}, which asserts for each ball $B_R(x_0)$,
there is a ratio $f(t)\in(0,1]$ with $\limsup_{t\to 0}f(t)>1/2$ 
which bounds the decrease  in measure whenever any Borel set 
$A \subset B_R(x_0)$ is contracted a fraction $t$
of the distance towards any $x \in B_R(x_0)$:
\begin{equation}\label{QND}
\mm(\ee_{t}(G)) \ge f(t) \mm(\ee_{0}((G))
\end{equation}
for $G=(\ee_{0} \times \ee_{1})^{-1}(A \times \{x\})$.
 Thus $\MCP$ permits one to invoke the following:
\begin{theorem}[Optimal dynamic transport is unique iff the space is essentially non-branching \cite{KellENB}]\label{teo:kell}
Let $(X,\sfd, \m)$ be a metric measure space with $\mm$ 
qualitatively non-degenerate. Then the following properties are equivalent:
\begin{enumerate}
\item[i)] $(X,\sfd,\m)$ is $p$-essentially non-branching;

\item[ii)]  
for every $\mu_0,\mu_1\in \mathcal{P}_p(X)$ with $\mu_0 \ll \mm$ there is a unique 
$\nu \in \OptGeo_{p}(\mu_{0},\mu_{1})$. Moreover, the $p$-optimal coupling 
$(\ee_0,\ee_1)_{\sharp}\nu$ is induced by a transport map and each interpolant
$\mu_t=(\ee_t)_{\sharp}\nu$, where $t \in{[0,1)}$, is absolutely 
continuous with respect to $\mm$.
\end{enumerate}
\end{theorem}
\begin{remark}\label{R:locp}
It is also worth recalling that the local version of 
$\CD(K,N)$, denoted by $\CD_{loc}(K,N)$,  is known to 
imply $\MCP(K,N)$ provided that  
$(X,\sfd)$ is a non-branching length space, see \cite{cavasturm:MCP}.
Since any $\CD_{p,loc}(K,N)$ gives the same information when considered                 
for Wasserstein geodesics arriving at a Dirac mass, we can conclude 
that the same argument of \cite{cavasturm:MCP}  shows
$\CD_{loc,p}(K,N)$ implies $\MCP(K,N)$, provided $(X,\sfd)$ is a non-branching length space.
 
Moreover, it has already been observed and used in the 
literature that the non-branching assumption 
can be weakened to essentially non-branching when $p =2$: 
the non-branching property in 
\cite{cavasturm:MCP} was used to obtain a partition of 
$X$ formed of all geodesics arriving at the same point $o \in X$
and subsequently to ensure uniqueness of 
a dynamical optimal plan connecting $\mu_{0}$ to $\mu_{1}$
with $\mu_{0}\ll \mm$. Both properties 
can be deduced from $p$-essentially non-branching together 
with Theorem \ref{teo:kell}; for more details 
see Section \ref{Ss:L1OT}. 
Hence,  we will tacitly use
that  for each $p>1$, a metric measure space                                                                                        
satisfying $\CD_{p,loc}(K,N)$ and which is
a $p$-essentially non-branching length space also verifies 
$\MCP(K,N)$. 
\end{remark}

We conclude this subsection with the $\CD^{1}(K,N)$ condition introduced 
in \cite{CMi}. Notice that this definition uses the additional assumption 
that $\mm(X) = 1$.
We will also need to recall some classical terminology from the distance cost optimal transport theory 
that we briefly recall. 
 
To any $1$-Lipschitz function $u : X \to \R$ there is a naturally associated $\sfd$-cyclically monotone set   
\begin{equation}\label{TransportSet1}
\Gamma_{u} : = \{ (x,y) \in X\times X : u(x) - u(y) = \sfd(x,y) \},
\end{equation}
 which we call the {\em transport ordering.} We write $x \ge_u y$ if and only if $(x,y) \in \Gamma_u$; the $1$-Lipschitz condition on $u$ implies 
$\ge_u$ is a partial-ordering.
The \emph{transport relation} $R_u$ and the 
\emph{transport set} $\mathcal{T}_{u}$ are defined as:
\begin{equation}\label{E:R1}
R_{u} := \Gamma_{u} \cup \Gamma^{-1}_{u} ~,~ \mathcal{T}_{u} := P_{1}(R_{u} \setminus \{ x = y \}) ,
\end{equation}
where $\{ x = y\}$ denotes the diagonal $\{ (x,y) \in X^{2} : x=y \}$, 
$P_{i}$ the projection onto the $i$-th component 
and  $\Gamma^{-1}_{u}= \{ (x,y) \in X \times X : (y,x) \in \Gamma_{u} \}$.
Since $u$ is $1$-Lipschitz, $\Gamma_{u}, \Gamma^{-1}_{u}$ and $R_{u}$ are closed sets, and so are $\Gamma_u(x)$ and $R_u(x)$
(recall that $\Gamma_u(x) = \set{y \in X \; ;\; (x,y) \in \Gamma_u}$ and similarly for $R_u(x)$). 
Consequently $\mathcal{T}_{u}$ is a projection of a Borel set and hence analytic; it follows that it is universally measurable, and in particular, $\mm$-measurable \cite{Srivastava}.

Following \cite[Definition 7.7]{CMi}, a maximal chain $R$ in $(X,\sfd,\leq_u)$ is called a \emph{transport ray} if it is isometric to a closed interval $I$ in $(\Real,\abs{\cdot})$ of positive (possibly infinite) length.

\begin{definition}\label{D:defCD1}
($\CD^1_{u}(K,N)$ when $\spt(\m)=X$) Let $(X,\sfd,\m)$ be a metric measure space such that $\spt(\m)=X$ and $\mm(X) = 1$. 
Let us consider $K,N\in{\mathbb{R}}$, $N>1$ and let $u:(X,\sfd)\to \mathbb{R}$ be a 1-Lipschitz function. We say that $(X,\sfd,\m)$ satisfies the $\CD^1_u$ condition if there exists a family $\{X_{\alpha}\}_{\alpha\in{Q}}\subset X$ such that:

\begin{enumerate}
\item[(1)]
There exists a disintegration of $\m \llcorner_{\mathcal{T}_{u}}$ on $\{X_{\alpha}\}_{\alpha\in{Q}}$:
\[
\m\llcorner_{\mathcal{T}_{u}}= \int_{Q} \m_{\alpha}\, \mathfrak{q}(d\alpha), \,\,\,\text{where}\, \m_{\alpha}(X_{\alpha})=1,\,\, \text{for}\,\mathfrak{q}\text{-a.e.} \,\alpha\in{Q}.                                
\]

\item[(2)] For $\mathfrak{q}$-a.e. $\alpha\in{Q}$, $X_{\alpha}$ 
is a \emph{transport ray} for $\Gamma_u$. 

\item[(3)] For $\mathfrak{q}$-a.e. $\alpha\in{Q}$, $\m_{\alpha}$ is supported on $X_{\alpha}$.

\item[(4)] For $\mathfrak{q}$-a.e. $\alpha\in{Q}$, the metric measure space $(X_{\alpha},\sfd,\m_{\alpha})$ satisfies $\CD(K,N)$.
\end{enumerate}
\end{definition}

\begin{remark}[The assumption $\mm(X) = 1$]\label{R:normalization}
For an overview (and a self-contained proof) of the                    
Disintegration Theorem we refer to \cite{biacar:cmono,Fre:measuretheory4} (see also \cite{CMi}).
It is worth mentioning here that the assumption $\mm(X) = 1$ is most probably purely technical. 
In the framework of general Radon measure, the Disintegration Theorem does not furnish a unique 
family of conditional measures and one has to consider an additional 
normalization function; for additional details we refer to \cite{CM18a} where 
a localization of synthetic lower Ricci curvature bounds has been obtained also for 
general Radon measure.
\end{remark}

Let us recall that it is well known  that 
the last condition of Definition \ref{D:defCD1} is equivalent to  
asking $\mm_{\alpha} \sim h_{\alpha} \mathcal{L}^{1}\llcorner_{[0,|X_{\alpha}|]}$ 
where $|X_{\alpha}|=\ell(X_\alpha)$ denotes the length of the transport ray $X_{\alpha}$
and $\sim$ means up to isometry of the space, and the density $h_{\alpha}$ 
has to satisfy the power-concavity inequality
$$
\left(h_{\alpha}^{1/(N-1)}\right)'' + \frac{K}{N-1}h_{\alpha}^{1/(N-1)}\leq 0,
$$
in the distributional sense.

Finally, we will say that  the metric measure space $(X,\sfd,\m)$ satisfies $\CD^1_{Lip}(K,N)$ if $(\spt(\m), \sfd,\m)$ verifies $\CD^1_u(K,N)$ for all $1$-Lipschitz functions $u:(\spt(\m),\sfd)\to \mathbb{R}$, and 
satisfies $\CD^1(K,N)$ if $(\spt(\m), \sfd,\m)$ verifies $\CD^1_u(K,N)$   whenever
$u$ is a  signed distance function defined as follows:
given a continuous function $f : (X,\sfd) \to \R$ such that $\set{f = 0} \neq \emptyset$, the function
\begin{equation}\label{E:levelsets}
d_{f} : X \to \R, \qquad d_{f}(x) : = \text{dist}(x, \{ f = 0 \} ) \sgn(f),                                       
\end{equation}
is called the signed distance function (from the zero-level set of $f$).  
Notice that $d_f$ is $1$-Lipschitz on $\set{f \geq 0}$ and $\set{f \leq 0}$. If $(X,\sfd)$ is a length space, then $d_f$ is $1$-Lipschitz on the entire $X$. 
%%%%%%%%%%%%%%%%%%%%%%%%%%%%%%%%%%%%%%%%%%%%%%%%%%%%%%%%%%%%%%%%%%%%%%%%%%%%%%%%%%%%%
%%%%%%%%%%%%%%%%%%%%%%%%%%%%%%%%%%%%%%%%%%%%%%%%%%%%%%%%%%%%%%%%%%%%%%%%%%%%%%%%%%%%%
\subsection{Derivatives} \label{subsec:prelim-derivatives}
In order to carry out a third order analysis of Kantorovich potentials, we will frequently use incremental ratios 
over arbitrary subsets of $\R$. We will use the following 
notation: for a function $g : A \rightarrow \Real$ on a subset $A \subset \Real$, denote its upper and lower derivatives at a point $t_0 \in A$ which is an accumulation point of $A$ by:
\[
\frac{\overline{d}}{dt} g(t_0) = \limsup_{A \ni t \rightarrow t_0} \frac{g(t) - g(t_0)}{t-t_0}  ~,~ \underline{\frac{d}{dt}} g(t_0) = \liminf_{A \ni t \rightarrow t_0} \frac{g(t) - g(t_0)}{t-t_0} .
\]
We will say that $g$ is differentiable at $t_0$ iff $\frac{d}{dt} g(t_0) := \frac{\overline{d}}{dt} g(t_0) = \underline{\frac{d}{dt}} g(t_0) \in\mathbb{R}$. 
This is a slightly more general definition of differentiability than the traditional one which requires that $t_0$  is an interior point of $A$.
\begin{remark} \label{R:diff-restriction}
Note that there are only a countable number of isolated points in $A$, so a.e. point in $A$ is an accumulation point. In addition, it is clear that if $t_0 \in B \subset A$ is an accumulation point of $B$ and $g$ is differentiable at $t_0$, then $g|_B$ is also differentiable at $t_0$ with the same derivative. In particular, if $g$ is a.e. differentiable on $A$ then $g|_B$ is also a.e. differentiable on $B$ and the derivatives coincide.
\end{remark}
\begin{remark}\label{R:differentiabilitydensity}
Denote by $A_1 \subset A$ the subset of density one points of $A$ (which are in particular accumulation points of $A$). By Lebesgue's Density Theorem $\L^1(A \setminus A_1) = 0$, where we denote by $\L^1$ the Lebesgue measure on $\Real$ throughout this work. If $g : A \rightarrow \Real$ is  Lipschitz, consider any  Lipschitz extension $\hat g : \Real \to \Real$ of $g$. Then it is easy to check that for $t_0 \in A_1$, $g$ is differentiable in the above sense at $t_0$ if and only if $\hat g$ is differentiable at $t_0$ in the usual sense, in which case the derivatives coincide. In particular, as $\hat g$ is a.e. differentiable on $\Real$, it follows that $g$ is a.e. differentiable on $A_1$ and hence on $A$, and it holds that $\frac{d}{dt} g = \frac{d}{dt} \hat g$ a.e. on $A$.
\end{remark}

\medskip
If $f : I \rightarrow \Real$ is a convex function on an open interval $I \subset \Real$, it is a  well-known fact that the left and right derivatives $f^{\prime,-}$ and $f^{\prime,+}$ exist at every point in $I$ and that $f$ is locally Lipschitz. 
In particular, $f$ is differentiable at a given point if
and only if the left and right derivatives coincide there. Denoting by $D \subset I$ the differentiability points of $f$ in $I$, it is also well-known that $I \setminus D$ is at most countable. Consequently, any point in $D$ is an accumulation point, and we may consider the differentiability in $D$ of $f' : D \rightarrow \Real$ as defined above. 

We will recall the following classical one-dimensional result
about twice differentiability a.e. of convex functions on $\Real^n$. 
The result extends to locally semi-convex and semi-concave functions as well; recall that a function $f : I \rightarrow \Real$ is called semi-convex (semi-concave) if there exists $C \in \Real$ so that $I \ni x \mapsto f(x) + C x^2$ is convex (concave).
\begin{lemma}[Second Order Differentiability of Convex Function] 
\label{lem:convex-2nd-diff}
Let $f : I \rightarrow \Real$ be a convex function on an open interval $I \subset \Real$, and let $\tau_0 \in I$ and $\Delta \in \Real$. Then the following statements are equivalent:
\begin{enumerate}
\item $f$ is differentiable at $\tau_0$, and if $D \subset I$ denotes the subset of differentiability points of $f$ in $I$, then $f' : D \rightarrow \Real$ is differentiable at $\tau_0$ with:
\[
(f')'(\tau_0) := \lim_{D \ni \tau \rightarrow \tau_0} \frac{f'(\tau) - f'(\tau_0)}{\tau-\tau_0} = \Delta .
\]
\item The right derivative $f^{\prime,+} : I \rightarrow \Real$ is differentiable at $\tau_0$ with $(f^{\prime,+})'(\tau_0) = \Delta$. 
\item The left derivative $f^{\prime,-}: I \rightarrow \Real$ is differentiable at $\tau_0$ with $(f^{\prime,-})'(\tau_0) = \Delta$. 
\item $f$ is differentiable at $\tau_0$ and has the following second order expansion there:
\[
f(\tau_0 + \eps) = f(\tau_0) + f'(\tau_0) \eps + \Delta \frac{\eps^2}{2} + o(\eps^2)  \text{ as $\eps \rightarrow 0$}. 
\]
In this case, $f$ is said to have a second Peano derivative at $\tau_0$. 
\end{enumerate}
\end{lemma}

For a locally semi-convex or semi-concave function $f$, we will say that $f$ is twice differentiable at $\tau_0$ if any (all) of the above equivalent conditions hold for some $\Delta \in \Real$, and write $(\frac{d}{d\tau})^{2}|_{\tau = \tau_0} f(\tau) = \Delta$.

\smallskip
Finally, we will recall the following slightly different version
of the second order differential.
\begin{definition}[Upper and lower second Peano derivatives]\label{D:Peano}
Given an open interval $I \subset \Real$ and a function $f : I \rightarrow \Real$ which is differentiable at $\tau_0 \in I$, we define its upper and lower second Peano derivatives at $\tau_0$, denoted $\overline{\P}_2 f(\tau_0)$ and  $\underline{\P}_2 f(\tau_0)$ respectively, by:
\begin{equation}\label{def-upper and lower second derivatives}
 \overline{\P}_2 f(\tau_0) := \limsup_{\eps\rightarrow 0} \frac{h(\eps)}{\eps^2} \geq \liminf_{\eps\rightarrow 0} \frac{h(\eps)}{\eps^2} =: \underline{\P}_2 f(\tau_0) ,
\end{equation}
where:
\begin{equation}\label{E:def-h}
h(\eps) := 2( f(\tau_0 + \eps) - f(\tau_0) - \eps f'(\tau_0)) .
\end{equation}
We say that $f$ has a second Peano derivative at $\tau_0$ iff $\overline{\P}_2 f(\tau_0) = \underline{\P}_2 f(\tau_0) \in\mathbb{R}$. 
\end{definition}
\begin{lemma} \label{lem:peano-inq}
Given an open interval $I \subset \Real$ and a locally absolutely continuous function $f : I \rightarrow \Real$ which is differentiable at $\tau_0 \in I$, we have:
\[
\underline{\frac{d}{dt}} f'(\tau_0) \leq \underline{\P}_2 f(\tau_0) \leq \overline{\P}_2 f(\tau_0) \leq \frac{\overline{d}}{dt} f'(\tau_0)  .
\]
\end{lemma}
%%%%%%%%%%%%%%%%%%%%%%%%%%%%%%%%%%%%%%%%%%%%%%%%%%%%%%%%%%%%%%%%%%%%%%%%%%%%%%%%%%%%%
%%%%%%%%%%%%%%%%%%%%%%%%%%%%%%%%%%%%%%%%%%%%%%%%%%%%%%%%%%%%%%%%%%%%%%%%%%%%%%%%%%%%%
\subsection{Notation}
Given a subset $D \subset X \times \Real$, we denote its sections by:
\[
D(t) := \set{ x \in X \;;\; (x,t) \in D} ~,~ D(x) := \set{t \in \Real \; ; \; (x,t) \in D} .
\]
Given a subset $G \subset \Geo(X)$, we denote by $\mathring{G} := \set{ \gamma|_{(0,1)} \;;\; \gamma \in G}$ the corresponding open-ended geodesics on $(0,1)$. For a subset of (closed or open) geodesics $\tilde{G}$, we denote:
\begin{equation}\label{E:image}
Im(\tilde{G}) := \set{ (x,t) \in X \times \Real \; ; \; \exists \gamma \in \tilde{G} ~,~ t \in \text{Dom}(\gamma) \;, \; x = \gamma_t } .
\end{equation}
%%%%%%%%%%%%%%%%%%%%%%%%%%%%%%%%%%%%%%%%%%%%%%%%%%%%%%%%%%%%%%%%%%%%%%%%%%%%
%%%%%%%%%%%%%%%%%%%%%%%%%%%%%%%%%%%%%%%%%%%%%%%%%%%%%%%%%%%%%%%%%%%%%%%%%%%%
%%%%%%%%%%%%%%%%%%%%%%%%%%%%%%%%%%%%%%%%%%%%%%%%%%%%%%%%%%%%%%%%%%%%%%%%%%%%
\section{Hopf-Lax transform with exponent $p$}\label{S:HopfLax}
In this  section we review the basic properties of the Hopf-Lax transform                  
in a metric measure space setting with a general exponent $p > 1$.
Some of following properties are well-known for the case $p = 2$,  hence we  omit  the proofs for general $p$ whenever they follow the same line of reasoning as the corresponding proofs for $p = 2$.
The main references for most of the definitions and proofs will be \cite{AGS11a,AGS11b,CMi}; 
further developments related to ours may also be found in 
\cite{AmbrosioFeng14} \cite{GozlanRobertoSamson14} \cite{GangboSwiech15} \cite{Bessi20+}  and their references. \\ 

As motivation for the needed properties of the metric measure space 
Hopf-Lax transform we remind the reader
of the relationship between the Hopf-Lax transform and the Eulerian view of optimal transport.
 We also provide a comparison between the results found in this paper to familiar results from 
 Euclidean spaces.

We illustrate the main relationship for the case $(\mathbb{R}^{n},d)$ with $d$ as the Euclidean distance, and the cost function
    $c(x,y)=\frac{d(x,y)^{p}}{p}$ where $p>1$.
    Recall that in the Eulerian view of optimal transport, the Wasserstein distance can be interpreted as the minimizing energy to the problem
    \begin{equation}\label{ContinuityEquation}
        \begin{cases}
           \rho_{t}+\nabla\cdot\left(\rho{}v\right)=0& 
                \text{in }\mathbb{R}^{n}\times(0,1)\\
            \rho(\cdot,0)=\rho_{0}&
                \text{in }\mathbb{R}^{n}\\
            \rho(\cdot,1)=\rho_{1}&
                \text{in }\mathbb{R}^{n}\\
        \end{cases}
    \end{equation}
    where $\rho,v$ are the distribution of mass and the velocity at position $x$ at time $t$ respectively \cite[Theorem 8.3.1]{AGSbook}. By choosing $v= DH\left(\nabla\varphi\right)$,  where in our case $ H(w)={|w|^{p'}}/{p'}$, and  $\varphi$ is a solution to the Hamilton-Jacobi equation
    \begin{equation}\label{HamiltonJacobiPDE}
        \begin{cases}
        \partial_{t}\varphi+H(\nabla \varphi)=0
        & \text{in }\mathbb{R}^{n}\times(0,\infty)\\
        \varphi(x,0)=\varphi_{0}(x)& \text{for }x\in\mathbb{R}^{n},
        \end{cases}
    \end{equation}
    where $\varphi_{0}$ is a Kantorovich potential for the optimal transport problem and $p'$ is the real number satisfying $\frac{1}{p}+\frac{1}{p'}=1$.
    That is, $p'$ is the H\"{o}lder dual of $p$.
     The method of characteristics   gives a solution to the Hamilton-Jacobi equation for a convex Hamiltonian $H$ 
     \cite{Evans:PDE}.  Furthermore, this solution can be expressed  by the Hopf-Lax formula
    \begin{equation*}
        \varphi(x,t)=
        \inf_{y\in\mathbb{R}^{n}}\left\{\varphi_{0}(y)+tL\left(\frac{x-y}{t}\right)\right\},
    \end{equation*}
    where the Lagrangian $L$ is defined by
    \begin{equation*}
        L(z)=\inf_{w\in\mathbb{R}^{n}}\left\{z\cdot{}w-H(w)\right\}.
    \end{equation*}
    In our case,  the Lagrangian is explicitly computed as $L(v)=\frac{|v|^{p}}{p}$, hence
    \begin{equation}\label{HopfLaxTransform}
        \varphi(x,t)=\inf_{y\in\mathbb{R}^{n}}\left\{\varphi_{0}(y)+\frac{|x-y|^{p}}{pt^{p-1}}\right\}.
    \end{equation}
     Finally, in the context of smooth manifolds, we can compute the spatial gradient as
    \begin{equation*}
        \nabla \varphi(x)=\frac{|x-y|^{p-2}(x-y)}{t^{p-1}},
    \end{equation*}
    where $y$ is chosen to  be a minimizer in the Hopf-Lax infimum \eqref{HopfLaxTransform}.
    Hence,
    \begin{equation}\label{NormGradient}
        \frac{|\nabla\varphi(x,t)|^{p'}}{p'}=\frac{(p-1)|x-y|^{p}}{pt^{p}}.
    \end{equation}
    Note that, due to $\eqref{NormGradient}$, $\eqref{HamiltonJacobiPDE}$ and $\eqref{HopfLaxTransform}$ can be compared to conclusion $3$ of  Theorem
    $\ref{teo:ags}$ and  Corollary \ref{cor:loclipfi} respectively.
    In particular,  the expression  in $\eqref{NormGradient}$ depends on $x$ only through its distance to the minimizing value $y$.
    This should be compared to Definition $\ref{DistanceFunctionDefinition}$.
    With the above in mind, we  now present  the details of the nonsmooth case.

In the following sections, we will only consider the cost function $c=\sfd^p/p$ on $X \times X$. 
\begin{definition*}[$c$-Concavity, Kantorovich Potential]
The $c$-transform of a function $\psi : X \rightarrow \Real \cup \set{\pm\infty}$ is defined as the following (upper semi-continuous) function:
\[
\psi^c(x) = \inf_{y \in X} \frac{\sfd(x,y)^p}{p} - \psi(y) . 
\]
A function $\varphi : X \rightarrow \Real \cup \set{\pm \infty}$ is called $c$-concave if $\varphi = \psi^c$ for some $\psi$ as above. 
It is well known that $\varphi$ is $c$-concave iff $(\varphi^c)^c = \varphi$. 
A $c$-concave function $\varphi : X \rightarrow \Real \cup \set{-\infty}$ which is not identically equal to $-\infty$ is also 
known as a Kantorovich  (or $p$-Kantorovich) potential, and this is how we will refer to such functions in this work. 
In that case, $\varphi^c : X \rightarrow \Real \cup \set{-\infty}$ is also a Kantorovich potential, called the dual or conjugate potential.
\end{definition*}

In these sections, we only assume that $(X,\sfd)$ is a \textbf{proper geodesic metric space}. (Here proper refers to the requirement that closed balls are compact).
%%%%%%%%%%%%%%%%%%%%%%%%%%%%%%%%%%%%%%%%%%%%%%%%%%%%%%%%%%%%%%%%%%%%%%
%%%%%%%%%%%%%%%%%%%%%%%%%%%%%%%%%%%%%%%%%%%%%%%%%%%%%%%%%%%%%%%%%%%%%%
\subsection{General definitions}\label{Ss:HopfLaxDef}
\begin{definition}[Hopf-Lax transform]\label{Hopf-Lax transform}
Let $f:X\to \mathbb{R}\cup\{\pm\infty\}$ be not identically $+\infty$ and $t>0$, $p>1$. The Hopf-Lax transform $Q_t f:X\to \mathbb{R}\cup{\{-\infty}\}$ is defined as \begin{equation}
\label{equ:hopf}
Q_t f(x):= \inf_{y\in{X}} \frac{{\sfd(x,y)}^p}{pt^{p-1}}+f(y).
\end{equation}
\end{definition}

\noindent
 If $Q_tf(\bar{x})\in \R$ for some $\bar{x}\in {X}$ and $t>0$, then $Q_sf(x)\in \R$ for all $x\in X$ and  $0<s \leq t$.  Hence defining
\[
t_{*}(f):=\sup \{t>0:  Q_tf \not\equiv -\infty\},
\]
 where we set $t_{*}(f)=0$ if the supremum is over an empty set, it holds that $Q_tf(x)\in \mathbb{R}$ for every $x\in{X}, t\in{(0,t_*(f))}$.  Moreover, we set $Q_0f:=f$.
The definition of $Q_tf$ can be extended to negative times $t<0$ by  setting
\begin{equation}\label{E:minust}
Q_tf(x)=-Q_{-t}(-f)(x)= \sup_{y\in{X}} - \frac{\sfd(x,y)^p}{p(-t)^{p-1}}+f(y), \quad t<0.
\end{equation}

If $(X,\sfd)$ is a length space (and in particular, if it is geodesic),
the Hopf-Lax transform is in fact a semi-group on $[0,\infty)$:
\[
Q_{s+t} f = Q_s \circ Q_t f \;\;\; \forall t,s \geq  0 .
\]
Being the infimum of continuous functions in $(t,x)$, the map $(0,\infty)\times X \ni (t,x)\mapsto Q_tf(x)$  is 
upper semi-continuous. Moreover, by definition  $[0,\infty)\ni t \mapsto Q_tf(x)$ is monotone non-increasing;  hence, it  is continuous from the left. 

We  define the {\em distance progressed}  as the 
length of the geodesic segment in $X$ along which information propagates from the initial values to $(t,x)$;
this geodesic plays the role of a characteristic curve.
Since we are modeling optimal transport, shocks do not form before unit time has elapsed \cite{Vil:topics}.
\begin{definition}\label{DistanceFunctionDefinition}{(Distance progressed $D^{\pm}_f)$}.
 Given $f:X\to \mathbb{R}\cup\{+\infty\}$ not identically $+\infty$, we define 
\[
D^+_f(x,t):= \sup \limsup_{n\to +\infty} \sfd(x,y_n)\geq \inf \liminf_{n\to +\infty} \sfd(x,y_n)=: D^-_f(x,t)
\]
where the supremum and the infimum are taken on the set of minimizing sequences $\{y_n\}_{n\in{\mathbb{N}}}$ in the definition of Hopf-Lax transform.  Using a diagonal argument, it is possible to show that the  supremum and infimum are attained, though they may differ in the presence of  shocks.
\end{definition}

For $p=2$, the following properties were established in 
\cite[Chapter 3]{AGS11a}. For a proof adopted to a similar framework 
we refer to \cite[Section 3.2]{CMi}.
\begin{theorem}[Hopf-Lax solution to metric space Hamilton-Jacobi equations] 
\label{teo:ags}
For any metric space $(X,\sfd)$ the following properties hold:
\begin{itemize}

\item[1.] Both functions $D^{\pm}_f(x,t)$ are locally finite on $X\times (0,t_*(f))$ and $(x,t) \mapsto Q_tf(x)$ is locally Lipschitz there.

\item[2.] The map  $(x,t) \mapsto D^{+}_f(x,t)$  $\bigl{(}(x,t) \mapsto D^{-}_f(x,t) \bigr{)}$  is upper (lower) semi-continuous on $X\times(0,t_*(f))$.

\item[3.] For every $x\in{X}$, 
$$
\partial^{\pm}_t Q_tf(x)=- \frac{(p-1)D^{\pm}_f(x,t)^{p}}{p t^{p}}, \qquad 
\forall \ t\in (0,t_*(f)),
$$ 
where $\partial^{-}_t$ and $\partial^{+}_t$   denote the left and right partial derivatives respectively. In particular, the map $(0,t_*(f)) \ni t \mapsto Q_t f(x)$ is locally Lipschitz and locally semi-concave. Moreover, it is differentiable at $t\in{(0,t_*(f))}$  if and only if $D^+_f(x,t)=D^-_f(x,t)$.
\end{itemize}
\end{theorem}
\begin{proof}
For the readers' convenience we will only address \emph{3.}
The claim can be found \cite[Remark 3.1.7]{AGSbook} and the proof for $p = 2$ is given 
in \cite[Theorem 3.1.4]{AGSbook}.

Fix $t_0 <t_1\in{(0, t_*(f))}$.  By  Lemma \ref{lem:atteined}, 
there exists  
$x_{t_1}\in \text{argmin} \left\{\frac{\sfd(x,y)^p}{p t_1^{p-1}}+f(y) \right\}$, 
for which $\sfd(x, x_{t_1})= D^{+}_f(x, t_1)$. In particular, it holds:
\begin{align*}
Q_{t_0}f(x)- Q_{t_1}f(x) &\leq  \frac{\sfd(x,x_{t_1})^p}{p t_0^{p-1}}-  \frac{\sfd(x,x_{t_1})^p}{p t_1^{p-1}}\\
&= \frac{ D^{+}_f(x, t_1)^p}{p}\cdot \biggl(\frac{t_1^{p-1}-t_0^{p-1}}{t_0^{p-1}\cdot t_1^{p-1}}\biggr).
\end{align*}
Applying again Lemma \ref{lem:atteined}, there exists 
$x_{t_0}\in{\text{Argmin} \biggl{\{}\frac{\sfd(x,y)^p}{p t_0^{p-1}}+f(y) \biggr{\}}}$ 
for which  $\sfd(x, x_{t_0})= D^{+}_f(x, t_0)$. 
Arguing as before, we get:
\begin{align*}
Q_{t_0}f(x)- Q_{t_1}f(x) &\geq  \frac{\sfd(x,x_{t_0})^p}{p t_0^{p-1}}-  \frac{\sfd(x,x_{t_0})^p}{p t_1^{p-1}}\\
&= \frac{ D^{+}_f(x, t_0)^p}{p}\cdot \biggl(\frac{t_1^{p-1}-t_0^{p-1}}{t_0^{p-1}\cdot t_1^{p-1}}\biggr).
\end{align*}
Dividing by $t_1-t_0>0$, we obtain:
\[
\frac{ D^{+}_f(x, t_0)^p}{p}\cdot \biggl(\frac{t_1^{p-1}-t_0^{p-1}}{(t_1-t_0) \cdot t_0^{p-1}\cdot t_1^{p-1}}\biggr) \leq  \frac{Q_{t_0}f(x)- Q_{t_1}f(x)}{t_1-t_0} \leq \frac{ D^{+}_f(x,t_1)^p}{p}\cdot \biggl(\frac{t_1^{p-1}-t_0^{p-1}}{(t_1-t_0)\cdot t_0^{p-1}\cdot t_1^{p-1}}\biggr)
\]
Sending $t_1$ to $t_0$ from the right we obtain:
$$
\partial^{+}_t Q_tf(x)=- \frac{(p-1)D^{+}_f(x,t)^{p}}{p t^{p}}, \qquad 
\forall \ t\in (0,t_*(f)),
$$ 
The same holds with the minus sign.
\end{proof}

The next property will be  used throughout the paper; we include a proof 
for the readers' convenience.
\begin{lemma}[Hopf-Lax attainment]
\label{lem:atteined}
Let $X$ be a proper metric  space,  $f:X \to \mathbb{R}$ a lower semi-continuous function, and $t_{*}(f)>0$. For  fixed $x\in{X}$ and $t\in{(0,t_{*}(f))}$, there exist $y^{\pm}_t\in{X}$ so that 
\begin{equation}
Q_tf(x)= \frac{\sfd(x,y^{\pm}_t)^p}{pt^{p-1}}+f(y^{\pm}_t).                              
\end{equation}
Moreover, the following holds: $\sfd(x,y^{\pm}_t)=D^{\pm}_f(x,t)$.
\end{lemma}
\begin{proof}
Let $\{y_t^{\pm,n}\}$ be a minimizing sequence such that 
\[
Q_t f(x)=\lim_{n\to\infty}\frac{\sfd(x,y_{t}^{\pm,n})^p}{pt^{p-1}}+f(y_{t}^{\pm,n})\,\, \text{and}\,\,  D^{\pm}_f(x,t)=\lim_{n\to\infty}\sfd(x,y_t^{\pm,n})
\]
By local finiteness of $D^{\pm}_f$, it follows that $D^{\pm}_f(x,t)<R$ for some $R<\infty$. The properness of the space $X$ guarantees that the closed geodesic ball $B_R(x)$ is compact, hence $\{y_t^{\pm,n}\}$  admits a  subsequence converging to $\{y_t^{\pm}\}$. Using the lower semi-continuity of $f$, we get:
\[
Q_tf(x)=\inf_{y\in{X}} \frac{\sfd(x,y)^p}{pt^{p-1}}+f(y)=\min_{y\in{B_R(x)}} \frac{\sfd(x,y)^p}{pt^{p-1}}+f(y)=\frac{\sfd(x,y^{\pm}_t)^p}{pt^{p-1}}+f(y^{\pm}_t).
\]
Hence, the claim holds true.
\end{proof}
\begin{lemma}[Time monotonicity of distance progressed]
Let $X$ be a proper metric space  and let $f:X\to \mathbb{R}\cup\{\color{black}+\infty\}\color{black}$ be a lower semi-continuous function. Then, for every $x\in{X}$, both functions $(0,t^{*}(f)) \ni t \mapsto D^{\pm}_f(x,t)$ are monotone non-decreasing and coincide except where they have jump discontinuities.
\end{lemma}
\begin{proof}
 Since trivially $D^{-}_f\leq D^{+}_f$, it is sufficient to prove that 
\[
D^{+}_f(x,s)\leq D^{-}_f(x,t),  \qquad 0<s<t<t^{*}(f)
\]
in order to conclude. By Lemma \ref{lem:atteined}, there exist $y^{+}_s,y^-_{t}$ such that   
\begin{align*}
& \frac{\sfd(x,y^{+}_s)^p}{ps^{p-1}}+f(y^{+}_s)=Q_s(f)(x) \leq \frac{\sfd(x,y^{-}_t)^p}{ps^{p-1}}+f(y^{-}_t),\\
& \frac{\sfd(x,y^{-}_t)^p}{pt^{p-1}}+f(y^{-}_t)= Q_t(f)(x) \leq  \frac{\sfd(x,y^{+}_s)^p}{pt^{p-1}}+f(y^{+}_s).
\end{align*}
Summing the two, we get
\[
\sfd(x,y^+_s)^p\cdot \biggl{(} \frac{1}{s^{p-1}}-\frac{1}{t^{p-1}} \biggr{)} \leq \sfd(x,y^-_t)^p\cdot \biggl{(} \frac{1}{s^{p-1}}-\frac{1}{t^{p-1}} \biggr{)}.
\]
Since the Lemma \ref{lem:atteined}  also guarantees that  $\sfd(x,y^{-}_t)=D^{-}_f(x,t)$ and $\sfd(x,y^{+}_s)=D^{+}_f(x,s)$, the claim follows.
\end{proof}
%%%%%%%%%%%%%%%%%%%%%%%%%%%%%%%%%%%%%%%%%%%%%%%%%%%%%%%%%%%%%%%%%%%%%%
%%%%%%%%%%%%%%%%%%%%%%%%%%%%%%%%%%%%%%%%%%%%%%%%%%%%%%%%%%%%%%%%%%%%%%
\subsection{Intermediate-time Kantorovich potentials}
\begin{definition}{(Interpolating Intermediate-Time Kantorovich Potentials).}
Given a Kantorovich potential $\varphi:X\to \mathbb{R}$, the interpolating $p$-Kantorovich potential at time $t\in{[0,1]}$,  denoted by $\varphi_t:X\to \mathbb{R}$, is defined for all $t\in{[0,1]}$ by:
\begin{equation}
\varphi_t(x):=Q_{-t}(\varphi)=-Q_t(-\varphi).
\end{equation}
Note that $\varphi_0 = \varphi$, $\varphi_1 = -\varphi^c$, and:
$$
-\varphi_t(x) = \inf_{y \in X} \frac{\sfd^p(x,y)}{pt^{p-1}} - \varphi(y) \;\;\;\; \forall t \in (0,1] .
$$
\end{definition}
Applying the previous general properties of the Hopf-Lax semi-group 
we directly obtain that 
\begin{itemize}
\item[1.] $(x,t)\mapsto \varphi_{t}(x)$ is lower semi-continuous on $X\times(0,1]$ and continuous on $X\times(0,1)$.

\item[2.] For every  $x\in{X}$, $[0,1]\ni t\mapsto \varphi_t(x)$ is monotone non-decreasing and continuous on $(0,1]$.
\end{itemize}
We also recall the following terminology:
given a Kantorovich potential $\varphi:X\to \mathbb{R}$, $\gamma \in{\Geo(X)}$ is called a \emph{ $(\varphi, p)$-Kantorovich geodesic}  if 
\begin{equation}\label{Kantorovich geodesics}
\varphi(\gamma_0)+\varphi^c(\gamma_1)= \frac{\sfd(\gamma_0,\gamma_1)^p}{p}=\frac{{\ell(\gamma)}^p}{p}.
\end{equation}
The set of all Kantorovich geodesics will be denoted with $G_{\varphi}$; the upper semi-continuity of $\varphi$ and $\varphi^c$ implies that $G_{\varphi}$ is a  closed subset of $\Geo(X)$.
Using the modified triangular inequality                                                                                       
\begin{equation}
\label{equ:holder}
\sfd(x, y)^p \leq \frac{{\sfd(x,z)}^p}{t^{p-1}}+ \frac{{\sfd(z,y)}^p}{(1-t)^{p-1}},
\end{equation}
valid for every choice of $x,y,z\in{X}$, we may conclude that
along $(\varphi, p)$-Kantorovich geodesics, $\varphi_{t}$ is affine in time, and it verifies the following nice expression:
\begin{equation}\label{affinity along characteristics}
\varphi_t(\gamma_t)= (1-t) \frac{\sfd(\gamma_0,\gamma_1)^p}{p}- \varphi^c(\gamma_1).
\end{equation}
This result easily implies the following corollary.
\begin{corollary}
\label{cor:diff}
Let $\gamma$ be a $(\varphi, p)$-Kantorovich geodesic. Then, for any $s,r\in{(0,1)}$, we have:
\begin{equation}
\varphi_s(\gamma_s)-\varphi_r(\gamma_r)= (r-s) \frac{\sfd(\gamma_0,\gamma_1)^p}{p}.
\end{equation}
\end{corollary}
\begin{lemma}
\label{lem:tcond}
Let $x,y,z$ be points in $X$ and let $t\in{(0,1)}$.  If
\begin{equation}
\label{equ:tcondition}
\frac{{\sfd(x,y)}^p}{pt^{p-1}}-\varphi(y)=\varphi^c(z)-\frac{{\sfd(x,z)}^p}{p(1-t)^{p-1}},
\end{equation}
then $x$ is a $t$-intermediate point between $y$ and $z$ with
\begin{equation}
\sfd(y,z)=\frac{\sfd(x,y)}{t}=\frac{\sfd(x,z)}{1-t}.
\end{equation}
Moreover there exists a  $(\varphi, p)$-Kantorovich geodesic $\gamma:[0,1]\to X$ with $\gamma_0=y$, $\gamma_t=x, \gamma_1=z$.
\end{lemma}
\begin{proof}
By definition of the $c$-transform, from the assumption \eqref{equ:tcondition} it follows that 
\[
\frac{{\sfd(x,y)}^p}{pt^{p-1}}+\frac{{\sfd(x,z)}^p}{p(1-t)^{p-1}}=\varphi(y)+\varphi^c(z)\leq \frac{\sfd(y,z)^p}{p}.
\]
Hence, the equality holds since  the reverse inequality is trivially satisfied by \eqref{equ:holder}. In particular, requiring the equality  in the H\"older inequality implies that 
\begin{equation}
\frac{\sfd(x,z)^p}{(1-t)^p}=\sfd(y,z)^p=\frac{\sfd(x,y)^p}{t^p}.
\end{equation}
So the concatenation $\gamma:[0,1]\to X$ of any constant speed geodesic $\gamma^1:[0,t]\to X$ between $x$ and $y$ with any constant speed geodesic $\gamma^2:[t,1]\to X$ between $y$ and $z$ so that $\gamma_0=y$, $\gamma_t=x$, $\gamma_1=z$ must be a constant speed geodesic itself by the triangle inequality.
In particular also 
\[
\varphi(y)+\varphi^c(z)\leq \frac{\sfd(y,z)^p}{p}.
\]
must hold as equality, impling $\gamma$ to be a $(\varphi, p)$-Kantorovich geodesic.
\end{proof}

In what follows, forward and backward evolution via the Hopf-Lax semi-group
will permit us to obtain regularity properties and key estimates on the
intermediate-time Kantorovich potential. 
However, it is immediate to show by inspecting the definitions that we always have 
$$
Q_{-s} \circ Q_s f \leq f  \text{ on $X$} \;\;\; \forall s > 0 ; 
$$
note that for $f = -\varphi$ where $\varphi$ is a Kantorovich potential, we do have equality for $s=1$, and in fact for all $s \in [0,1]$; 
for $f = Q_t(-\varphi)$, $t \in (0,1)$ and $s=1-t$, we can only assert an \emph{inequality} 
\begin{equation} \label{eq:no-duality}
(\varphi^c)_{1-t} = Q_{-(1-t)} \circ Q_1 (-\varphi)   \leq  Q_t(-\varphi) = -\varphi_t \text{ on $X$,}
\end{equation}
and equality  need not hold at every point of $X$.
\begin{definition*}[Time-Reversed Interpolating Potential]
Given a Kantorovich potential $\varphi: X \rightarrow \Real$, 
define the time-reversed interpolating Kantorovich potential at time $t \in [0,1]$,
$\varphic_t : X \rightarrow \Real$, as:
\[
\varphic_t := -(\varphi^c)_{1-t} =  Q_{1-t}(-\varphi^c)  = - Q_{-(1-t)} \circ Q_{1-t}(-\varphi_t) .
\]
Note that $\varphic_0 = \varphi$, $\varphic_1 = -\varphi^c$, and:
\[
\varphic_t(x) = \inf_{y \in X} \frac{\sfd^p(x,y)}{p(1-t)^{p-1}} - \varphi^c(y) \;\;\;\; \forall t \in [0,1) .
\]
\end{definition*}

\noindent
Note that, since any  Kantorovich potential $\varphi$ is upper semi-continuous,   Lemma \ref{lem:atteined} applies to $f=-\varphi$.

\begin{lemma}[Relating forward to reverse evolution of potentials]\label{lem:phibar}
The following properties hold true:
\begin{enumerate}

\item $\varphi_0=\bar{\varphi}_0=\varphi$ and $\varphi_1=\bar{\varphi}_1=-\varphi^c$;

\item For all $t\in{[0,1]}$, $\varphi_t\leq \bar{\varphi}_t$;

\item For any $t\in{(0,1)}$, $\varphi_t(x)=\bar{\varphi}_t(x)$ if and only if $x\in{\ee_{t}(G_{\varphi})}$.
\end{enumerate}
\end{lemma}
\begin{proof}
Point \emph{1.} is  a trivial consequence of the definitions. Also \emph{2.} is straightforward, since 
\[
\bar{\varphi}_t:= Q_{1-t}(-\varphi^c)=-Q_{-(1-t)}\circ Q_{1-t}(-\varphi_t)\geq \varphi_t.
\]
To demonstrate \emph{3.}, let us consider a point $x=\gamma_t$ with $\gamma\in{G_{\varphi}}$ and use the following notation $\ell(\gamma) = \sfd(\gamma_{0},\gamma_{1})$ for length.
Applying Corollary \ref{cor:diff} to $\gamma$ with $s=0$ and $r=t$ we get
\[
\varphi(\gamma_0)-\varphi_t(\gamma_t)=t \frac{{\ell(\gamma)}^p}{p},
\]
while applying the same result to $\gamma^c\in{G_{\varphi^c}}$, the time reversed curve, with $s=1$, and $r=(1-t)$ we obtain
\begin{align*}
-\varphi(\gamma_0)-\varphi^c_{1-t}(\gamma_t)&= (\varphi^c)_1(\gamma^c_1)-(\varphi^c)_{1-t}(\gamma^c_{1-t})\\
&=-t\frac{\ell(\gamma^c)^p}{p}=-t\frac{\ell(\gamma)^p}{p}.
\end{align*}
Summing the two identities,  it follows that 
$\varphi_t(\gamma_t)=-(\varphi^c)_{1-t}(\gamma_t)=\bar{\varphi_t}(\gamma_t)$.

For the other implication, let us assume that for some $x\in{X}$, $t\in{(0,1)}$
$\varphi_t(x)=-(\varphi^c)_{1-t}(x)$.
Applying  
Lemma \ref{lem:atteined}  to the lower semi-continuous functions $-\varphi$ and $-\varphi^c$, it turns out that there exist $y_t$,$z_t \in{X}$ such that 
\begin{align*}
-&\varphi_t(x)= Q_t(-\varphi)(x)=\frac{{\sfd(x,y_t)}^p}{pt^{p-1}}-\varphi(y_t),\\
&\varphi_t(x)=Q_{1-t}(-\varphi^c)(x)=\frac{{\sfd(x,z_t)}^p}{pt^{p-1}}-\varphi^c(z_t).
\end{align*}
Summing the two equations, we get that
\begin{equation*}
\frac{{\sfd(x,y_t)}^p}{pt^{p-1}}-\varphi(y_t)=\varphi^c(z_t)-\frac{{\sfd(x,z_t)}^p}{p(1-t)^{p-1}},
\end{equation*}
so we are in position to apply Lemma \ref{lem:tcond}, obtaining the claim.
\end{proof}

Motivated by Lemma \ref{lem:phibar} we will also consider the following set
\begin{equation}\label{E:domain}
 D(\mathring{G}_\varphi) = \set{(x,t) \in X \times (0,1) \; ; \; \varphi_t(x) = \varphic_t(x) },
\end{equation}
 which is a closed subset of $X \times(0,1)$.
%%%%%%%%%%%%%%%%%%%%%%%%%%%%%%%%%%%%%%%%%%%%%%%%%%%%%%%%%%%%%%%
%%%%%%%%%%%%%%%%%%%%%%%%%%%%%%%%%%%%%%%%%%%%%%%%%%%%%%%%%%%%%%%
\subsection{First and Second Order inequalities}\label{Ss:FirstAndSecondRegularity}
 Let us now introduce the speed 
along which each characteristic is traversed;  since the particles move freely,
this coincides with the total length of the characteristic,  which is why the same functions are called length
functions $\ell_t$ in \cite{CMi}. To emphasize the dynamic point of view, we shall also refer to 
$(p-1)\ell_t^p/p = (\ell_t^{p-1})^{p'}/p'$ as the {\em energy},  though it is really the energy per unit mass transported. 
\begin{definition}[Speed functions $\ell^{\pm}_t, \bar{\ell}^{\pm}_t$] Given a Kantorovich potential $\varphi:X\to \mathbb{R}$, define the speed functions $\ell^{\pm}_t, \bar{\ell}^{\pm}_t$ as follows:
\begin{equation*}
\ell^{\pm}_t(x):=\frac{D^{\pm}_{-\varphi}(x,t)}{t},\quad  \bar{\ell}^{\pm}_t(x):=\frac{D^{\pm}_{-\varphi^c}(x,1-t)}{1-t},\quad (x,t)\in{X\times (0,1)}.
\end{equation*}
\end{definition}
Let us mention that we will shortly see that if $x = \gamma_t$ with $\gamma \in G_\varphi$ and $t \in (0,1)$, then:
\[
\ell^{+}_t(x) = \ell^{-}_t(x) = \ellc^{+}_t(x) = \ellc^{-}_t(x) = \len(\gamma) .
\]
In particular, all  $(\varphi,p)$-Kantorovich geodesics having $x$ as their $t$-mid-point have necessarily the same length.
For $\tilde{\ell} \in \{\ell,\bar{\ell}\}$, we define the  set:
\begin{equation}\label{well-defined speed domain in spacetime}
D_{\tilde{\ell}}:=\{ (x,t)\in{X\times(0,1)}:\tilde{\ell}^+_t(x)=\tilde{\ell}^-_t(x) \}.
\end{equation}
On $D_{\tilde{\ell}}$ we set $\tilde{\ell}_t(x):=\tilde{\ell}^{-}_t(x)=\tilde{\ell}^{+}_t(x)$.
Recalling that $\varphi_t = -Q_t(-\varphi)$ and $\varphic_t = Q_{1-t}(-\varphi^c)$, we 
can apply Theorem \ref{teo:ags} to deduce the following:

\begin{corollary}[Time semi-continuity of speeds] \label{cor:loclipfi}
Let $\varphi:X \to \mathbb{R}$ denote a Kantorovich potential. Then:
\begin{enumerate}
\item  Choosing $\tilde{\ell}  \in \{\ell,\bar{\ell}\}$ and $\tilde{\varphi} \in \{\varphi,\bar{\varphi}\}$  correspondingly, 
$\tilde{\ell}^{\pm}_t(x)$  are locally finite on $X\times (0,1)$, and $(x,t)\mapsto \tilde{\varphi}_t(x)$ is locally Lipschitz there.

\item For  $\tilde{\ell} \in \{\ell,\bar{\ell}\}$ the map $(x,t)\mapsto \tilde{\ell}^{+}_t(x)$ $((x,t)\mapsto \tilde{\ell}^{-}_t(x))$ is upper (lower) semi-continuous on $X\times(0,1)$. In particular, $D_{\tilde{\ell}}\subset X\times (0,1)$ is Borel and $(x,t)\mapsto \tilde{\ell}_t(x)$ is continuous on ${D_{\tilde{\ell}}}$.

\item For every $x\in{X}$ we have:
\[
\partial^{\pm}_t \varphi_t(x)= \frac{{(p-1)\ell_t^{\pm}(x)}^{p}}{p}, \quad \partial^{\pm}_t \bar{\varphi}_t(x)= \frac{(p-1){\bar{\ell}_t^{\pm}(x)}^{p}}{p}\quad \forall t\in{(0,1)}.
\]
In particular, for $\tilde{\ell}  \in \{ \ell,\ellc\}$ and  the corresponding $\tilde{\varphi} \in \{\varphi,\varphic\}$, the map $(0,1) \ni t \mapsto \tilde{\varphi}_t(x)$ is locally Lipschitz, and it is differentiable at $t \in (0,1)$ iff $t \in D_{\tilde{\ell}}(x)$, the set on which both maps $(0,1) \ni t \mapsto \tilde{\ell}^{\pm}_t(x)$ coincide. $D_{\tilde{\ell}}(x)$ is precisely the set of continuity points of both maps, and thus coincides with $(0,1)$ with at most countably exceptions.\end{enumerate}
\end{corollary}

 All four maps $(0,1) \ni t \mapsto t \ell^{\pm}_t(x)$ and 
$(0,1) \ni t \mapsto (t-1) \ellc^{\pm}_t(x)$ are monotone non-decreasing;
in particular, both $D_{\ell}(x) \ni t \mapsto \ell_t^{p}(x)$  
and $D_{\bar \ell}(x) \ni t \mapsto \bar \ell_t^{p}(x)$ are differentiable a.e..
From monotonicity it is straightforward to deduce
$$
\underline{\partial}_t \ell_t(x) \geq -\frac{1}{t} \ell_t(x) \;\;\; \forall t \in D_{\ell}(x),
$$
 as well as a similar estimate for $\bar \ell_{t}$.  In particular, the following estimates 
holds (see \cite[Corollary 3.10]{CMi}).
\begin{corollary}[Energies are locally Lipschitz in time] 
\label{cor:derivatives}
The following estimates hold for every $x\in{X}$:
\begin{equation}
\label{equ:under}
\underline{\partial}_t\frac{\ell_t^p(x)}{p} \geq -\frac{1}{t} \ell_t^p(x),\,\,\quad \forall t\in{D_{\ell}(x)}.
\end{equation}
\begin{equation}
\label{equ:over}
\overline{\partial}_t \frac{\bar \ell_t^p(x)}{p} \leq \frac{1}{1-t} \bar \ell_t^p(x),\,\,\quad \forall t\in{D_{\bar  \ell}(x)}.
\end{equation}
\end{corollary}

The first and the last points of the next Theorem can be compared with                        
\cite[Theorem 2.13]{CMi} in the case $p=2$.
\begin{theorem}[Time-derivatives of energies bound second time-derivatives of potentials]
\label{teo:loclip}
Let $\varphi:X\to \mathbb{R}$ be a Kantorovich potential. Then the following holds true:
\begin{enumerate}
\item For all $x\in{\ee_{t}(G_\varphi)}$ with $t\in{(0,1)}$, we have:
\[
\ell^+_t(x)=\ell^-_{t}(x)=\bar{\ell}^+_t(x)=\bar{\ell}^-_{t}(x)=\ell(\gamma).
\]
                                          
\item For all $x\in{X}$, $\mathring{G}_{\varphi}(x)\ni t \mapsto \ell_{t}(x)=\bar{\ell}_t(x)$ is locally Lipschitz and, provided $\ell(\gamma)>0$,  the following estimate holds true 
\[
\frac{1-s}{1-t}\leq \frac{\ell_t(x)}{\ell_s(x)} \leq \frac{s}{t},  \,\,\,\, 0<t\leq s<1.
\]

\item For all $(x,t)\in{D(\mathring{G}_{\varphi})}\subset D_{\ell}\cap D_{\bar{\ell}}$ we have that the following estimate holds true for the upper and lower second derivatives
$z  \in \{ \underline{\mathcal{P}}_2\bar{\varphi}_t(x), \overline{\mathcal{P}}_2 \varphi_t(x)\}$  in time
 of \eqref{def-upper and lower second derivatives}:
\[
-\frac{p-1}{t}  \ell_t^p(x)
\leq \underline{\partial}_t \frac{(p-1)\ell_t^p(x)}{p}
 \leq \underline{\mathcal{P}}_2\varphi_t(x) \leq z \leq  \overline{\mathcal{P}}_2 \bar{\varphi}_t(x) 
 \leq \overline{\partial}_t \frac{(p-1)\bar{\ell}_t^p(x)}{p} 
 \leq \frac{p-1}{1-t} \ell_t^p(x).
\]
\end{enumerate}
\end{theorem}
\begin{proof}
Let $(x,t)\in{D(\mathring{G_{\varphi}})}$ (recall \eqref{E:domain}).  An application of  Lemma \ref{lem:atteined} implies that  there exist  $y^{\pm},z^{\pm}\in{X}$ such that                                                                            
\begin{align*}
&-\varphi_t(x)=\frac{{\sfd(x,y_{\pm})}^p}{pt^{p-1}}-\varphi(y_{\pm}),\\
&-\bar{\varphi}_t(x)=-\frac{{\sfd(x,z_{\pm})}^p}{pt^{p-1}}+\varphi^c(z_{\pm}).
\end{align*}
Since  $\varphi_t(x)=\bar{\varphi}_t(x)$ by Lemma \ref{lem:phibar}, we can equate the two expressions, obtaining that the assumption \eqref{equ:tcondition} in the Lemma \ref{lem:tcond} is satisfied. Hence $x$ is the $t$-midpoint of a geodesic connecting $y_{\pm}$ and $z_{\pm}$ for all four possibilities.
The same lemma guarantees that 
\[
\frac{\sfd(x,y^{\pm})}{t}=\frac{\sfd(x,z^{\pm})}{1-t}
\]
and thus $\ell^{\pm}_t(x)=\bar{\ell}^{\pm}_t(x)$. Recall now that if $x=\gamma_t$ for some $\gamma\in{G_{\varphi}}$ then Corollary \ref{cor:diff}  implies that 
\[
Q_t(-\varphi)(x)=-\varphi_t(x)=\frac{{\sfd(x,\gamma_0)}^p}{pt^{p-1}}-\varphi(\gamma_0)
\]
and thus the sequence $\{y_n\}$ with $y_n\equiv \gamma_0$ is in the class of admissible sequences for the infimum and supremum in the definition of $D^{\pm}_{-\varphi}(x,t)$. Hence
\[
t \ell^-_t(x)=D^{-}_{-\varphi}(x,t)\leq \sfd(x,\gamma_0)=t\ell(\gamma)\leq D^{+}_{-\varphi}(x,t)=t\ell^+_t(x),
\]
and \emph{1.} follows.

In order to prove \emph{2.}, we use that, by the discussion following Corollary \ref{teo:loclip}, for $x\in{}X$ the maps $t\mapsto{}t\ell_{t}^{\pm}(x)$ are
monotone non-decreasing and the maps $t\mapsto(1-t)\bar{\ell}_{t}^{\pm}$ are monotone non-increasing combined with the previous conclusion of this
theorem to obtain that for $x\in{}X$ and $t,s\in\mathring{G}_{\varphi}(x)$ with $t<s$:
\[
t\ell_{t}(x)\le{}s\ell_{s}(x),
\hspace{10pt}
(1-s)\ell_{s}(x)\le(1-t)\ell_{t}(x).
\]
For $\gamma\in{}G_{\varphi}$ with $\ell(\gamma)>0$ we conclude the desired statement by rearranging.
This allows us to conclude that $\ell_{\cdot}(x)$ is locally Lipschitz.

To obtain \emph{3.}, as in \eqref{E:def-h} let us define $\tilde{h}=h,\bar{h}$ as  
\[
\tilde{h}(\varepsilon):= 2(\tilde{\varphi}_{t_0+\varepsilon}(x)-\tilde{\varphi}_{t_0}(x)-\varepsilon\partial_t\tilde{\varphi}_{t_0}(x)).
\]
Recall that, by Lemma \ref{lem:phibar},  for all $t\in{[0,1]}$ it holds $\varphi_t\leq \bar{\varphi}_t$ with the equality satisfied  in the case $x\in{\ee_{t}(G_{\varphi})}$. 
Moreover, since $\mathring{G}_{\varphi}(x) \subset D_{\ell}(x)\cap D_{\bar{\ell}}(x)$, the maps $t \mapsto \tilde{\varphi}_{t}(x)$ are differentiable at $t_0\in{\mathring{G}_{\varphi}(x)}$ and 
${(p-1)\ell_{t_0}^{p}(x)}/{p}=\partial_t\vert_{t=t_{0}}\varphi_{t}(x)=\partial_t\vert_{t=t_{0}}\bar{\varphi}_{t}(x)={(p-1)\bar{\ell}_{t_0}^{p}(x)}/{p}$. 
These facts imply that $h\leq \tilde{h}$ on $(-t_0,1-t_0)$. Dividing by $\varepsilon^2$ and taking subsequential limits, we obtain
\[
\underline{\mathcal{P}}_2 \varphi_t(x)\leq \underline{\mathcal{P}}_2 \bar{\varphi}_t(x),\quad  \overline{\mathcal{P}}_2 \varphi_t(x)\leq \overline{\mathcal{P}}_2 \bar{\varphi}_t(x).
\]
Combining these inequalities with those of Lemma \ref{lem:peano-inq}, \eqref{equ:under} and  \eqref{equ:over} we get the claim.
\end{proof}

We conclude with the following result; for its proof we refer to 
\cite[Corollary 3.13]{CMi}.
\begin{corollary}
For all $x\in{X}$, for a.e. $t\in{\mathring{G_{\varphi}}(x)}$, $\partial_t \ell_t^p(x)$  and $\partial_t  \overline{\ell}_t^p (x)$ exist, coincide, and satisfy:
\begin{align}\label{E:3.14}
-\frac{\ell_t^p(x) }{t} 
&\leq \partial_t \frac{\ell_t^p(x)}{p}
= \partial_t \frac{\ell_t^p(x)\mid_{\mathring{G_{\varphi}}(x)}}{p} \nonumber\\ 
&=\partial_t \frac{ \overline{\ell}_t^p(x)\mid_{\mathring{G_{\varphi}}(x)}}{p}
= \partial_t \frac{ \overline{\ell}_t^p(x)}{p} 
\leq   \frac{ {\overline \ell_t^p}(x)}{1-t} .
\end{align}
\end{corollary}
\begin{remark}
Recall that we already proved that $\partial^{\pm}_{\tau}|_{\tau=s} \varphi_{\tau}(x)=(p-1)\ell^{\pm}_{s}(x)^p/{p}$ and $\ell^{\pm}_s(\gamma_s)=\ell$ for all $s\in{(0,1)}$.
\end{remark}
%%%%%%%%%%%%%%%%%%%%%%%%%%%%%%%%%%%%%%%%%%%%%%%%%%%%%%%%%%%%%%%%%%%%%%%%%%%
%%%%%%%%%%%%%%%%%%%%%%%%%%%%%%%%%%%%%%%%%%%%%%%%%%%%%%%%%%%%%%%%%%%%%%%%%%%
\subsection{Third order inequality}\label{Ss:ThirdOrder}
Just as  the solution to a Hamilton-Jacobi equation with Hamiltonian $H(w)=|w|^{p'}/p'$ behaves affinely in time on its characteristics,  \eqref{affinity along characteristics}  similarly shows that the $t$ interpolant $\varphi_t$ 
of a Kantorovich potential becomes an affine function of time $t$ along a $\varphi$-Kantorovich geodesic $\gamma_t$.
The goal of this and the next sections  is to show that $\partial^2_t\varphi_t$ is non-decreasing 
along such curves and provide a positive lower bound \eqref{equ:geomean}--\eqref{equ:terzordine} for 
 the slope of $z(t) := [\partial^2_t\varphi_t](\gamma_t)$ --- at least
under certain regularity hypotheses which can be subsequently verified  
 for a large enough family of $\varphi$-Kantorovich geodesics  that serve our purposes.  
For $p=p'=2$, such estimates were discovered in \cite{CMi}, but
their proof does not generalize to our case.  However,  
Cavalletti and Milman \cite{CMi} also provided a heuristic argument  
in the smooth setting 
which can be adapted to $p\ne 2$ as follows.

Start from the Hamilton-Jacobi equation 
\begin{equation*}
\partial_{t}\varphi_{t}=H(\nabla\varphi_{t})
\end{equation*}
satisfied by the time $t$ interpolant $\varphi_t$ of a Kantorovich potential $\varphi$ on a Riemannian manifold.
Differentiating in $t$ gives
\begin{equation}\label{HJE_t}
\partial_{t}^{2}\varphi_{t}= DH|_{\nabla\varphi_{t}}(\nabla\partial_{t}\varphi_{t}).
\end{equation}
Setting
$
z(t)=\left[\partial_{t}^{2}\varphi_{t}\right](\gamma_{t})
$
where $\gamma_{t}$ is the time $t$ evaluation of a $\varphi$-Kantorovich geodesic,
we observe using $\gamma'(t)=-DH(\nabla\varphi_{t})$ that
\begin{equation*}
z'(t)=\partial_{t}^{3}\varphi_{t}(\gamma_{t})-\left<\nabla\partial_{t}^{2}\varphi_{t}(\gamma_{t}),DH(\nabla\varphi_{t}(\gamma(t)))\right>.
\end{equation*}
On the other hand
\begin{equation*}
\partial_{t}^{3}\varphi_{t}=D^2H|_{\nabla \varphi_t}( \nabla\partial_{t}\varphi_{t}, \nabla\partial_{t}\varphi_{t} )
+DH|_{\nabla\varphi_{t}}(\nabla\partial_{t}^{2}\varphi_{t} ).
\end{equation*}
Inserting this into the previous equation yields
\begin{align*}
z'(t)&=D^2H|_{\nabla \varphi_t}( \nabla\partial_{t}\varphi_{t}, \nabla\partial_{t}\varphi_{t} )
\\ &=  |\nabla\varphi_{t}(\gamma(t))|^{p'-2}  |\nabla\partial_{t}\varphi_{t}(\gamma(t))|^2 
+ (p'-2)  |\nabla\varphi_{t}(\gamma(t))|^{p'-4}  \left<\nabla\varphi_{t}(\gamma(t)),  \nabla\partial_{t}\varphi_{t}(\gamma(t))
\right>^2.
\end{align*}
Convexity of $H(w)=|w|^{p'}/p'$ shows  that $z(t)$ is non-decreasing (hence confirming differentiability a.e.)
and allows its derivative to be estimated from below in terms of $|\nabla\varphi_{t}(\gamma(t))|$ and 
$|\nabla\partial_{t}\varphi_{t}(\gamma(t))|$ --- both of which exist a.e. since $\varphi_t$ is 
locally semiconvex in the halfspace $t>0$.  
The Cauchy-Schwarz inequality gives
\begin{align*}
z'(t) 
&\ge (p'-1) {\left|\nabla\varphi_{t}(\gamma(t))\right|^{p'-4}}
{\left<\nabla\varphi_{t}(\gamma(t)),\nabla\partial_{t}\varphi_{t}(\gamma(t))\right>^{2}}
\\&=\frac{1}{p-1}\frac{z(t)^{2}}{\ell_{t}^{p}},
\end{align*}
where 
 $\ell_t = |DH(\nabla \varphi_t)|$ and $(p-1)(p'-1)=1$, and
\eqref{HJE_t} has been used to identify
$
z(t)
=|{\nabla\varphi_{t}(\gamma_{t})}|^{p'-2} 
\left<{\nabla\varphi_{t}(\gamma_{t})}, \nabla\partial_{t}\varphi_{t}(\gamma_{t})\right>
$.
At least heuristically, this establishes \eqref{equ:terzordine}.

In order to obtain rigorous estimates on third order variations of Kantorovich potentials,  we introduce   the quantities  $\tilde{r} \in \{r,\bar{r}\}$   which measure the time partial of energies along 
a fixed $\varphi$-Kantorovich geodesic (which plays the role of a characteristic in the nonsmooth setting);
for every $s\in{(0,1)}$ set
\begin{align*} 
&\tilde{r}_{+}^{\gamma}(s)
=\tilde{r}_{+}(s)
:=\overline{\partial}_{\tau}|_{\tau=s} \frac{(p-1)}{p} \tilde{\ell}_\tau^p(\gamma_s)
=(p-1) \tilde{\ell}^{p-1}\overline{\partial}_{\tau}|_{\tau=s}\tilde{\ell}_{\tau}(\gamma_s),\\
&\tilde{r}_{-}^{\gamma}(s)
=\tilde{r}_{-}(s):=\underline{\partial}_{\tau}|_{\tau=s} \frac{(p-1)}{p}\tilde{\ell}^p_{\tau}(\gamma_s)
=(p-1)\tilde{\ell}^{p-1}\underline{\partial}_{\tau}|_{\tau=s}\tilde{\ell}_{\tau}(\gamma_s).
\end{align*}
By definition,  $\tilde{r}_{-}(s)\leq\tilde{r}_{+}(s)$; moreover, equality holds $\tilde{r}_{-}(s)=\tilde{r}_{+}(s)=\tilde{r}$ if and only if the map $\tau \mapsto (p-1)\tilde{\ell}_\tau^p(\gamma_s)/p$ is differentiable at $\tau=s$ with derivative $\tilde{r}$.

\noindent
 We also define upper and lower  second order Peano derivatives  in time (Definition \ref{D:Peano})
$\tilde{q}_{\pm} \in \{q_{\pm},\bar{q}_{\pm}\}$ of the (forward and backward) interpolated Kantorovich potentials respectively,
evaluated along the same characteristic,  as follows:
\begin{align*}
& \tilde{q}_{+}(s):= \overline{\mathcal{P}}_2 \tilde{\varphi}_s(x)|_{x=\gamma_s}= \limsup_{\varepsilon\to 0} \frac{\tilde{h}(s,\varepsilon)}{\varepsilon^2},\\
& \tilde{q}_{-}(s):= \underline{\mathcal{P}}_2 \tilde{\varphi}_s(x)|_{x=\gamma_s}=\liminf_{\varepsilon\to 0} \frac{\tilde{h}(s,\varepsilon)}{\varepsilon^2},
\end{align*}
where $\tilde h(s,\varepsilon)$ is defined analogously to \eqref{E:def-h}.
By definition, $\tilde{q}_{-}(s)=\tilde{q}_{+}(s)=\tilde{q}$  hold if and only if the map $\tau \mapsto \tilde{\varphi}_{\tau}(\gamma_s)$ has second-order Peano derivative at $\tau=s$ given by $\tilde{q}$.
We summarize the relation between $\tilde q_{\pm}$ and $\tilde r_{\pm}$ 
implied by Lemma \ref{lem:convex-2nd-diff} and Lemma \ref{lem:peano-inq} 
in the following :
\begin{corollary}[First differentiability of energy is equivalent to second differentiability of potential]
\label{cor:confrontoder}
The following statements are equivalent for a given $s\in{(0,1)}$:
\begin{enumerate}
\item $\tilde{r}_{-}(s)=\tilde{r}_{+}(s)=\tilde{r}\in{\mathbb{R}}$, i.e. the map  $D_{\tilde{\ell}}(\gamma_s) \ni \tau \mapsto (p-1)\tilde{\ell}_\tau^p(\gamma_s)/p$ is differentiable at $\tau=s$ with derivative $\tilde{r}$.

\item  $\tilde{q}_{-}(s)=\tilde{q}_{+}(s)=\tilde{q}\in{\mathbb{R}}$, i.e. the map  $(0,1)\ni \tau \mapsto \tilde{\varphi}_{\tau}(\gamma_s)$ has second order Peano derivative at $\tau=s$ given by $\tilde{q}$.
\end{enumerate}

If one of the two conditions above is satisfied, the map $(0,1)\ni \tau \mapsto \tilde{\varphi}_{\tau}(\gamma_s)$ is twice differentiable at $\tau=s$, and we have :
\[
\partial^{2}_{\tau}|_{\tau=s} \tilde{\varphi}_{\tau}(\gamma_s)
=\partial_{\tau}|_{\tau=s} \frac{(p-1)\tilde{\ell}_\tau^p(\gamma_s)}{p}
= (p-1)\tilde{\ell}^{p-1} \cdot \partial_{\tau}|_{\tau=s} \ell_{\tau}(\gamma_s)=\tilde{r}=\tilde{q}.
\]
\end{corollary}

We are now  in a position to obtain lower bounds on 
the incremental ratio of $\tilde q$.                                                                                        
This provides the required third-order information concerning $\varphi_t$ even when the upper and lower derivatives in question do not agree.
For the geometric interpretation of the following  discretized differential
inequalities, we refer to  the discussion of the case $p = 2$ in \cite[Section 5.1]{CMi}.
\begin{theorem}[Third-order difference quotient bounds on potential along its characteristics]
\label{teo:mainestimate}
For all $0<s<t<1$ and both possibilities for $\pm$, we have
\begin{equation}
\label{equ:qq}
\frac{q_+(t)-q_-(s)}{t-s}\geq \frac{s}{t} \frac{r_{\pm}(s)^2}{(p-1) \ell^p},
\end{equation}
\begin{equation}
\label{equ:qbar}
\frac{\bar{q}_{+}(t)-\bar{q}_{-}(s)}{t-s}\geq  \frac{1-t}{1-s} \frac{\bar{r}_{\pm}(t)^2}{(p-1)\ell^p}.
\end{equation}
\end{theorem}
The proof of the analogous estimate for $p = 2$ (\cite[Theorem 5.2]{CMi}) does not work in the general case $p > 1$. 
\begin{proof}
By definition of the Hopf-Lax transform and by Lemma \ref{lem:atteined}, we have that
\[
-\varphi_{s+\varepsilon}(\gamma_s)=Q_{s+\varepsilon}(-\varphi)(\gamma_s)= \frac{\sfd(y^{\pm}_{\varepsilon},\gamma_s)^p}{p(s+\varepsilon)^{p-1}}-\varphi(y^{\pm}_\varepsilon),
\]
with $\sfd(y^{\pm}_{\varepsilon},\gamma_s)=D^{\pm}_{-\varphi}(\gamma_s,s+\varepsilon)=(s+\varepsilon) \ell^{\pm}_{s+\varepsilon}(\gamma_s)=:D^{\pm}_{s+\varepsilon}$. Moreover, the following inequality trivially holds:
\[
-\varphi_{t+\varepsilon}(\gamma_t)\leq \frac{\sfd(y^{\pm}_{\varepsilon},\gamma_t)^p}{p(t+\varepsilon)^{p-1}}-\varphi(y^{\pm}_\varepsilon).
\]
Subtracting the  two expressions above, we obtain:
\[
\varphi_{t+\varepsilon}(\gamma_t)-\varphi_{s+\varepsilon}(\gamma_s)\geq -\frac{\sfd(y^{\pm}_{\varepsilon},\gamma_t)^p}{p(t+\varepsilon)^{p-1}}+\frac{\sfd(y^{\pm}_{\varepsilon},\gamma_s)^p}{p(s+\varepsilon)^{p-1}},
\]
hence  recalling \eqref{E:def-h}
\begin{align}
\label{equ:convessita}
\frac{1}{2} (h(t,\varepsilon)-h(s,\varepsilon)) &\geq -\varphi_t(\gamma_t)+\varphi_s(\gamma_s)-\frac{\sfd(y^{\pm}_{\varepsilon},\gamma_t)^p}{p(t+\varepsilon)^{p-1}}+\frac{\sfd(y^{\pm}_{\varepsilon},\gamma_s)^p}{p(s+\varepsilon)^{p-1}}, \notag\\ \notag
&= {(t-s)}\frac{\ell^p}{p} -\frac{\sfd(y^{\pm}_{\varepsilon},\gamma_t)^p}{p(t+\varepsilon)^{p-1}}+\frac{\sfd(y^{\pm}_{\varepsilon},\gamma_s)^p}{p(s+\varepsilon)^{p-1}},\\ 
&=  {(t-s)}\frac{\ell^p}{p} -\frac{\sfd(y^{\pm}_{\varepsilon},\gamma_t)^p}{p(t+\varepsilon)^{p-1}}+\frac{(s+\varepsilon)(\ell^{\pm}_{s+\varepsilon}(\gamma_s))^p}{p}
 .
\end{align}
We need now an estimate from below of the second term. In order to do that, let us observe that 
\[
\sfd(y^{\pm}_{\varepsilon},\gamma_t)\leq \sfd(y^{\pm}_{\varepsilon},\gamma_s)+ \sfd(\gamma_s,\gamma_t)=D^{\pm}_{s+\varepsilon}+D_t-D_s,
\]
where we put $D_r=r\ell=\sfd(\gamma_r,\gamma_0)$, for $r=s,t$. In particular, 
\begin{align*}
D^{\pm}_{s+\varepsilon}+D_t-D_s&=(s+\varepsilon)(\ell^{\pm}_{s+\varepsilon}(\gamma_s)-\ell_s(\gamma_s))+(t+\varepsilon)\ell\\
& =(t+\varepsilon) \biggl[ \frac{s+\varepsilon}{t+\varepsilon}\ell^{\pm}_{s+\varepsilon}(\gamma_s)+ \biggl( 1- \frac{s+\varepsilon}{t+\varepsilon} \biggr)\ell \biggr].
\end{align*}
Thus, substituting this expression in \eqref{equ:convessita}, we get
\begin{equation}
\label{equ:conv}
\frac{h(t,\varepsilon)-h(s,\varepsilon)}{2\varepsilon^2}  \geq \frac{t+\varepsilon}{p\varepsilon^2} \biggl[\frac{t-s}{t+\varepsilon}\ell^p+\frac{s+\varepsilon}{t+\varepsilon }(\ell^{\pm}_{s+\varepsilon}(\gamma_s))^p  -  \biggl( \frac{s+\varepsilon}{t+\varepsilon}\ell^{\pm}_{s+\varepsilon}(\gamma_s)+  \frac{t-s}{t+\varepsilon} \ell \biggr)^p \biggl].
\end{equation}
In other words, denoting with $f(x):=x^p$  and defining  for every $\lambda\in{[0,1]}$ the functions 
\begin{align*}
s_{x,y}(\lambda)= \lambda f(x)+ (1-\lambda)f(y),\qquad g_{x,y}(\lambda)=f(\lambda x+(1-\lambda)y),
\end{align*}
we want to estimate from below the quantity $s_{x,y}(\lambda)-g_{x,y}(\lambda)$ for the following choices of $\lambda,x,y$:
\begin{equation}
\label{equ:choice}
\lambda=\frac{s+\varepsilon}{t+\varepsilon}, \qquad x=\ell^{\pm}_{s+\varepsilon}(\gamma_s),\qquad y=\ell=\ell_s(\gamma_s).
\end{equation}
Appling the following inequality $s_{x,y}(\lambda)-g_{x,y}(\lambda) \geq   \min_{[y,x]}{f''}\cdot \frac{\lambda(1-\lambda)}{2} (x-y)^2$, 
for all $\lambda\in [0,1]$,
with the choices of $x,y,\lambda$ given by \eqref{equ:choice}, we get
\begin{equation}
\label{equ:aaa}
\frac{h(t,\varepsilon)-h(s,\varepsilon)}{2\varepsilon^2}  \geq \frac{t+\varepsilon}{p\varepsilon^2} \biggl[ \min_{z\in[\ell_s(\gamma_s),\ell^{\pm}_{s+\varepsilon}(\gamma_s)]}z^{p-2}\cdot\frac{ p(p-1)}{2} \cdot \frac{t-s}{t+\varepsilon}\cdot\frac{s+\varepsilon}{t+\varepsilon } \cdot  (\ell^{\pm}_{s+\varepsilon}(\gamma_s)-\ell_s(\gamma_s))^2 \biggl].
\end{equation}
Taking appropriate subsequential limits as $\varepsilon \to 0$, we obtain
\[
\frac{q^+(t)-q^-(s)}{2(t-s)}\geq \frac{s}{t} \frac{(p-1)}{2} \ell^{p-2} (\partial_{\tau}|_{\tau=s} \ell^{\pm}_{\tau}(\gamma_s))^2.
\]
In particular, it turns out that
\[
\frac{q^+(t)-q^-(s)}{t-s}\geq \frac{s}{t} \frac{r_{\pm}(s)^2}{(p-1) \ell^{p}}.
\]
Next, we will deduce inequality \eqref{equ:qbar} from \eqref{equ:qq} by simply using the duality between $\varphi$ and $\varphi^c$. Indeed, since by definition it holds that  $\bar{\varphi}_t=-\varphi_{1-t}^c$, we deduce that : 
\[
\bar{h}^{\varphi}_{\gamma}(r,\varepsilon)=-h^{\varphi^c}_{\gamma^c}(1-r,-\varepsilon).
\]
Moreover, it holds
\[
\frac{(p-1)(\ell^{\varphi^c,\pm}_{1-r-\varepsilon}(\gamma^c_{1-r}))^p}{p}=-\partial^{\mp}_r \varphi^c_{1-r-\varepsilon}(\gamma^c_{1-r})=\partial_r^{\mp} \varphi_{r+\varepsilon}(\gamma_r)=\frac{(p-1)(\bar{\ell}^{\varphi,\pm}_{r+\varepsilon}(\gamma_r))^p}{p};
\]
 hence, choosing as $\varphi$, $\gamma$, $\varepsilon$, $s$, $t$  respectively  $\varphi^c$, $\gamma_c$, $-\varepsilon$,$1-t$, $1-s$ we get the second claim. 
 \end{proof}
%%%%%%%%%%%%%%%%%%%%%%%%%%%%%%%%%%%%%%%%%%%%%%%%%%%%%%%%%%%%%%%%%%%%%%%%
%%%%%%%%%%%%%%%%%%%%%%%%%%%%%%%%%%%%%%%%%%%%%%%%%%%%%%%%%%%%%%%%%%%%%%%%
\subsection{Consequences}
We start by noticing an immediate consequence of Theorem \ref{teo:mainestimate}:
\begin{corollary}
For both $\tilde{q}=q,\bar{q}$, the functions  $t\mapsto \tilde{q}_{\pm}(t)$ are monotone non-decreasing on $(0,1)$.
\end{corollary}
We now put together previous  regularity results on time behaviour 
of Kantorovich potential together with Theorem \ref{teo:mainestimate} 
in order to have a clear statement on the third order variation of Kantorovich 
potentials.

\begin{theorem}[A priori third-order bounds for potential along its characteristics]
\label{teo:zz}
Assume that for \textbf{a.e.} $t\in{(0,1)}$:
\begin{equation}
\label{equ:differentiability}
(0,1) \ni{\tau} \mapsto \tilde{\varphi}_{\tau}(\gamma_t) 
\quad \text{is twice differentiable at }\tau=t  
\quad \text{for both}\quad  \tilde{\varphi}= \varphi, \bar{\varphi}, 
\end{equation}
in any of the equivalent senses of Corollary \ref{cor:confrontoder} and that moreover:
\[
\partial^2_{\tau}|_{\tau=t} \varphi_{\tau}(\gamma_t)=\partial^2_{\tau}|_{\tau=t} \bar{\varphi}_{\tau}(\gamma_t)\quad \text{for a.e.}\,\, t\in{(0,1)}.
\]
If there exists a continuous function $z$ for which
\[
\partial^2_{\tau}|_{\tau=t} \varphi_{\tau}(\gamma_t)=\partial^2_{\tau}|_{\tau=t} \bar{\varphi}_{\tau}(\gamma_t)=z(t)\quad \text{for a.e.}\,\, t\in{(0,1)},
\]
then \eqref{equ:differentiability} holds for \textbf{all} $t\in{(0,1)}$ and for all $t\in (0,1)$
\begin{equation}
\label{equ:pointw}
\partial^2_{\tau}|_{\tau=t} \varphi_{\tau}(\gamma_t)=\partial^2_{\tau}|_{\tau=t} \bar{\varphi}_{\tau}(\gamma_t)= \partial_{\tau}|_{\tau=t} \frac{(p-1)\ell^p_{\tau}(\gamma_t)}{p}= \partial_{\tau}|_{\tau=t} \frac{(p-1)\bar{\ell}^p_{\tau}(\gamma_t)}{p}= z(t).
\end{equation}
Finally, the following third order information on $\varphi_t(x)$ at $x=\gamma_t$ holds true:
\begin{equation}
\label{equ:geomean}
\frac{z(t)-z(s)}{t-s}\geq \sqrt{\frac{s}{t}\frac{1-t}{1-s}} \frac{|z(s)||z(t)|}{(p-1) \ell^p},\quad \forall\,\,0<s<t<1.
\end{equation}
In particular, for any point $t\in{(0,1)}$ where $z(t)$ is differentiable we have
\begin{equation}
\label{equ:terzordine}
z'(t)\geq \frac{z(t)^2}{(p-1)\ell^p}.
\end{equation}
\end{theorem}
\begin{proof}
By  Corollary \ref{cor:confrontoder}, it follows 
that  $\tilde{q}_{-}(t)=\tilde{q}_{+}(t)=z(t)$ for a.e. $t\in{(0,1)}$. 
More precisely, the same holds true for every $t\in{(0,1)}$ 
by the monotonicity of $\tilde{q}_{\pm}$ and the continuity of $z$; 
thus, \eqref{equ:pointw} is satisfied. 
Moreover, Corollary \ref{cor:confrontoder} 
also implies that $\tilde{r}_{-}(t)=\tilde{r}_{+}(t)=z(t)$ for both $\tilde{r}=r,\bar{r}$ and for all $t\in{(0,1)}$. Taking the 
geometric mean of \eqref{equ:qq} and \eqref{equ:qbar}, we get \eqref{equ:geomean}. Finally, passing to the limit as  $s\to t$ in \eqref{equ:geomean}, we obtain \eqref{equ:terzordine}.
\end{proof}

 The assumptions of Theorem \ref{teo:zz} 
will hold true for a.e. $t \in (0,1)$ 
only for a certain family of Kantorovich geodesics.  Nonetheless, this family shall be sufficient for our purposes.

Finally, inequality \eqref{equ:terzordine} will be crucial to deduce 
concavity of certain one-dimensional factors. We include here 
a result that will be used later. For its proof we refer to 
\cite[Lemma 5.7]{CMi}.

\begin{lemma}[Concavity restatement]
\label{lem:concavita}
 Assume that for some locally absolutely continuous function $z$ on $(0,1)$ we have: 
\[
\partial_{\tau}|_{\tau=t} \frac{(p-1)\ell^p_{\tau}(\gamma_t)}{p}= z(t)\quad \text{for a.e.} \,\, t\in{(0,1)}.  
\]
Then for any fixed $r_0\in{(0,1)}$, the function:
\[
L(r)=\exp \biggl{(}-\frac{1}{\ell^p (p-1)} \int_{r_0}^r \partial_{\tau}|_{{\tau}=t}  \frac{(p-1)\ell^p_{\tau}(\gamma_t)}{p}\,dt \biggr{)}= \exp \biggr{(} -\frac{1}{\ell^p(p-1)}  \int_{r_0}^r z(t)\,dt \biggr  {)}
\]
is concave on $(0,1)$.
\end{lemma}
%%%%%%%%%%%%%%%%%%%%%%%%%%%%%%%%%%%%%%%%%%%%%%%%%%%%%%%%%%%%%%%%%%%%%%%%%%%%
%%%%%%%%%%%%%%%%%%%%%%%%%%%%%%%%%%%%%%%%%%%%%%%%%%%%%%%%%%%%%%%%%%%%%%%%%%%%
\subsection{Time propagation of Intermediate Kantorovich potentials}\label{Ss:IntermediateKantorovichPotential}
Finally we recall the definition of time-propagated 
intermediate Kantorovich potentials as introduced in \cite{CMi}.
\begin{definition}Given a Kantorovich potential $\varphi:X\to \mathbb{R}$ and $s,t\in{(0,1)}$, define the $t$-propagated $s$-Kantorovich potential $\Phi_s^t$ on  the domain $D_\ell(t)$ where forward speed is well-defined
 and its time-reversed version $\Phic_s^t$ on  the domain $D_{\ellc}(t)$ from \eqref{well-defined speed domain in spacetime}, by:   
\[
\Phi_s^t := \varphi_t + (t-s) \frac{\ell_t^p}{p}  
\text{ on $D_{\ell}(t)$},
\qquad
\Phic_s^t := \varphic_t + (t-s) \frac{\ellc_t^p}{p} 
\text{ on $D_{\ellc}(t)$}.
\]
\end{definition}

\noindent
Using Theorem \ref{teo:loclip}, it follows that for all $s,t\in{(0,1)}$:
\begin{equation}
\Phi_s^t=\bar{\Phi}_s^t= \varphi_s\circ e_s\circ (\ee_{t} |_{G_\varphi}^{-1}), \qquad \text{on}\, \,\,\ee_{t}(G_{\varphi}).
\end{equation}
Indeed, for any $\gamma\in{G_{\f}}$ it holds
$$
\Phi_s^t (\gamma_t) = \varphi_t(\gamma_t) + (t-s)\frac{{\ell_t(\gamma_t)}^p}{p}
= \varphi_t(\gamma_t) + (t-s)\frac{{\ell(\gamma)}^p}{p}
=\varphi_s(\gamma_s).
$$
 Consequently, on $\ee_t(G_\varphi)$, $\Phi_s^t=\Phic_s^t$ is identified as the push-forward of $\varphi_s$ via $\ee_t \circ \ee_s^{-1}$, i.e. its propagation along $G_\varphi$ from time $s$ to time $t$. 
\begin{proposition}[Linear expansion of energy in time generates propagation of potential]
\label{prop:Phi}
For any $s\in{(0,1)}$, the following properties hold:
\begin{enumerate}
\item The maps $(x,t)\mapsto \Phi^t_s(x)$ and $(x,t)\mapsto \bar{\Phi}^t_s(x)$ are continuous  on $D_{\ell}$ and on $D_{\bar{\ell}}$ respectively;

\item For each $x\in{X}$, denoting  $\tilde{\Phi} \in \{\Phi,\bar{\Phi}\}$ and the corresponding
$\tilde{\ell}\in \{\ell, \bar{\ell}\}$, the map $D_{\tilde{\ell}}(x)\ni t\mapsto \tilde{\Phi}^t_s(x)$ is differentiable at $t$ 
 if and only if $D_{\tilde{\ell}}(x)\ni t \mapsto \tilde{\ell}^p_t(x)$ is differentiable at $t$ or if $t=s\in{D_{\tilde{\ell}}(x)}$. In particular,  $t \mapsto  \tilde{\Phi}^t_s(x)$  is a.e.\ differentiable. At any point of differentiability:
\[
\partial_t \tilde{\Phi}^t_s(x)= {\tilde{\ell}}^p_t(x)+ (t-s) \frac{\partial_{t}\tilde{\ell}^p_t(x)}{p}
\]
In particular, if $s\in{D_{\tilde{\ell}}(x)}$ then $\partial_t|_{t=s} \tilde{\Phi}^t_s(x)$ exists and is given by $\tilde{\ell}^p_t(x)$.

\item For each $x\in{X}$, the map $G_{\varphi}\ni t\mapsto \Phi^t_s(x)=\bar{\Phi}^t_s(x)$ is locally Lipschitz;

\item For all $t\in{(0,1)}$:
\begin{equation}
\begin{cases}
{\displaystyle \underline{\partial}_t \Phi_s^t(x)\geq \frac{s}{t}\ell^p_t(x)}, 
		& t\geq s \crcr \\
{\displaystyle \overline{\partial}_t \Phi_s^t(x)\leq \frac{s}{t}\ell^p_t(x) }, 
		& t\leq s
\end{cases}
\quad \forall x\in{D_{\ell}(t)};
\qquad
\begin{cases}
{\displaystyle \overline{\partial}_t \bar{\Phi}_s^t(x)\leq \frac{1-s}{1-t}\bar{\ell}^p_t(x)}, 
		& t\geq s \crcr \\
{\displaystyle \underline{\partial}_t \bar{\Phi}_s^t(x)\geq \frac{1-s}{1-t}\bar{\ell}^p_t(x) }, 
		& t\leq s
\end{cases}
\quad \forall x\in{D_{\bar{\ell}}(t)}.
\end{equation}
\end{enumerate}
\end{proposition}
\begin{proof}
By lower semi-continuity and  Corollary \ref{cor:diff},  1) and 2) follow trivially.  By Corollary \ref{cor:loclipfi} and Theorem \ref{teo:loclip}, $3)$ holds true.
\noindent
To see 4), observe that for every $x\in{D_{\tilde{\ell}}}(t)$,
\begin{align*}
&\underline{\partial}_t \tilde{\Phi}_s^t(x)=\tilde{\ell}_t^p(x)+(t-s)\underline{\partial}_t \frac{\tilde{\ell}^p_t(x)}{p}, \quad t\geq s\\
&\underline{\partial}_t \tilde{\Phi}_s^t(x)=\tilde{\ell}_t^p(x)+(t-s)\overline{\partial}_t \frac{\tilde{\ell}^p_t(x)}{p}, \quad t\leq s
\end{align*}
with analogous identities holding for $\overline{\partial}_t \tilde{\Phi}_s^t(x)$. Using  estimates \eqref{equ:under} and \eqref{equ:over} of Corollary \ref{cor:derivatives}, the claim follows.
\end{proof}
%%%%%%%%%%%%%%%%%%%%%%%%%%%%%%%%%%%%%%%%%%%%%%%%%%%%%%%%%%%%%%%%%%%%%%%%%%%
%%%%%%%%%%%%%%%%%%%%%%%%%%%%%%%%%%%%%%%%%%%%%%%%%%%%%%%%%%%%%%%%%%%%%%%%%%%
%%%%%%%%%%%%%%%%%%%%%%%%%%%%%%%%%%%%%%%%%%%%%%%%%%%%%%%%%%%%%%%%%%%%%%%%%%%
%%%%%%%%%%%%%%%%%%%%%%%%%%%%%%%%%%%%%%%%%%%%%%%%%%%%%%%%%%%%%%%%%%%%%%%%%%%
%%%%%%%%%%%%%%%%%%%%%%%%%%%%%%%%%%%%%%%%%%%%%%%%%%%%%%%%%%%%%%%%%%%%%%%%%%%
\section{Curvature-Dimension conditions: from $p>1$ to $p =1$}
\label{S:firstimplication}
We will now focus on the main  goal of this paper: to show that for 
essentially non-branching spaces,  the synthetic $(p=2)$ curvature-dimension condition can be equivalently formulated in terms 
of entropic convexity conditions along $p$-Wasserstein geodesics for any  other $p > 1$.
Our approach is to show that for essentially non-branching spaces, the  $\CD_p(K,N)$ for  $p >1$ is equivalent to  $\CD^1(K,N)$,
which is an appropriate concavity statement about the factor measures which arise whenever 
$\mm$ is disintegrated 
along the needles of the signed distance to the zero level-set of an arbitrary continuous function.

The first implication that we will address is the following one: 
if $(X,\sfd,\mm)$ is a $p$-essentially non-branching metric measure space 
verifying $\CD_{p}(K,N)$ 
then it satisfies $\CD^{1}(K,N)$ (actually the stronger $\CD^{1}_{Lip}(K,N)$). 

Before we begin the proof,  we recall the concepts used in $L^{1}$ optimal transport  theory.
For simplicity, we will illustrate the case $X=\mathbb{R}^{n}$ paired with the
Euclidean metric,  and a restriction of Lebesgue as the ambient measure.
  
\smallskip
To a $1$-Lipschitz function, $u:\mathbb{R}^{n}\to\mathbb{R}$, 
we associate a transport ordering,
$\Gamma_{u}$, defined as in $\eqref{TransportSet1}$.
Membership of $(x,y)$ in this set should be understood as ``$y$ travels to $x$ along a transport
 ray determined by $u$".
In particular, it is helpful to consider $u(x)=\left| x \right|$ 
in which case the transport rays
are polar rays emanating from $0$, and $(x,y)\in\Gamma_{u}$ means that $x$ and $y$ lie on the
 same polar ray  with $x$ being larger in norm than $y$.
    
 It is desirable to associate the points that travel along a geodesic with the geodesic itself.
    Towards this goal,  consider a symmetric relation  $R_u$ composed of $\Gamma_{u}$ 
    together with its inverse relation,  and denote the projection of $R_u$ onto its first component  by $\mathcal{T}_{u}$ 
(see \eqref{E:R1}).
    We refer to $R_{u}$ as the transport relation and $\mathcal{T}_{u}$ as the transport set.
     Even though $R_{u}$ is a symmetric relation over $\mathcal{T}_{u}$, it is not transitive. This obstruction to transitivity is called branching,  where two distinct points, $z$ and $w$, travel to or from a point $x$ but no transport ray  of  $u$ transports $z$ to $w$ or vice versa.  More specifically we have the forward and backward branching points $A_{+}$, $A_{-}$ as defined in $\eqref{APlus}$ and $\eqref{AMinus}$.
    To overcome this difficulty we simply remove the offending points and consider the resulting
    equivalence relation.
    That is, we consider $\mathcal{T}_{u}^{b}:=\mathcal{T}_{u}\setminus{}(A_{+}\cup{}A_{-})$
    and $R_{u}^{b}=R_{u}\cap\left(\mathcal{T}_{u}^{b}\times\mathcal{T}_{u}^{b}\right)$.   As seen in Theorem $\ref{teo:branch}$, this procedure removes only a negligible  set of points.  
    
\smallskip
We then use this equivalence relation to break $\mm$, restricted to the branched transport set, into measures
    supported on each of the transport rays determined by $u$.
    We do this by applying the Disintegration  Theorem.
    Using our example of $u(x)=|x|$ over
    $(\mathbb{R}^{n},|\,\cdot\,|,\frac{1}{\omega_{n}}\mathcal{L}^{n}\big\vert_{B_{1}(0)})$ we
    arrive at
\begin{equation}\label{PolarRayDisintegration}
\mathcal{L}^{n} \llcorner_{B_{1}(0)}(\mathrm{d}x)
= \int_{S^{n-1}} |x|^{n-1}\mathcal{H}^{1} \llcorner_{[0, \alpha]}
{\mathcal H}^{n-1}( \mathrm{d}\alpha)
\end{equation}
where ${\mathcal H}^{k}$ is $k$-dimensional Hausdorff measure 
and  $\frac{\omega_{n}}{n}=\mathcal{L}^{n}\left(B_{1}(0)\right)$.
Hence, in this case, disintegration gives polar integration.
We observe that $\eqref{PolarRayDisintegration}$ can be compared 
to $\eqref{E:disint}$ where $Q=S^{n-1}$ and polar rays form the set of
non-branched transport geodesics.
    
Finally, we remind the reader about $\CD(K,N)$.
The $\CD(K,N)$ condition represents, 
in a very generalized sense, a Ricci curvature bound from
below by $K\in\mathbb{R}$ and a dimension bound from above by $N\in(1,\infty)$.
In particular, over an interval in $\mathbb{R}$ and a measure
$\mm=h\mathcal{L}^{1}\bigm\vert_{[0,L]}$, where $h>0$ on $(0,L)$, the $\CD(K,N)$ 
condition reduces to
\begin{equation*}
\left(h^{\frac{1}{N-1}}\right)''+\frac{K}{N-1}h^{\frac{1}{N-1}}\le0.
\end{equation*}
 This condition is equivalent to ($K,N$) convexity  of $-\log(h)$ \cite{EKS}.
That is, $-\log(h)$ satisfies
\begin{equation*}
\left(-\log h\right)''\ge\frac{1}{N-1}\left(\left(-\log h\right)'\right)^{2}+K.
\end{equation*}
Hence, in one-dimensional space, the $\CD(K,N)$ condition amounts to a concavity condition on the density of the reference measure with respect to Lebesgue.
In particular, using the example of $\mm=nr^{n-1}\mathcal{L}^{1} \llcorner_{(0,1)}$ over a polar
ray of length one, the $\CD(K,N)$ conditions becomes
\begin{equation*}
Kr^{2}\le (n-1)\biggl(1- \frac{n-1}{N-1} \biggr)\hspace{10pt}\text{for }0\le{}r\le1.
\end{equation*}
This is satisfied for all $K \le 0$ and $N \ge n$ which is consistent with the curvature 
and dimensionality of $\mathbb{R}^{n}$.
%%%%%%%%%%%%%%%%%%%%%%%%%%%%%%%%%%%%%%%%%%%%%%%%%%%%%%%%%%%%%%%%%%%%%%%%%%%
%%%%%%%%%%%%%%%%%%%%%%%%%%%%%%%%%%%%%%%%%%%%%%%%%%%%%%%%%%%%%%%%%%%%%%%%%%%
\subsection{$L^1$ optimal transport}\label{TransportSets} 
We  recall a standard fact about $1$-Lipschitz functions
and their associated transport set.

To any $1$-Lipschitz function $u : X \to \R$ there is a naturally associated $\sfd$-cyclically monotone set $\Gamma_{u}$ defined in \eqref{TransportSet1}  
that we call the transport ordering; we write $x \ge_u y$ if and only if $(x,y) \in \Gamma_u$ and we recall that $\ge_u$ is a partial-ordering.
The \emph{transport relation} $R_u$ and the 
\emph{transport set} $\mathcal{T}_{u}$ are defined  in \eqref{E:R1}.

The transport ``flavor'' of  the previous definitions can be seen in the next property that is immediate to verify: for any $\gamma\in \Geo(X)$ such that 
$(\gamma_0,\gamma_1)\in \Gamma_{u}$, then
\[
(\gamma_s,\gamma_t)\in \Gamma_{u},\quad \forall \  0\leq s\leq t\leq 1.
\]

\noindent
Finally, recall the definition of the  \emph{forward and backward branching points} of $\T_{u}$ that  was introduced in \cite{cava:MongeRCD}:
\begin{align}
&A_+:=\{x\in{\T_{u}: \exists\,z,w\in{\Gamma_{u}(x)}, (z,w)\notin{R_{u}}}\},\label{APlus}\\
&A_-:=\{x\in{\T_{u}: \exists\,z,w\in{\Gamma_{u}(x)^{-1}}, (z,w)\notin{R_{u}}}\}.\label{AMinus}
\end{align}

\noindent
Once branching points are removed, we obtain the \emph{non-branched transport set}  and the \emph{non-branched transport relation},
$$
\T^b_{u}:=\T_{u}\setminus (A_+\cup A_{-}), \quad
R_{u}^b:=R_{u}\cap (\T_{u}^b\times \T_{u}^b).
$$
The following was obtained in \cite{cava:MongeRCD} and highlights the motivation for removing branching points.

\begin{proposition}[Transport relation is an equivalence relation on the non-branched transport set]
The non-branched transport relation
$R_{u}^b\subset X\times X$ is an equivalence relation on $\T_{u}^b$.
\end{proposition}

\noindent
Noticing that once we fix $x\in \T_{u}^b$, for any choice of  
$z,w\in R_{u}(x)$, there exists $\gamma\in  \Geo(X)$ such that 
\[
\{x,z,w\} \subset  \{\gamma_s: s \in{[0,1]}\},
\]
it is not hard to deduce that each equivalence class is a geodesic. 
 These equivalence classes are sometimes called transport rays 
\cite{EvansGangbo99} 
or needles \cite{klartag}.

It is a classical procedure then to construct an $\mm$-measurable quotient map 
$\QQ$ for the equivalence relation $R^{b}_{u}$ over $\T_{u}^{b}$; 
in particular, there will be an $\mm$-measurable quotient 
set $Q \subset \T_{u}^{b}$ which is the image of $\QQ$. The Disintegration Theorem then implies the following disintegration formula:
\begin{equation}\label{E:disint}
\m \llcorner_{\T^b_{u}}= \int_{Q} \m_{\alpha} \mathfrak{q}(d\alpha),
\end{equation}
where $\qq = \QQ_{\sharp} \m \llcorner_{\T_{u}^b}$, and for 
$\qq$-a.e. $\alpha \in Q$
we have $\mm_{\alpha} \in \mathcal{P}(X)$, 
$\mm_{\alpha}(X \setminus X_{\alpha}) = 0$, where we have used the notation $X_{\alpha}$ to denote the equivalence class of the element $\alpha \in Q$ 
(indeed $X_{\alpha} = R_{u}^{b}(\alpha)$).

\begin{remark}\label{R:disintegration}
For a brief account on  the Disintegration Theorem, we refer to \cite{biacar:cmono} 
and references therein (see also \cite{CMi}). It is worth mentioning here 
that the map $Q \ni \alpha \mapsto \mm_{\alpha} \in \mathcal{P}(X)$ is essentially unique 
(meaning that any two maps for which \eqref{E:disint} holds true have to coincide 
$\qq$-a.e.) thanks to the assumption $\mm(X) = 1$,
while $\mm_{\alpha}(X \setminus X_{\alpha}) = 0$ 
(also called strongly consistence of the disintegration) is a consequence of 
the existence an $\mm$-measurable quotient map $\QQ$.
\end{remark} 

Again in \cite{cava:MongeRCD}, it was proved also that assuming the $\RCD(K,N)$ condition 
(which enhances $\CD(K,N)$ with an infinitesmal Hilbertianity assumption),
the measure of the set of branching points is zero. As already observed several times in the literature, the  $p=2$ proof only requires all optimal plans to be maps,
and so the same argument works for $CD_p(K,N)$ and any $p > 1$:    
\begin{theorem}[Negligibility of forward and backward branching points]
\label{teo:branch}
Let $(X,\sfd, \mm)$ be a m.m.s. such that for any 
$\mu_0,\mu_1\in{\mathcal{P}_p(X)}$ with $\mu_0 \ll \m$ any optimal transference plan for $W_p$ is concentrated on the graph of a function. Then
\[
\m(A_{+})=\m(A_{-})=0.
\]
\end{theorem}

\noindent
From Theorem \ref{teo:kell}, the $p$-essentially non-branching  hypothesis implies that  for every $\mu_0,\mu_1\in \mathcal{P}_p(X)$ with $\mu_0 \ll m$ there exists a  unique $p$-optimal plan and it is induced by a map. 
Hence, the assumptions of Theorem \ref{teo:branch} are satisfied, and therefore   
\begin{equation} \label{E:measurezero}
\m(A_{+})=\m(A_{-})=0.
\end{equation}
Putting together \eqref{E:disint} and  \eqref{E:measurezero} we obtain: 
\begin{equation}\label{E:disint2}
\m \llcorner_{\T_{u}}= \int_{Q} \m_{\alpha} \mathfrak{q}(d\alpha).
\end{equation}
In what follows we will prove that $(X,\sfd,\mm)$ verifies $\CD_{u}^{1}(K,N)$.
%%%%%%%%%%%%%%%%%%%%%%%%%%%%%%%%%%%%%%%%%%%%%%%%%%%%%%%%%%%%%%%%%%%%%%%%%%%%%%%%%
%%%%%%%%%%%%%%%%%%%%%%%%%%%%%%%%%%%%%%%%%%%%%%%%%%%%%%%%%%%%%%%%%%%%%%%%%%%%%%%%%
\subsection{Curvature estimates}
\label{Ss:curvature}
Recalling Definition \ref{D:defCD1}, one will observe that to prove $(X,\sfd,\mm)$ verifies $\CD_{u}^{1}(K,N)$ it suffices to show that, for $\mathfrak{q}$-a.e. $\alpha\in{Q}$, the one dimensional metric measure space 
$(X_{\alpha},\sfd,\mm_{\alpha})$ is a $\CD(K,N)$ space, i.e. if 
$X_{\alpha}$ is isometric to $[0,L_{\alpha}]$ where 
$L_{\alpha}$ is the length of $X_{\alpha}$ then, 
$$
\mm_{\alpha} = h_{\alpha} \mathcal{L}^{1}\llcorner_{[0,L_{\alpha}]}, 
\qquad 
\left(h_{\alpha}^{\frac{1}{N-1}}\right)'' + \frac{K}{N-1} h_{\alpha}^{\frac{1}{N-1}} \leq 0,
$$
where the inequality has to be understood in the distributional sense.
Notice indeed that, by construction, the transport rays $X_{\alpha}$ 
are the maximal  totally-ordered subsets of $\T_u^b\subset X$ 
under the partial-ordering $\leq_{u}$ given
 by $\Gamma_u$.

First we recall a result relating $\sfd^{p}$-cyclically monotone sets 
to $\sfd$-cyclically monotone set, presented in \cite{cava:decomposition} for $p = 2$.

\begin{lemma}[Certain $\sfd$-cyclically monotone sets are also $\sfd^p$-cyclical monotone] 
\label{lem:pmonotone}
Let $p > 1$ be any real number and let $\Delta \subset \Gamma_{u}$ be any set such that 
\[
(x_0,y_0), (x_1,y_1)\in{\Delta} \,\, 
\implies (u(y_1)-u(y_0))\cdot (u(x_1)-u(x_0)) \geq 0.
\]
Then $\Delta$ is $\sfd^p$-cyclically monotone.
\end{lemma}
\begin{proof}
By hypothesis the set 
\[
\Lambda:=\{ (u(x),u(y)):(x,y)\in{\Delta} \} \subset \R^2
\]
is monotone in the Euclidean sense. Since $\Lambda\subset \R^2$,  it is a standard fact that it is also c-cyclically monotone, for any  cost $c(x,y)=\vartheta( |x-y|)$ with $\vartheta:[0,+\infty) \to [0,+\infty)$ convex and such that $\vartheta(0)=0$. Hence, in particular, $\Lambda$ is $|\cdot|^p$-cyclically monotone.

Fix now $\{(x_i,y_i)\}_{i=1}^n\subset \Delta$. Using that   $u$ is 1-Lipschitz and  $\Delta\subset \Gamma$,  it turns out that 
\begin{align*}
\sum_{i=1}^n \sfd^p(x_i,y_i)&=\sum_{i=1}^n |u(x_i)-u(y_i)|^p\\
&\leq \sum_{i=1}^n |u(x_i)-u(y_{i+1})|^p
\leq \sum_{i=1}^n \sfd^p(x_i,y_{i+1}).
\end{align*}
Hence the claim.
\end{proof}
\begin{example}[Dimensional count in the smooth case]
If $d$ is the geodesic distance on an $n$-dimensional Riemannian manifold $X$ (or Euclidean space),
then --- away from the cut locus --- any $d$-cyclically monotone subset $\Delta$ is contained in a $n+1$ dimensional subset of $X^2$, the extra dimension being due to the degeneracy of $d$ along the direction of transport \cite{Pass12}. 
On the other hand,  if the left projection $P_1(\Delta) \subset \{\tilde u=0\}$ for some
$C^1$ function $\tilde u$ whose derivative is non-vanishing on its zero set,  we expect the dimension of $\Delta$
to be reduced to $n$, which coincides with the dimensional bound on a $d^p$-cyclically monotone set for $p>1$.
This example helps motivate both the previous lemma and the construction to follow.
\end{example}

 Similarly, in the nonsmooth setting, fixing $\delta \in \R$ and considering pairs $\Delta \subset \Gamma_u$ 
of partners $(x,y) \in \Delta$ whose lower endpoint lies on a fixed level set $u(y)=\delta$,
it follows that $\Delta$ is $\sfd^p$-cyclically monotone for all $p>1$.  Equivalently,
for each  $C\subset \T_{u}^b$ and $\delta\in{\R}$, the set $\Delta:=(C\times \{u=\delta\})\cap\Gamma_{u}$ is $\sfd^p$-cyclically monotone. 
Setting
$$
C_{\delta}=P_1((C\times \{u=\delta\})\cap \Gamma_{u}),
$$ 
 we see that if $\m(C_{\delta})>0$, then by Theorem \ref{teo:kell}, 
there exists a unique $\nu\in{\OptGeo_{p}(\mu_{0},\mu_{1})}$ such that
\[
(\ee_0)_{\sharp}\nu=\m(C_{\delta})^{-1}\m\llcorner{C_{\delta}}, \qquad (\ee_0,\ee_1)_{\sharp}\nu(C\times \{u=\delta\}\cap \Gamma_{u})=1,
\]
 and whose push-forwards by $\ee_{t}$ verify the entropic concavity statement \eqref{eq:CDKN-def} for all $t \in [0,1]$. 
Letting $C$ and $\delta$ vary, it is a standard procedure, see for example 
\cite{cava:MongeRCD}, to deduce that: 
\begin{enumerate}
\item[-] for $\q$-a.e. $\alpha\in{Q}$, the conditional probabilities $\mm_{\alpha}$ are absolutely continuous w.r.t. $\L^1\llcorner_{X_{\alpha}}$;
\item[-] if $\mm_{\alpha} = h_{\alpha}\L^1\llcorner_{X_{\alpha}}$, then 
$h_{\alpha} > 0$ in the relative interior of $X_{\alpha}$ and is locally Lipschitz.
\end{enumerate}

\begin{figure}[hbt!]
    \centering
    \scalebox{0.5}{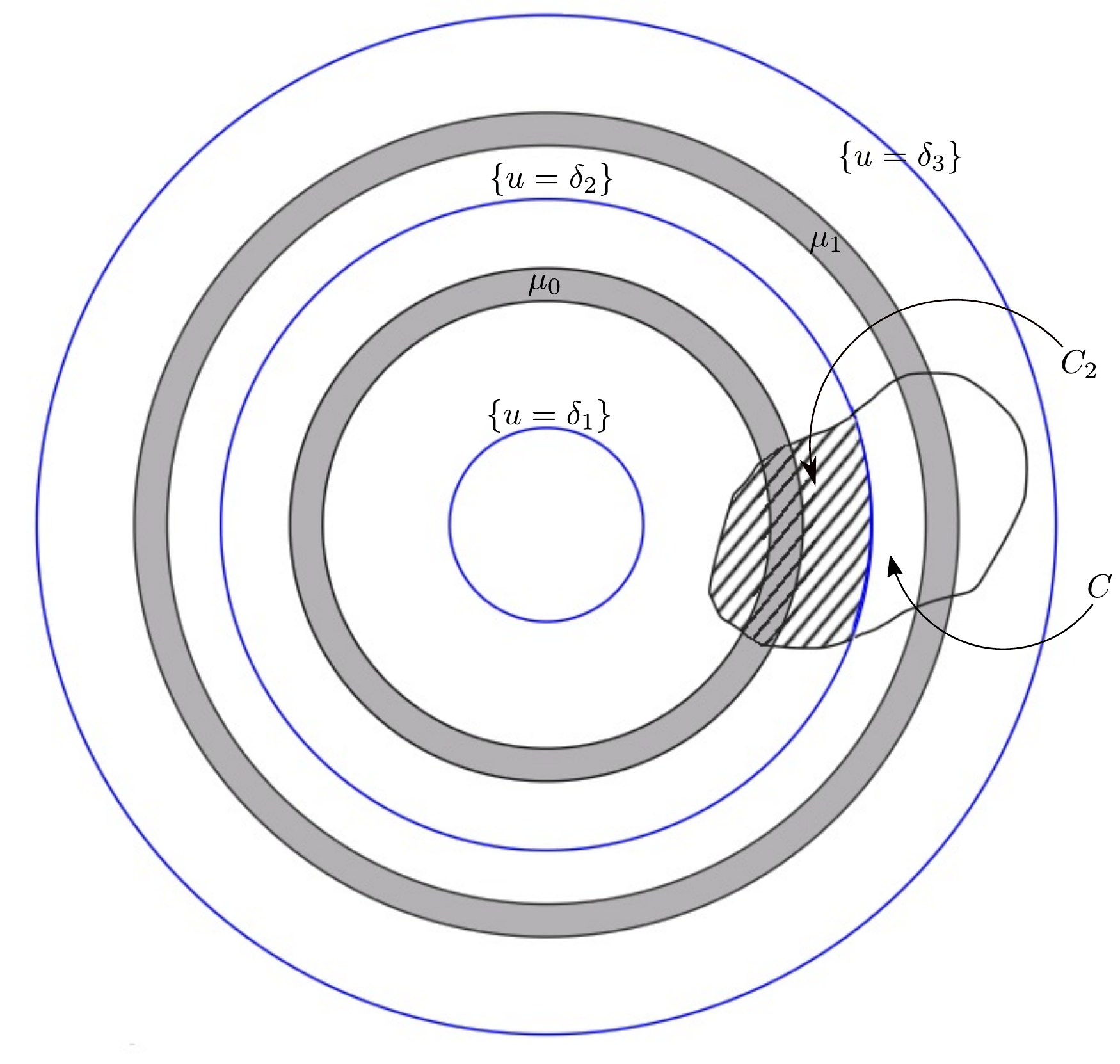}
   \caption{(The sets $C_\delta$) Transporting the sets $\mu_0$ to $\mu_1$ along radial transport geodesics determined by a radial 1-Lipschitz function $u$ associated to the radial Kantorovich potential $\varphi$. If we assume that $u$ behaves like the euclidean norm, then we see that $C_{\delta_1}=\varnothing, C_{\delta_2}=C_2, C_{\delta_3}=C.$ }
\end{figure}
The next step is to prove the $\CD(K,N)$ inequality 
for $\qq$-a.e.\ one-dimensional density $h_{\alpha}$.
This follows repeating verbatim the proof of \cite[Theorem 4.2]{CM17a} where the same 
implication was proved assuming $\CD_{2,loc}(K,N)$ and $2$-essentially non-branching. 
The main ingredient being Lemma \ref{lem:pmonotone} for $p = 2$, the argument
carries over for any $p > 1$. 

Putting together what has been discussed so far, we see that we have obtained the following:
\begin{theorem}[Non-branching $\CD_{p,loc}$ spaces are $\CD^1_{Lip}$]
\label{teo:cm} 
Let $(X,\sfd, \m)$ be a p-essentially non-branching m.m.s. 
satisfying the $\CD_{p,loc}(K,N)$ condition for some $p\in{(1,+\infty)}$, $K\in{\R}$, and $N\in{[1,\infty)}$ and $\mm(X) = 1$. 

Then, for any fixed $1$-Lipschitz function $u:X\to \R$, the transport relation 
$R^{b}_{u}$ induces on the  transport set a disintegration of $\m\llcorner_{\T_u}$ into conditional measures, $\mm_{\alpha}$, that for $\qq$-a.e. $\alpha$ satisfy
$\mm_{\alpha} = h_{\alpha} \mathcal{L}^{1}\llcorner_{X_{\alpha}}$ and:
\[
h_{\alpha}((1-s)t_0+st_1)^{1/(N-1)}\geq \sigma_{K,N-1}^{(1-s)}(t_1-t_0)h_{\alpha}(t_0)^{1/(N-1)}+\sigma_{K,N-1}^{(s)}(t_1-t_0)h_{\alpha}(t_1)^{1/(N-1)},
\]
for all $s\in{[0,1]}$ and for $t_0,t_1\in[0,L_{\alpha}]$ with $t_0<t_1$,
where we have identified the transport ray $X_{\alpha}$ 
with the real interval $[0,L_{\alpha}]$ having the same length.
\end{theorem}

Notice that the $\qq$-measurability of the disintegration, ensured by the Disintegration 
Theorem, implies joint measurability of the map  $(\alpha,t) \to h_{\alpha}(t)$.

\begin{remark}[Enhancing $\CD^1_{Lip}$]\label{R:CD1strong}
%\marginpar{FC: clarified with more details Remark \ref{R:CD1strong}. 
%Notice that $\mm(X) = 1$ is used in the uniqueness of disintegration theorem.
%}
It is worth underlining that the conclusion of Theorem \ref{teo:cm} is actually stronger than claiming that $(X,\sfd,\mm)$ verifies $\CD^{1}_{Lip}(K,N)$. 
Notice, indeed, that while $\CD^{1}_{Lip}(K,N)$ asks for \emph{a} disintegration of 
$\mm\llcorner_{\T_{u}}$  (no partition required, see Definition \ref{D:defCD1}) where each conditional measure is concentrated along 
a maximal transport ray and verifies $\CD(K,N)$, Theorem \ref{teo:cm} 
shows that we have a \emph{partition} of the 
transport set made of maximal transport rays and the associated 
\emph{essentially unique} disintegration verifies $\CD(K,N)$ (recall Remark \ref{R:disintegration}).  
In what follows we will show that this property is enough to prove that 
$(X,\sfd,\mm)$ also verifies $\CD_{q}(K,N)$ for any $q > 1$, 
provided it is also $q$-essentially non-branching.

To complete the picture we mention that in \cite[Proposition 8.13]{CMi} 
it is shown that $\CD_{Lip}^{1}(K,N)$ coupled with 
essentially non-branching (hence $p =2$) 
implies that the disintegration of $\mm\llcorner_{\T_{u}^{b}}$ 
coming from the partition induced by the transport relation 
$R_{u}^{b}$ indeed verifies all the conditions required by $\CD_{Lip}^{1}(K,N)$. 
We refer to \cite[Proposition 8.13]{CMi} for additional details.
\end{remark}

\begin{remark}[Strategy of proof]\label{Theorem4.4Proof}
Here we briefly comment on the technique used in \cite{CM17a} to prove Theorem $\ref{teo:cm}$.
The idea is to first establish the existence of a disintegration of $\mm$ into a collection of conditional 
measures, $\left\{\mm_{\alpha}\right\}_{\alpha\in{}Q}$, that are
supported along transport rays determined by an arbitrary $1$-Lipschitz function $u$ as in $\eqref{E:disint2}$.
In particular, we may express $\mm_{\alpha}$, the portion of measure of $\mm$ that lives on the transport geodesic of parameter $\alpha$, as
\begin{equation*}
\mm_{\alpha}=g(\alpha,\cdot)_{\sharp}\left(h_{\alpha}(t)\mathcal{L}^{1}\left(dt\right)\right)
\end{equation*}
where $g:Q\times[0,1]\to{}X$ is such that for each $\alpha$ we have that $\text{dom}(g(\alpha,\cdot))$ is convex and
$h_{\alpha}:\text{dom}(g(\alpha,\cdot))\to[0,\infty)$.
Next, the quotient set $Q$, which labels the various transport rays, is covered by a countable disjoint collection of sets
$\left\{Q_{i}\right\}_{i\in{}I}$ where each $Q_{i}$ is contained in a rational level set of $u$.
Finally, along each $Q_{i}$ we consider the transport of one uniform measure to another, of possibly differing size, along the transport rays of $u$.
More specifically, our countable decomposition is constructed to provide for each $i$ 
a uniform subinterval
\begin{equation*}
(a_{0},a_{1})\subset\text{dom}(g(\alpha,\cdot))\hspace{10pt}\text{for all }\alpha\in{}Q_{i}
\end{equation*}
as well as real numbers  $A_{0}, A_{1} \in (a_0,a_1)$  and 
$L_{0},L_{1} \in (0,\infty)$ such that
\begin{equation*}
A_{0}+L_{0}<A_{1}\ \mbox{\rm and}\ A_{1}+L_{1}<a_{1}.
\end{equation*}
This allows us to consider the measures
\begin{equation*}
\mu_{0}=\int_{Q_{ i}}\!{}g(\alpha,\cdot)_{\sharp}\left(\frac{1}{L_{0}}\mathcal{L}_{[A_{0},A_{0}+L_{0}]}^{1}(dt)\right)\qq(d\alpha),
\hspace{20pt}
\mu_{1}=\int_{Q_{i}}\!{}g(\alpha,\cdot)_{\sharp}\left(\frac{1}{L_{1}}\mathcal{L}_{[A_{1},A_{1}+L_{1}]}^{1}(dt)\right) \qq(d\alpha).
\end{equation*}
Transporting these measures allows us to  deduce concavity information for the 
density $h_{\alpha}$ of $\qq$-a.e.\ $\mm_{\alpha}$ from the entropic 
concavity \eqref{eq:CDKN-def} asserted by $\CD_{p}(K,N)$.
\end{remark}
%%%%%%%%%%%%%%%%%%%%%%%%%%%%%%%%%%%%%%%%%%%%%%%%%%%%%%%%%%%%%%%%%%%%%%%%%%%
%%%%%%%%%%%%%%%%%%%%%%%%%%%%%%%%%%%%%%%%%%%%%%%%%%%%%%%%%%%%%%%%%%%%%%%%%%%
%%%%%%%%%%%%%%%%%%%%%%%%%%%%%%%%%%%%%%%%%%%%%%%%%%%%%%%%%%%%%%%%%%%%%%%%%%%
%%%%%%%%%%%%%%%%%%%%%%%%%%%%%%%%%%%%%%%%%%%%%%%%%%%%%%%%%%%%%%%%%%%%%%%%%%%
%%%%%%%%%%%%%%%%%%%%%%%%%%%%%%%%%%%%%%%%%%%%%%%%%%%%%%%%%%%%%%%%%%%%%%%%%%%
\section{Curvature-Dimension conditions: from $p=1$ to $q >1$}
\label{S:anyp}
Before tackling  Theorems $\ref{thm:main1}$ and $\ref{thm:main2}$ we explore an example which illustrates some of the strategies and notations used.
\begin{example}[Radial transport]\label{E:radial}
Let $X=\mathbb{R}^{n}$, $\sfd$ be Euclidean distance, and set $\mm=\mathcal{L}^{n}$.
Let $\mu_{0}( dx)=\frac{1}{\omega_{n}|x|^{n-1}} \mathcal{L}^{n} \llcorner_{A_{1,2}}(dx)$ and
$\mu_{1}(dx)=\frac{1}{\omega_{n}|x|^{n-1}}\mathcal{L}^{n} \llcorner_{A_{3,4}}(dx)$ where for
$0<s<r<\infty$,  $A_{s,r}$ is defined as the spherical shell
\begin{equation*}
A_{s,r}=B_{r}(0)\setminus{}B_{s}(0).
\end{equation*}
We use the cost $c(x,y)=\frac{\sfd(x,y)^{q}}{q}$ where $1<q<\infty$.
For this transport problem, the optimal map is
$T(x)=(|x|+2)\frac{x}{|x|}$, 
the Kantorovich potential is $\varphi(x)=-2^{q-1}|x|$, 
and its interpolated potentials are
\begin{equation*}
\varphi_{t}(x)
=
\begin{cases}
\frac{-|x|^{q}}{qt^{q-1}},& {\rm if}\ |x|\le{}2t,\\
-2^{q-1}\left[|x|-\frac{2t}{q'}\right],& {\rm if}\ 2t<|x|,
\end{cases}
\end{equation*}
where $q'$ is the H\"{o}lder dual to $q$.
It is possible to show that the set $G_\varphi$ of $(\varphi, q)$-Kantorovich geodesics \eqref{Kantorovich geodesics}
consist of all segments of length two pointed away from the origin. 
Notice that not all such geodesics are involved in the transport of $\mu_{0}$ to $\mu_{1}$: indeed only 
those starting in the source $A_{1,2}$ (and therefore ending in the target $A_{3,4}$) 
 are.
In particular, only the subset of geodesics starting at a point in $A_{1,2}$ will have mass passing along 
them at all times $t\in(0,1)$.
This restriction should be compared to condition $3$ from  Definition $\ref{GoodGeodesics}$.
In particular, we use $G$ to denote a good subset of $G_{\varphi}$ of full measure which meet the stipulations of  Definition $\ref{GoodGeodesics}$.

Since we wish to apply the Disintegration Theorem, we have to associate the geodesics of $G_{\varphi}$ with the transport rays of a $1$-Lipschitz function.
 We do so by choosing our $1$-Lipschitz function to be the signed distance 
to a level set of $\varphi$. 
In our example we can use the norm since $\varphi$ is a monotone radial function.
However, in the general case, we must use the signed distance 
$d_{a,s} := d_{\varphi_s-a}$
with respect to the $a$ level set of $\varphi$.
Note that the ordinary distance function was not used so that we could refer to the level sets of $ d_{a,s}$ uniquely.
This idea is the basis of the discussion in subsection $\ref{Ss:L1OT}$.
In both cases we see that we are working with a subset, $G$, of the transport set according to the $1$-Lipschitz 
function we chose.
This should be compared to Lemma $\ref{lem:contenuto}$.

Finally,  we demonstrate how the change of variables formula from Theorem $\ref{teo:changeof}$ applies to our example.
 For $0<t<1$ and $\gamma\in{}G$, the
interpolating maps, measures, and densities are given by:
\begin{align*}
T_{t}(x)&=\left(|x|+2t\right)\frac{x}{|x|},\\
\mu_{t}(dx)&=\frac{1}{\omega_{n}|x|^{n-1}}
\mathcal{L}^{n} \llcorner_{A_{1+2t,2+2t}}(dx), \\
\mbox{\rm and}\ \rho_{t}(\gamma_{t})&=\frac{1}{\omega_{n}\left(|\gamma_{0}|+2t\right)^{n-1}}.
\end{align*}
Hence, for $s,t\in(0,1)$ we have
\begin{equation}\label{ChangeOfVariablesCalculation}
\frac{\rho_{t}(\gamma_{t})}{\rho_{s}(\gamma_s)}=\left(\frac{1+\ell+2s}{1+\ell+2t}\right)^{n-1}
\end{equation}
if $|\gamma_{0}|=1+\ell$.
For fixed $s\in(0,1)$, we note that if $a=-2^{q-1}\left[1+\ell+\frac{2s}{q}\right]$ for $0\le\ell\le1$ and $\gamma\in{}G$ is a geodesic such that
$\left|\gamma_{0}\right|=1+\ell$ then
\begin{equation*}
\varphi_{s}\left(\gamma_s\right)=a.
\end{equation*}
In particular, using this notation we can write $G_{\as}=\left\{\gamma\in{}G:\varphi_{s}(\gamma_s)=\awos\right\}$.
Hence,
\begin{align*}
\ee_{s}\left(G_{a,s}\right)&=\partial{}B_{1+\ell+2s}(0)\\
\ee_{[0,1]}\left(G_{a,s}\right)&=A_{1+\ell,3+\ell}
\end{align*}
where $-\ell=2^{1-q}a+ 1 + \frac{2s}q$. 
Using the Disintegration Theorem, 
as in $\eqref{PolarRayDisintegration}$, for any $1\le \ell \le 1$, 
we obtain
\begin{equation*}
\mathcal{L}^{n}\llcorner_{A_{1+\ell,3+\ell}}(dx)
= \int_{\partial B_1(0)} |x|^{n-1}{\mathcal H}^1\llcorner_{\{ r\alpha \mid 1+\ell \le r \le 3+ \ell\}} {\mathcal H}^{n-1}(d\alpha)
\end{equation*}
where ${\mathcal H}^{k}$ denotes $k$-dimensional Hausdorff measure.
Notice that we can rewrite this as
\begin{align}
\mathcal{L}^{n}\llcorner_{A_{1+2\ell,3+2\ell}}
&=\int_{\partial{}B_{1+\ell+2s}(0)}\!{}
g^{a,s}\left(\alpha,\cdot\right)_{\sharp}\left(2\left(\frac{1+\ell+2t}{1+\ell+2s}\right)^{n-1}\chi_{[0,1]}(t){d}t\right){\mathcal H}^{n-1}(d\alpha)
\nonumber\\
&=\int_{0}^{1}g^{a,s}(\cdot,t)_{\sharp}\left(2\left(\frac{1+\ell+2t}{1+\ell+2s}\right)^{n-1}
\chi_{[0,1]}(t){d}{\mathcal H}^{n-1}\right)\mathcal{L}^1(dt)
\label{PushforwardRepresentation}
\end{align}
where $g^{a,s}:{\ee}_{s}(G_{a,s})\times[0,1]\to{}X$ and $g^{a,s}(\alpha,\cdot)= \ee_{s} \llcorner_{G_{a,s}}^{-1}(\alpha)$.
Hence, $h_{\alpha}^{a,s}(t)=\left(\frac{1+\ell+2t}{1+\ell+2s}\right)^{n-1}$
where we have normalized this function so that $h_{\alpha}^{a}(s)=1$.
Next notice that
\begin{equation}\label{TransitionPhi}
\Phi_{s}^{t}(x)=-2^{q-1}\left[|x|-2t+\frac{2s}{q}\right].
\end{equation}
We may also compute that
\begin{equation*}
\partial_{\tau}\Big\vert_{\tau=t}\Phi_{s}^{\tau}(x)=2^{q}.
\end{equation*}
This allows us to show that
\begin{equation*}
\frac{\partial_{\tau}\Big\vert_{\tau=t}\Phi_{s}^{\tau}(\gamma_{t})}{\ell^{q}(\gamma)}\cdot\frac{1}{h_{\gamma_s}^{\varphi_{s}(\gamma_s), s}(t)}
=\left(\frac{1+\ell+2s}{1+\ell+2t}\right)^{n-1}
\end{equation*}
if $\left|\gamma_{0}\right|=1+\ell$ which, of course, matches $\eqref{ChangeOfVariablesCalculation}$ and verifies Theorem $\ref{teo:changeof}$.
Note that in general one will not have such explicit information.
As such, an expression like $\eqref{PushforwardRepresentation}$ will be not at disposal; hence, it is necessary to  deduce information by comparing the
disintegration described in $\eqref{PushforwardRepresentation}$ with another one.
Observe that the measure being pushed forward in $\eqref{PushforwardRepresentation}$ lives on 
$\ee_{t}(G_{a,s})$ and was obtained from a
disintegration with respect to a time varying partition of $\ee_{t}(G_{a,s})$.
For the second disintegration we instead focus on varying the level set values $a$ to form a partition of $\ee_{t}(G)$.
This description should be compared with subsection $\ref{Ss:LqOT}$ and the comparison done in subsection $\ref{Ss:comparison}$.
\end{example}

Let $(X,\sfd,\m)$ be a $p$-essentially non-branching metric measure space satisfying $\CD_p(K,N)$ and, consequently from Theorem \ref{teo:cm}, 
also the enhanced  $\CD^{1}_{Lip}(K,N)$ described in Remark \ref{R:CD1strong}.
This will be needed to close the argument: 
the enhanced $\CD^{1}_{Lip}(K,N)$ will give a ``canonical'' way of disintegrating 
the measure $\mm$ that will be crucial in the implementation of  the strategy outlined in 
the last few lines of Example \ref{E:radial}.

Given any $q > 1$, we will prove that $(X,\sfd,\m)$ also verifies $\CD_{q}(K,N)$,
provided the space is $q$-essentially non-branching as well.
Recall that without loss of generality we can assume $\spt(\m)=X$ and 
we have the standing assumption that $\mm(X) = 1$.                                    

Fix $\mu_0,\mu_1\in{\mathcal{P}_{q}(X,\sfd,\mm)}$.                                                            
From the curvature assumption it follows that $(X,\sfd)$ is a geodesic space,
hence, from Section \ref{Ss:geomeas}, $(\mathcal{P}_{q}(X),W_{q})$ 
is a geodesic space as well;
therefore the set of $q$-optimal dynamical plan $\OptGeo_{q}(\mu_0,\mu_1)$ 
is non-empty.

Recall moreover that $\CD_{p}(K,N)$ implies
qualitative non-degeneracy \eqref{QND} by 
\cite{KellENB},
hence Theorem~\ref{teo:kell} yields 
a unique $\nu \in \OptGeo_{q}(\mu_0,\mu_1)$ and 
$$
[0,1] \ni t \mapsto \mu_t:=(\ee_t)_{\sharp}\nu = \rho_{t} \mm.
$$

Finally, let $\varphi:X\to \mathbb{R}$ be a Kantorovich potential for the Optimal transport problem from $\mu_0$ to $\mu_1$ associated to the  cost $c:=\sfd^q/q$. Recall that $G_{\varphi}\subset \Geo(X)$ denotes the set of $(\varphi,q)$-Kantorovich geodesics, i.e. all the geodesics $\gamma$ for which
\[
\varphi(\gamma_0)+\varphi^c(\gamma_1)=\frac{\sfd^q(\gamma_0,\gamma_1)}{q}.
\]
We  denote with $G_{\varphi}^0$ the  set of null $(\varphi,q)$-Kantorovich geodesics defined as follows:
\[
G_{\varphi}^0:=\{\gamma\in{G_{\varphi}}:\ell(\gamma)=0\},
\]
and its complement in $G_{\varphi}$ by $G^{+}_{\varphi}$.

Using \cite[Proposition 9.1]{CMi}, 
the $\MCP(K,N)$  condition implies some non-trivial regularity properties on the time
behaviour of the density $\rho_{t}$: indeed the implication $(1) \Rightarrow (4)$ 
of \cite[Proposition 9.1]{CMi} gives a Lipschitz-type bound whenever $\mu_{1}$ reduces 
to a Dirac mass $\delta_{o}$ for some $o \in X$ (notice that from 
\cite[Remark 9.4]{CMi} this implication does not require any type of essential non-branching property). Then the case of a general $\mu_{1}$ can be obtained via approximation: using the $q$-essential non-branching property in its equivalent 
formulation given by Theorem \ref{teo:kell}, one can repeat the arguments 
of \cite[Proposition 9.1]{CMi} in the implications $(4) \Rightarrow (2)$ and  
$(2) \Rightarrow (3)$ where the main points were uniqueness of optimal dynamical plans and upper semi-continuity of entropies, which are  both still valid in our framework.
We summarize this discussion in the next statement:
\begin{corollary}[Logarithmic finite difference bounds for interpolating densities along characteristics]
\label{cor:version}
Let $(X,\sfd,\m)$ be a $q$-essentially non-branching m.m.s.~verifying $\MCP(K,N)$. 
Then for all $\mu_0,\mu_1 \in{\mathcal{P}_q(X)}$  with $\mu_0 \ll \mm$ there exists a unique 
$\nu\in \OptGeo_{q}(\mu_0,\mu_1)$ and a map $S:X \to \Geo(X)$ such that  $\nu=S_{\sharp}\mu_0$. 

Moreover $\mu_t=(\ee_t)_{\sharp}\nu \ll \m$ for $t\in{[0,1)}$ and there exist versions of the densities $\rho_t=\frac{d\mu_t}{d\m}$,  such that  for $\nu$-a.e. $\gamma\in{\Geo(X)}$, for all $0\leq s\leq t<1$, it holds
\begin{equation}\label{E:Lipestim}
\rho_s(\gamma_s)>0, \quad \biggl{(} \tau_{K,N}^{(\frac{s}{t})}(\sfd(\gamma_0,\gamma_t)) \biggr{)}^N \leq \frac{\rho_t(\gamma_t)}{\rho_s(\gamma_s)}\leq \biggl{(} \tau_{K,N}^{(\frac{1-t}{1-s})}(\sfd(\gamma_s,\gamma_1)) \biggr{)}^{-N}.
\end{equation}
In particular, for $\nu$-a.e. $\gamma$, the map $t\mapsto \rho_t(\gamma_t)$ is locally Lipschitz on $(0,1)$ and upper semi-continuous at $t=0$.
\end{corollary}

A further consequence of Corollary \ref{cor:version} can be obtained considering 
the regularity property of the map $t \mapsto \mm(\ee_{t}(G))$, for some compact subset $G$
of $\varphi$-Kantorovich geodesics (see for instance \cite[Proposition 9.6]{CMi}).
\begin{proposition}[Near continuity of the evolution of $\spt \mu_t$] 
\label{P:densityreversed}
Let $(X,\sfd,\mm)$ be a $q$-essentially non-branching m.m.s. 
verifying $\MCP(K,N)$. For $\mu_{0}, \mu_{1} \in \mathcal{P}_{q}(X)$ 
with $\mu_{0} \ll \mm$, let $\nu$ 
denote the unique element of  $\OptGeo_{q}(\mu_{0},\mu_{1})$.

Then for any compact set $G \subset \Geo(X)$ with $\nu(G) > 0$,  such that \eqref{E:Lipestim} holds true for all $\gamma \in G$ and $0 \leq s\leq t < 1$, 
it holds for any $t \in (0,1)$: 
$$
\lim_{\ve \to 0+} \frac{\mathcal{L}^{1} \big( G(x) \cap (t- \ve , t+\ve) \big)   }{2\ve} = 1 \;\;\; \text{ in $L^1(\ee_{t}(G),\mm)$},
$$
where
$G(x)=   \bigcup\limits_{\gamma \in G} \gamma^{-1}(x)$.
\end{proposition}

Finally, we conclude this first part by recalling the definition
of a special class of Kantorovich geodesics.

\begin{definition}[Good collections of geodesics]\label{GoodGeodesics}
Given $\mu_{0}, \mu_{1} \in \mathcal{P}_{q}(X)$ with $\mu_0 \ll \mm$,
we say that $G\subset G_{\varphi}^{+}$ is a \emph{good} subset of geodesics
if the following properties hold true:

\begin{enumerate}
\item $G$ is  compact;
\item there exists a constant $c>0$ such that for every $\gamma \in{G}$:
$c\leq \ell(\gamma)\leq 1/c$;
\item for every $\gamma\in{G}$, $\rho_t(\gamma_t)>0$ for all $t\in{[0,1]}$ and the map $(0,1)\ni t \mapsto \rho_t(\gamma_t)$ is continuous;
\item the claim of Proposition \ref{P:densityreversed} holds true for $G$;
\item The map $\ee_t|_{G}:G \to X$ is injective.
\end{enumerate}
\end{definition}

From now on we will assume $G \subset G_{\varphi}^{+}$ to 
be  a good subset. In particular all the results contained in Sections 
\ref{Ss:L1OT}, \ref{Ss:LqOT}, and \ref{Ss:comparison} 
will be obtained tacitly 
assuming any optimal dynamical plan to be concentrated on a
good subset of geodesics.                                                                                        

We will dispose of this assumption in Section \ref{Ss:change} via an approximation 
argument. Notice indeed that under $q$-essentially non-branching and $\MCP(K,N)$ 
 for any $\nu \in \OptGeo_{q}(\mu_{0},\mu_{1})$ with $\mu_{0}\ll \mm$, and any 
$\varepsilon > 0$ there exists a \emph{good} compact subset $G^{\varepsilon}\subset G^{+}_{\varphi}$ such that $\nu(G^{\varepsilon})\geq \nu(G_{\varphi}^{+})-\varepsilon$ for any $\varepsilon>0$. 
Without loss of generality,  we can also assume that $G^{\varepsilon}$ increases 
 along any given sequence of $\varepsilon$ decreasing to $0$.
\smallskip

In what follows we will use a suitable collection of $L^{1}$-optimal transport problems
to decompose the Jacobian of the evolution of the $W_{q}$-geodesic 
$t\to \mu_{t}$ and to obtain key estimates on both components: our interest will be focused 
on finding a codimension-$1$ Jacobian orthogonal to the evolution 
and a one-dimensional counterpart. 
For both of these factors,
curvature estimates will be obtained via $L^{1}$- optimal transport techniques, 
in particular Theorem \ref{teo:cm},
 by comparing two families of conditional measures: one coming from 
the aforementioned $L^{1}$-optimal transport problem and the other one from 
the $q$-Kantorovich potential.

The decomposition technique will be very similar to the one developed 
in \cite{CMi}; we will not repeat all the proofs but just list 
the main differences and include additional details where needed.

%%%%%%%%%%%%%%%%%%%%%%%%%%%%%%%%%%%%%%%%%%%%%%%%%%%%%%%%%%%%%%%%%%%%%%%%%%%
%%%%%%%%%%%%%%%%%%%%%%%%%%%%%%%%%%%%%%%%%%%%%%%%%%%%%%%%%%%%%%%%%%%%%%%%%%%
%%%%%%%%%%%%%%%%%%%%%%%%%%%%%%%%%%%%%%%%%%%%%%%%%%%%%%%%%%%%%%%%%%%%%%%%%%%
\subsection{$L^1$ Partition}\label{Ss:L1OT}
For $s\in{[0,1]}$ and $a\in{\mathbb{R}}$, we define the set of geodesics $G_{a,s}\subset G_{\varphi}$ as follows:
\[
G_{a,s} 
=\{ \gamma\in{G}:\varphi_s(\gamma_s)=a\}.
\]
Let us observe that since $G$ is compact and $\ee_s:G\to X$ is continuous, 
$\ee_s(G)$ is still compact. 
Moreover, for $s\in{(0,1)}$, $\varphi_s:X\to\mathbb{R}$ is continuous and hence $G_{a,s}$ is compact as well.

\noindent
Let us fix $a\in{\varphi_s(\ee_s(G))}$. The aim of the next subsection will be to analyze  the structure of the evolution of the set $G_{a,s}$, i.e.  
$\ee_{[0,1]}(G_{a,s})$. 

From now on we will denote the signed-distance function from a level set 
$a$ of $\varphi_s$ with $d_{a,s}:=d_{\varphi_s-a}$ (recall the notation of \eqref{E:levelsets}). Since $d_{a,s}$ is a $1$-Lipschitz function, 
we can associate to it all the sets introduced in Section \ref{TransportSets},
 including
the transport ordering $\Gamma_{d_{a,s}}=\le_{d_{\as}}$, relation $R_{d_{\as}}= \Gamma_{d_{\as}} \cup \Gamma_{d_{\as}}^{-1}$ and set $\T_{d_{\as}} \subset P_1(R_{d_{\as}})$. 

\begin{lemma}
\label{lem:contenuto}
Let $(X,\sfd)$ be a geodesic space. Once $s\in{[0,1]}$ and $\awos\in{\varphi_s(\ee_s(G))}$ are fixed, then for each $\gamma\in{G_{\as}}$ and  for every $0\leq r \leq t \leq 1$, $(\gamma_r,\gamma_t)\in{\Gamma}_{d_{\as}}$. In particular, 
\[
\ee_{[0,1]}(G_{\as}) \subset \T_{\sfd_{\as}}.
\]
\end{lemma}

The proof goes along the same lines of \cite[Lemma 10.3]{CMi}
 which we have included for the reader's convenience.

\begin{proof}
 Let us fix  $\gamma\in G_{\as}$. By Corollary \ref{cor:diff} and Lemma \ref{lem:phibar} (2), we have that if $s\in{[0,1)}$ then for any $x\in{\{\varphi_s=\awos\}}$, it holds 
\[
\frac{\sfd^p(\gamma_s,\gamma_1)}{p(1-s)^{p-1}}= \varphi_s(\gamma_s)+ \varphi^c(\gamma_1)= \varphi_s(x)+ \varphi^c(\gamma_1) \leq \bar{\varphi}_s(x)+\varphi^c(\gamma_1) \leq \frac{\sfd^p(x,\gamma_1)}{p(1-s)^{p-1}}.
\]
Hence $\sfd(\gamma_s,\gamma_1)\leq \sfd(x,\gamma_1)$. In the same way, if $s\in{(0,1]}$, then for any $y\in{\{\varphi_s=\awos\}}$ we have that
\[
\frac{\sfd^p(\gamma_s,\gamma_0)}{ps^{p-1}}=\varphi(\gamma_0) - \varphi_s(\gamma_s)= \varphi(\gamma_0)- \varphi_s(y) \leq \frac{\sfd^p(y,\gamma_0)}{ps^{p-1}}.
\]
So   $\sfd(\gamma_s,\gamma_0)\leq \sfd(y,\gamma_0)$, which is also trivially satisfied in the case $s=0$. Thus, for any $x,y\in{\{\varphi_s=\awos\}}$  we have
\[
\sfd(\gamma_0,\gamma_1)\leq \sfd(\gamma_0,x)+\sfd(y,\gamma_1).
\]
 Taking the infimum over $x$ and $y$ we get that 
\[
\sfd(\gamma_0,\gamma_1)\leq d_{\as}(\gamma_0)-d_{\as}(\gamma_1),
\]
where the sign of $\sfd_{\as}$ was determined by the fact that $s\mapsto \varphi_s(\gamma_s)$ is decreasing. More precisely,
the latter relation turns out to hold as an equality by $1$-Lipschitz regularity of $d_{\as}$, thus $(\gamma_0,\gamma_1)\in{\Gamma_{\sfd_{\as}}}$. This implies that   for every $0\leq r \leq t \leq 1$, $(\gamma_r,\gamma_t)\in{\Gamma}_{d_{\as}}$.
\end{proof}
By Theorem \ref{teo:cm}  we have that, 
choosing $u=d_{\as}$, the following disintegration formula holds
\begin{equation}
\mm \llcorner_{\T_{\sfd_{\as}}}= \int_{Q} \hat{\mm}^{\as}_{\alpha} \hat{\q}^{\as}(d\alpha),
\end{equation}
where $Q$ is a section of the partition of $\T^b_{d_{\as}}$ given by the equivalence classes $\{R^b_{d_{\as}}(\alpha)\}_{\alpha\in{Q}}$, and for $\hat{q}^{\as}$-a.e. $\alpha\in{Q}$, $\hat{\mm}^{\as}_{\alpha}$ is a probability measure supported on the transport ray $X_{\alpha}=R_{d_{\as}}(\alpha)$ and $(X_{\alpha,}\sfd, \hat{\mm}^{\as}_{\alpha})$ verifies $\CD(K,N)$.
By Lemma \ref{lem:contenuto}, it follows that:
\[
\mm \llcorner_{\ee_{[0,1]}(G_{\as})}= \int_{Q} \hat{\mm}^{\as}_{\alpha}\llcorner_{\ee_{[0,1]}(G_{\as})} \hat{\q}^{\as}(d\alpha).
\]
From the very definition of $G_{\as}$ and the $p$-essentially non-branching property, in the previous disintegration formula the quotient set $Q$ can be naturally identified with $\ee_{s}(G_{\as})$; moreover, we can consider the 
Borel parametrization
$$
g^{\as}:\ee_s(G_{\as})\times[0,1]\to X, \quad g^{\as}(\alpha,\cdot)=
 (\ee_s{\llcorner_{G_\as}})^{-1}(\alpha),
$$ 
yielding the following disintegration formula:
\begin{equation}\label{E:h}
\m\llcorner_{\ee_{[0,1]}(G_{\as})}= \int_{\ee_s(G_{\as})} 
g^{\as}(\alpha,\cdot)_{\sharp}\left( h^{\as}_{\alpha}\cdot \L^1\llcorner_{[0,1]} \right) \q^{\as}(d\alpha),
\end{equation}
where $\q^{\as}$ is a Borel measure concentrated on ${\ee_s(G_{\as})}$, 
and for $\q^{\as}$-a.e. $\alpha\in{\ee_s(G_{\as})}$, 
 $h^{\as}_{\alpha}$ is a $\CD(\ell_s(\alpha)^2K,N)$ density on $[0,1]$. Notice 
 that the factor $\ell_s(\alpha)^2  =\mathcal{H}^1(X_\alpha)^2$ is due to the reparametrization of 
 the transport ray on $[0,1]$. 

This permits, invoking Fubini's theorem, to reverse the order of integration 
so to have:
\begin{equation}\label{E:disintL1}
\mm\llcorner_{\ee_{[0,1]}(G_{\as})}= \int_{[0,1]} g^{\as}(\cdot, t)_{\sharp}(h^{\as}_{\cdot}(t)\cdot \q^{\as}) \L^1(dt)
= \int_{[0,1]} \mm_t^{\as}\, \L^1(dt),
\end{equation}
where we defined
\[
\mm_t^{\as}:=g^{\as}(\cdot,t)_{\sharp}( h^{\as}_{\cdot}(t)\cdot \q^{\as}).
\]
Finally, the previous disintegration formula does not change if  we
multiply and divide conditional measures by $h_{\alpha}^{\as}(s)$; 
therefore, changing $\qq^{\as}$, we can assume $h_{\alpha}^{\as}(s)=1$,
yielding  $\mm_s^{\as}=\q^{\as}$ and 
\begin{equation}
\label{equ:mtms}
\mm_t^{\as}:=g^{\as}(\cdot,t)_{\sharp}( h^{\as}_{\cdot}(t)\cdot \mm_s^{\as} ).
\end{equation}
Moreover (see \cite[Proposition 10.7]{CMi}),
for any $s\in (0,1)$ and $\awos \in \varphi_s(\ee_s(G))$, the map 
$$
(0,1) \ni t \mapsto \mm^{\as}_{t}
$$
is continuous in the weak topology and if $\mm(\ee_{[0,1]}(G_{\as})) > 0$, 
then $\mm^{\as}_{t}(\ee_{t}(G_{\as}))  > 0$,  for all $t \in (0,1)$.
Finally,
$$
\forall t\in [0,1] \;\;\; \mm^{\as}_{t}(\ee_{t}(G_{\as})) =  \| \mm^{\as}_{t} \| \leq C \; \mm(\ee_{[0,1]}(G_{\as})),
$$
for some $C> 0$ depending only on $K$, $N$ and $\{\ell(\gamma): \gamma \in G_{\as}\}$.
%%%%%%%%%%%%%%%%%%%%%%%%%%%%%%%%%%%%%%%%%%%%%%%%%%%%%%%%%%%%%%%%%%%%%%%%%%%%%
%%%%%%%%%%%%%%%%%%%%%%%%%%%%%%%%%%%%%%%%%%%%%%%%%%%%%%%%%%%%%%%%%%%%%%%%%%%%%
\subsection{$L^q$ partition}\label{Ss:LqOT}
We will now consider a decomposition of $\mm$ into conditional measures induced
by Kantorovich potentials. 

Hence for any $s,t\in{(0,1)}$, let us consider $\awos\in{\Phi_s^t(\ee_t(G))=\varphi_s(\ee_s(G))}$. 
With such a choice of $\awos$, the compact set $\ee_t(G)$  admits a partition given by $\ee_t(G)\cap\{\Phi_s^t=\awos\}_{\awos\in{\R}}$ . 

Continuity of  $\Phi_s^t$ makes it possible to apply the Disintegration Theorem. 
 Since $\mm[\ee_t(G)]<\infty$,  there exists an essentially unique disintegration of 
$\mm \llcorner_{\ee_t(G)}$ strongly consistent with respect to the 
quotient map $\Phi_s^t$:
\begin{equation}
\label{equ:lppartition}
\mm \llcorner_{\ee_{t}(G)}= \int_{\varphi_s(\ee_s(G))} 
\hat{\mm}^t_{\as}\q^{t}_{s}(d\awos)
\end{equation}
where $\q_s^t=(\Phi_s^t)_{\sharp}\mm \llcorner_{\ee_t(G)}$ 
and $\hat{\mm}^t_{\as}$ is  a probability measure concentrated on 
the set $\ee_t(G)\cap\{\Phi^t_s=\awos\}=\ee_t(G_{\as})$.

Notice that, as one would expect,  being the image of a
time propagation of an intermediate Kantorovich potential,
the quotient set $\varphi_s(\ee_s(G))$ does not depend on $t$.

The next follows with no modification from \cite[Proposition 10.8]{CMi}.

\begin{proposition}
The following properties hold true:
\begin{itemize}
\item For any $s,t,\tau\in{(0,1)}$, the quotient measures $\q^t_s$ and $\q^{\tau}_s$ are mutually absolutely continuous;

\item For any $s,t\in{(0,1)}$, the quotient measure $\q_s^t$ is absolutely continuous with respect to Lebesgue measure $\mathcal{L}^1$ on $\mathbb{R}$.
\end{itemize}
\end{proposition}

Employing what we obtained so far, we can rewrite \eqref{equ:lppartition} in the following way:
\begin{equation}\label{E:disintlq}
\mm \llcorner_{\ee_{t}(G)}= \int_{\varphi_s(\ee_s(G))} 
\mm^t_{\as}\mathcal{L}^1(d\awos),
\end{equation}
where $\mm^t_{\as}:= (d\q^t_s/d \mathcal{L}^1)\cdot \hat{\m}^t_{\as}$
is concentrated on $\ee_t(G_{\as})$ for $\mathcal{L}^1$-a.e. $\awos\in{\varphi_s(e_s(G))}$. 

Over the set $\ee_{t}(G)$ we also have the measure $\mu_{t}$; as it can be lifted 
to the set $\Geo(X)$, it makes sense to notice that 
the family of sets $\{G_{\as}\}_{\awos\in{\R}}$ provides a partition of $G$. Hence an application of the  Disintegration Theorem guarantees the existence of an essentially unique disintegration of $\nu$   strongly consistent with respect $\varphi_s\circ \ee_s$:
\begin{equation}
\label{equ:disintnu}
\nu=\int_{\varphi_s(\ee_s(G))} \nu_{\as} \q_s^{\nu}(d\awos)
\end{equation}
where the probability measure $\nu_{\as}$ is concentrated on $G_{\as}$ for $\q_s^{\nu}$-a.e. $\awos\in{\varphi_s(\ee_s(G))}$. In particular, $q^{\nu}_s(\varphi_s(\ee_s(G)))=||\nu||=1$.

Multiplying \eqref{E:disintlq} by $\rho_{t}$ and applying $(\ee_{t})_{\sharp}$
to \eqref{equ:disintnu} produces the same measure $\mu_{t}$: this 
permits to deduce what follows. For all the missing details
we refer to \cite[Corollary 10.10]{CMi}.
\begin{corollary}\label{cor:1010}
We have the following
\begin{enumerate}
\item For any $s\in{(0,1)}$, the quotient measure $\mathfrak{q}^{\nu}_s$ is mutually absolutely continuous with respect to $\mathfrak{q}^{s}_s$. In particular, it is absolutely continuous with respect to $\L^1$.

\item For any $s,t\in{(0,1)}$ and $\L^1$-a.e. $\awos\in{\varphi_s(\ee_s(G))}$:
\[
\rho_t\cdot \mm^t_{\as}=q^{\nu}_s(\awos)\cdot (\ee_t)_{\sharp}\nu_{\as},
\]
where $q_{s}^{\nu}:=d\mathfrak{q}^{\nu}_s/d\L^1$. 
In particular, $\m^t_{\as}$ and $(\ee_t)_{\sharp}\nu_{\as}$ 
are mutually absolutely continuous for $\q_s^{\nu}$-a.e. 
$\awos\in{\varphi_s(\ee_s(G))}$.

\item For any $s\in{(0,1)}$ and $\mathfrak{q}_s^{\nu}$-a.e. $\awos\in{\varphi_s(e_s(G))}$, the maps 
\[
[0,1] \ni t \mapsto\rho_t \cdot \m_{\as}^t, \quad [0,1]\ni t \mapsto (\ee_t)_{\sharp} \nu_{\as}
\]
coincide  for $\L^1$-a.e. $t\in{[0,1]}$  up to a  positive multiplicative constant $C_{\as}$ depending only on $\as$.
\end{enumerate}
\end{corollary}
%%%%%%%%%%%%%%%%%%%%%%%%%%%%%%%%%%%%%%%%%%%%%%%%%%%%%%%%%%%%%%%%%%%%%%
%%%%%%%%%%%%%%%%%%%%%%%%%%%%%%%%%%%%%%%%%%%%%%%%%%%%%%%%%%%%%%%%%%%%%%
\subsection{Comparison between conditional measures}\label{Ss:comparison}
We will now link the seemingly unrelated disintegrations 
\eqref{E:disintL1} and \eqref{E:disintlq}.

Observe that $\mm^t_{\as}$ and $\mm^{\as}_t$ are concentrated on $\ee_t(G_{\as})$, for each $t\in(0,1)$ for $\L^1$-a.e. $\awos\in{\varphi_s(e_s(G))}$ and for each $\awos\in{\varphi_s(e_s(G))}$ and all $t\in{(0,1)}$, respectively.

The common feature of the two families of conditional measures
$\mm^t_{\as}$ and $\mm^{\as}_t$ is that they are both coming from 
a disintegration formula with quotient measure the Lebesgue measure.
We can exploit this property in the next lemma.
\begin{lemma}
For every $s,t\in{(0,1)}$ and $\awos\in{\varphi_s(e_s(G))}$, the limit 
\[
\m_t^{\as}=\lim_{\epsilon\to 0} \frac{1}{2\varepsilon}\m\llcorner_{e_{[t-\varepsilon, t+\varepsilon]}(G_{\as})}
\]
holds true in the weak topology.
\end{lemma}
\begin{proof}
Since $(0,1) \ni t\mapsto \mm^{\as}_{t}$ is continuous in the weak topology, and so together with \eqref{E:disintL1}, we see that for any $f \in C_b(X)$:
 \[
\lim_{\ve \to 0} \frac{1}{2\ve} \int_X f(z) \mm\llcorner_{\ee_{[t-\ve,t+\ve]}(G_{\as})}(dz)  = \lim_{\ve \to 0}\frac{1}{2\ve} \int_{t-\ve}^{t+\ve} \bigg(  \int_{X} f(z)\mm^{\as}_{\tau}(dz) \bigg)  \mathcal{L}^{1}(d \tau) = \int_{X} f(z) \mm^{\as}_{t}(dz) ,
\]
thereby concluding the proof. 
\end{proof}

We are now in position to compare $\mm^t_{\as}$ and $\mm^{\as}_t$ by 
comparing $\mm$ in a neighborhood of $\ee_{t}(G_{\as})$ obtained 
varying $t$ and then varying $\awos$. We refer to 
\cite[Theorem 11.3]{CMi} for all the details in the case $q = 2$  and simply note that the argument 
works for any $q> 1$; (the main ingredients needed for the proof are the disintegration formulas 
\eqref{E:disintL1}, \eqref{E:disintlq} and temporal regularity of $\Phi_{s}^{t}$
obtained in Section \ref{S:HopfLax}).                                

\begin{theorem}[Relating factorization by potential values and by $\varphi$-Kantorovich geodesics via Fubini]
\label{thm:comparison}
For any $s\in{(0,1)}$,
\[
\mm_s^{\as}=\ell^p_s\cdot \mm_{\as}^s, \quad \text{for}\quad \L^1\text{-a.e.} \,\, \awos\in{\varphi_s(\ee_s(G))}.
\]
Moreover, for any $s\in{(0,1)}$ and $\L^1$-a.e.  $t\in(0,1)$ including at $t=s$, $\partial_t \Phi^t_s(x)$ exists and is positive, and for $\L^1$-a.e. $\awos\in{\varphi_s(\ee_s(G))}$ and $\mm_{\as}^t$-a.e. $x$ we have:
\begin{equation}
\label{equ:mtas}
\mm_t^{\as}=\partial_t \Phi^t_s\cdot \mm_{\as}^t, \quad \text{for}\quad \L^1\text{-a.e.}\quad  \awos\in{\varphi_s(\ee_s(G))}.
\end{equation}
\end{theorem}
\medskip
%%%%%%%%%%%%%%%%%%%%%%%%%%%%%%%%%%%%%%%%%%%%%%%%%%%%%%%%%%%%%%%%%%%%%%%%%%%%%%%%
%%%%%%%%%%%%%%%%%%%%%%%%%%%%%%%%%%%%%%%%%%%%%%%%%%%%%%%%%%%%%%%%%%%%%%%%%%%%%%%%
\subsection{Change of variable formula}\label{Ss:change}
Building on Theorem \ref{thm:comparison}, we are now in position 
to write the Jacobian associated to the evolution of $\mu_{t}$ 
as the product of two factors. 

All the results obtained until now  will be used to prove the following:
\begin{theorem}[Change of variables formula]
\label{teo:changeof}
Let $(X,\sfd,\mm)$ be a $p$-essentially non branching m.m.s. satisfying $\CD_p(K,N)$ and assume it is also $q$-essentially non branching. 

Let us consider  $\mu_0,\mu_1\in{\mathcal{P}_q(X,\sfd,\mm)}$ and let  $\nu$ denote the unique element of $\OptGeo_{q}(\mu_0,\mu_1)$. Setting $\mu_t=(\ee_t)_{\sharp}\nu\ll\mm$, we will consider the densities $\rho_t:=d\mu_t/d\m,\,t\in{[0,1]}$, given by Corollary \ref{cor:version}.

Then for any $s\in{(0,1)}$, for $\L^1$-a.e. $t\in{(0,1})$ and $\nu$-a.e. $\gamma\in{G_{\varphi}^{+}}$, $\partial_{\tau}|_{\tau=t} \Phi_s^{\tau}(\gamma_t)$ exists and the following formula holds:
\begin{equation}
\label{equ:changeofv}
\frac{\rho_t(\gamma_t)}{\rho_s(\gamma_s)}= \frac{\partial_{\tau}|_{\tau=t} \Phi_s^{\tau}(\gamma_t)}{\ell^p(\gamma)}\cdot \frac{1}{h_{\gamma_s}^{\varphi_s(\gamma_s),s}(t)}.
\end{equation}
Here $h_{\gamma_s}^{\varphi_s(\gamma_s), s}$ is the $\CD(\ell(\gamma)^2K,N)$ density on $[0,1]$ from \eqref{E:h}, renormalized in such a way $h_{\gamma_s}^{\varphi_s(\gamma_s), s}(s)=1$.
Finally,  for all $\gamma\in{G}^{0}_{\varphi}$, it holds:
\begin{equation}
\rho_t(\gamma_t)=\rho_s(\gamma_s), \,\quad\forall t,s\in{[0,1]}.
\end{equation}
\end{theorem}
\begin{proof}
By \cite[Lemma 6.11]{CMi} and the discussion below Definition \ref{D:Ohta1},
$(X,\sfd,\m)$ verifies $\MCP(K,N)$ and Corollary \ref{cor:version} guarantees the existence of versions of the densities satisfying \eqref{E:Lipestim}.
For any $\varepsilon>0$, there exists a \emph{good} compact subset $G^{\varepsilon}\subset G^{+}_{\varphi}$ such that $\nu(G^{\varepsilon})\geq \nu(G_{\varphi}^{+})-\varepsilon$ and such that $G_{\varepsilon}$ increases along a sequence of $\varepsilon$ decreasing to $0$.
 Fixing $\varepsilon>0$ on this sequence and the good subset $G^{\varepsilon}$, let us set
\[
\nu^{\varepsilon}= \frac{1}{\nu(G^{\varepsilon})} \nu\llcorner_{G^{\varepsilon}}, \quad \mu_t^{\varepsilon}:=(\ee_t)_{\sharp} \nu^{\varepsilon} \ll \m.
\]
In particular we have that
$ \mu^{\varepsilon}_t=\frac{1}{\nu(G^{\varepsilon})}\mu\llcorner_{\ee_t(G_{\varepsilon})}$,  for all $t \in[0,1]$ 
and therefore:
\[
\mu^{\varepsilon}_t=\rho^{\varepsilon}_t \m, \,\,\,\,\rho^{\varepsilon}_t:= \frac{1}{\nu(G^{\varepsilon})}\rho_t|_{\ee_t(G^{\varepsilon})}, \,\, \forall t\in{[0,1]}.
\]
As we proved in Corollary \ref{cor:1010}, for each $s\in{(0,1)}$ and $\mathfrak{q}_s^{\varepsilon, s}$-a.e. $\awos\in{\varphi_s(\ee_s(G^{\varepsilon}))}$, 
the map $[0,1]\ni t \mapsto \rho_t \cdot \m_{\as}^{\varepsilon,t}$ coincides for $\L^1$-a.e. $t\in{[0,1]}$ with the geodesic $t\mapsto (\ee_t)_{\sharp}\nu_{\as}^{\varepsilon}$ up to a constant $C^{\varepsilon}_{\as}>0$.
Hence, for such $s$ and $\awos$, for $\L^1$ a.e $t\in{[0,1]}$, we have that for any Borel set $H\subset G^{\varepsilon}$ the quantity
\begin{equation}
\label{equ:x}
\int_{\ee_{t}(H)} \rho^{\varepsilon}_t(x)\m_{\as}^{\varepsilon,t}(dx)=C^{\varepsilon}_{\as} \int_{\ee_{t}(H)} (\ee_t)_{\sharp} \nu^{\varepsilon}_{\as}(dx)=C^{\varepsilon}_{\as}\nu^{\varepsilon}_{\as}(H)
\end{equation}
is constant in $t$, where in the last equality we used the injectivity 
of the map  $\ee_t :G^{\varepsilon}\to X$.  
By Theorem \ref{thm:comparison}, for $\L^1$-a.e. $t\in{(0,1)}$  and $\L^1$-a.e. $\awos\in{\varphi_s(G^{\varepsilon}_s)}$, 
$\partial_t \Phi^t_s(x)$ exists and is positive 
for $\mm_{\as}^{\varepsilon,t}$-a.e. $x$;  moreover \eqref{equ:mtas} holds. 
Thus, for all $\as$ and $t$ for which the previous condition and \eqref{equ:x} hold, we have 
\begin{align}
\label{equ:cont}
C^{\varepsilon}_{\as}\nu^{\varepsilon}_{\as}(H)&=\int_{\ee_t(H)} \rho^{\varepsilon}_t(x)\m_{\as}^{\varepsilon,t}(dx)=\int_{e_t(H)} \rho^{\varepsilon}_t(x)(\partial_t \Phi^t_s(x))^{-1}\m_{t}^{\varepsilon,\as}(dx)\\
&= \int_{\ee_s(H)} \rho^{\varepsilon}_t(g^{\as}(\alpha,t))(\partial_{\tau}|_{\tau=t} \Phi^{\tau}_s(g^{\as}(\alpha,t)))^{-1} h^{\as}_{\alpha}(t)\m_{s}^{\varepsilon,\as}(d\alpha)\notag\\ 
&= \int_{\ee_s(H)} \rho^{\varepsilon}_t(g^{\as}(\alpha,t))(\partial_{\tau}|_{\tau=t} \Phi^{\tau}_s(g^{\as}(\alpha,t)))^{-1} h^{\as}_{\alpha}(t) \ell^p_s(\alpha)\m_{\as}^{\varepsilon,s}(d\alpha)\notag 
\end{align}
where the two last equalities  follow  from  
\eqref{equ:mtms} and Theorem \ref{thm:comparison}, respectively.

Since the left-hand side of \eqref{equ:cont} does not depend on $t$,  it follows that for all $s\in{(0,1)}$ and for $q_s^{\varepsilon,s}$-a.e. $\awos\in{\varphi_s(e_s(G^\varepsilon))}$, there exists a subset $T\subset (0,1)$ of  full $\L^1$ measure such that for all $H\subset G^{\varepsilon}_{\as}$  the map
\[
T\ni t \mapsto \int_{\ee_s(H)} \rho^{\varepsilon}_{t}(g^{\as}(\alpha, t))(\partial_{\tau}|_{\tau=t} \Phi^{\tau}_s(g^{\as}(\alpha,t)))^{-1} h^{\as}_{\alpha}(t) \ell^p_s(\alpha)\m_{\as}^{\varepsilon,s}(d\alpha),
\]
is constant. 
In particular, since any Borel subset of $\ee_s(G_{\as})$ can be written in the form $\ee_s(H)$, we have that for $t,t'\in{T}$
\begin{equation}
\label{equ:confrontodensita}
\rho^{\varepsilon}_{t'}(\gamma_{t'})(\partial_{\tau}|_{\tau=t'} \Phi^{\tau}_s(\gamma_{t'}))^{-1} h^{\as}_{\gamma_s}(t') =\rho^{\varepsilon}_{t}(\gamma_{t})(\partial_{\tau}|_{\tau=t} \Phi^{\tau}_s(\gamma_{t}))^{-1} h^{\as}_{\gamma_s}(t),
\end{equation}
for $\m_{\as}^{\varepsilon,s}$-a.e. $\alpha\in{\ee_s(G^{\varepsilon}_{\as})}$ where $\gamma=\ee_s^{-1}(\alpha)=g^{\as}(\alpha,\cdot)\in{G^{\varepsilon}_{\as}}$, 
with the exceptional set depending on $t,t'$. 
Recall that, by Corollary  \ref{cor:1010}, given $t'\in{T}$,  
$\partial_{\tau}|_{\tau=t'} \Phi^{\tau}_s(\gamma^{\alpha}_{t'})$ exists for 
$\mm_{\as}^{\varepsilon,s}$-a.e. $\alpha\in \ee_s(G^{\varepsilon}_{\as})$.
Thus, in particular, the equality \eqref{equ:confrontodensita} holds for a  countable sequence of $\{ t'\} \subset T$ dense in $(0,1)$. 
Using  the normalization $h_{\gamma_s}^{\as}(s)=1$, the continuity of $h^{\as}_{\gamma_s}(\cdot)$, $\rho^{\varepsilon}_{\cdot}(\gamma_{\cdot})$ and the fact that
\[
\lim_{T\ni t'\to s} \partial_{\tau}|_{\tau=t'} \Phi^{\tau}_s(\gamma^{\alpha}_{t'})= \ell_s(\gamma_s^{\alpha})^p=\ell(\gamma^{\alpha})^p,
\]
 it is possible to pass to the limit for $t'\to s$ in \eqref{equ:confrontodensita}
\begin{equation}
\rho^{\varepsilon}_{s}(\gamma_s)\ell(\gamma)^{-p}=\rho^{\varepsilon}_{t}(\gamma_t)( \partial_{\tau}|_{\tau=t} \Phi_s^{\tau}(\gamma_t))^{-1}h^{\as}_{\gamma_s}(t),
\end{equation}
for $\mm_{\as}^{\varepsilon,s}$-a.e. $\alpha\in \ee_s(G^{\varepsilon}_{\as})$, 
with $\gamma=\ee_s^{-1}(\alpha)\in{G^{\varepsilon}_{\as}}$.

By corollary \ref{cor:1010}, the measures $\mm_{\as}^{\varepsilon, s}$ and 
$(\ee_s)_{\sharp}\nu_{\as}^{\varepsilon}$ are mutually absolutely continuous for 
$\qq_s^{\varepsilon,s}$-a.e. $\awos\in{\varphi_s(\ee_s(G^{\varepsilon}))}$. 
In particular, this implies that for all $s\in{(0,1)}$, for 
$\qq_s^{\varepsilon,s}$-a.e. $\awos\in{\varphi_s(\ee_s(G^{\varepsilon}))}$   
and $\L^1$-a.e. $t\in{(0,1)}$,  the equality \eqref{equ:confrontodensita} holds  for $\nu_{\as}$-a.e. $\gamma$. 
By Corollary \ref{cor:1010}, it follows that the measures  $\qq_s^{\varepsilon,s}$ and  $\qq_s^{\varepsilon,\nu}$ are mutually absolutely continuous; thus, by the disintegration formula \eqref{equ:disintnu}, it follows that for all $s\in{(0,1)}$ and $\L^1$-a.e. $t\in{(0,1)}$:
\[
\rho^{\varepsilon}_{s}(\gamma_s)\ell(\gamma)^{-p}=\rho^{\varepsilon}_{t}(\gamma_t)( \partial_{\tau}|_{\tau=t} \Phi_s^{\tau}(\gamma_t))^{-1}h_{\gamma_s}^{\varphi_s(\gamma_s), s}(t),
\]
for $\nu$-a.e. $\gamma\in{G^{\varepsilon}}$. 
Passing to the limit as $\varepsilon\to 0$ along the chosen sequence, it turns out that all $s\in{(0,1)}$, $\L^1$-a.e. $t \in{(0,1)}$ and $\nu$-a.e. $\gamma\in{G_{\varphi}^{+}}$  satisfy
\[
\rho_{s}(\gamma_s)\ell(\gamma)^{-p}=\rho_{t}(\gamma_t)( \partial_{\tau}|_{\tau=t} \Phi_s^{\tau}(\gamma_t))^{-1}h_{\gamma_s}^{\varphi_s(\gamma_s), s}(t).
\]
By Fubini 's Theorem, for $\nu$-a.e. $\gamma\in{G^{+}_{\varphi}}$, we have that \eqref{equ:changeofv} holding for $\L^1$-a.e. $s,t\in{(0,1)}$.
\end{proof}

\begin{remark}
\label{rem:phiphibar}

All of the results of this section also hold for $\bar \Phi_s^t$ in place of $\Phi_s^t$. Indeed, recall that for all $x \in X$, $\Phi_s^t(x) = \Phic_s^t(x)$ for $t \in \mathring G_\varphi(x)$, and that by Proposition \ref{prop:Phi}, 
$\partial_t \Phi_s^t(x) = \partial_t \Phic_s^t(x)$ for a.e. $t \in \mathring G_\varphi(x)$. As these were the only two properties used in the above derivation  the assertion follows. 
\end{remark}
By Proposition \ref{prop:Phi}, we know that the differentiability points of $\tau \mapsto \tilde{\Phi}^{\tau}_s(x)$ and $\tau \mapsto \tilde{\ell}_{\tau}^p(x)$ coincide for all $\tau\neq s$ and at these points 
\[
\partial_{\tau} \tilde{\Phi}^{\tau}_s(x)= \tilde{\ell}^p_{\tau}(x)+(\tau-s)\partial_{\tau} \frac{\tilde{\ell}^p_{\tau}(x)}{p}.
\]
Hence by Remark \ref{rem:phiphibar}, we deduce that for $\nu$-a.e. geodesic $\gamma\in{G^{+}_{\varphi}}$ and for a.e. $t\in{(0,1)}$  both quantities 
\[
  \partial_{\tau}|_{\tau=t} {\ell^{p}_{\tau}(\gamma_t)} 
= \partial_{\tau}|_{\tau=t} {\bar{\ell}^{p}_{\tau}(\gamma_t)}
\]
exist and  coincide.  
We can therefore rewrite the change of variable formula in the following way: for $\nu$-a.e. geodesic $\gamma\in{G^{+}_{\varphi}}$
\begin{equation}
\label{equ:change}
\frac{\rho_s(\gamma_s)}{\rho_t(\gamma_t)}=\frac{h_{\gamma_s}^{\varphi_s(\gamma_s) ,s}(t)}{1+(t-s)\frac{\partial_{\tau}|_{\tau=t} \ell_{\tau}^p(\gamma_t)}{ p\ell(\gamma)^p}}
=\frac{h_{\gamma_s}^{\varphi_s(\gamma_s) ,s}(t)}{1+(t-s)\partial_{\tau}|_{\tau=t} \log \bar{\ell}_{\tau}(\gamma_t)}, 
\quad \text{for a.e.}\,\, t,s \in{(0,1)}.
\end{equation}
\noindent
For sake of brevity, once the geodesic $\gamma$ is fixed, we will use the following notation: 
$\rho(t)=\rho_t(\gamma_t)$, $h_s(t):=h_{\gamma_s}^{\varphi_s(\gamma_s)}(t)$ and $K_0=K\cdot\ell(\gamma)^2$. 
We recall that, by Corollary \ref{cor:version} and \eqref{E:h}, given by Theorem \ref{teo:cm}, the following properties hold true for $\nu$-a.e $\gamma\in{G^{+}_{\varphi}}$:
\begin{itemize}
\item[(A)] $(0,1)\ni t\mapsto {\rho(t)}$ is locally 
Lipschitz and strictly positive.

\item[(B)] For all $s\in{(0,1)}$, $h_s$ is 
a $\CD(K_0,N)$  density on $[0,1]$ satisfying $h_s(s)$=1.
\end{itemize}
Fix now a geodesic $\gamma\in{G^{+}_{\varphi}}$ satisfying the change of variable formula \eqref{equ:rig}, (A), (B) above.

The formula \eqref{equ:change} implies that there exists a set $I\subset (0,1)$ of full measure such that for all $s\in{I}$ the functions
\[
t \mapsto \partial_{\tau}|_{\tau=t} \frac{\tilde \ell^{p}_{\tau}/p(\gamma_t)}{\tilde  \ell(\gamma)^p},\quad t\mapsto z_s(t):=
\frac{\frac{\rho(t)} {\rho(s)} h_{s}(t) -1}{t-s}
\]
coincide a.e. on $(0,1)$ for both $\tilde{\ell} \in \{\ell,\bar{\ell}\}$, 
with $z_{s}$ defined on $(0,1) \setminus \{s\}$.
Hence, by continuity,  the functions $\{z_s\}_{s\in{I}}$ must all coincide, where defined, with a unique function $t\mapsto z(t)$ defined  on $(0,1)$ such that
\begin{equation}\label{def-z}
z(t) = \frac{\partial}{\partial{\tau}}\bigg|_{\tau=t} \log{\ell_{\tau}(\gamma_t)}
= \frac{\partial}{\partial{\tau}}\bigg|_{\tau=t} \log {\bar{\ell}_{\tau}(\gamma_t)}
,\,\, \text{for a.e.} \,\,t\in{(0,1)}.
\end{equation}

Since $\CD(K,N)$ densities are locally Lipschitz in the interior of the domain where they are defined,
we see that  $z$ is locally Lipschitz in $(0,1)$ from \eqref{equ:rig}. Combining \eqref{def-z} with
the third order information provided by Theorem \ref{teo:zz} (up to constant factors) yields:
\begin{itemize}
\item[(C)]$(0,1)\ni t \mapsto z(t)$ is locally Lipschitz. Moreover, for any $\delta\in{(0,1/2)}$ there exists $C_{\delta}>0$ so that:
\[
\frac{z(t)-z(s)}{t-s}\geq (1-C_{\delta}(t-s))|z(s)||z(t)|, \,\, \forall \, 0 < \delta \leq s<t\leq 1-\delta<1.
\]
In particular,  $z'(t)\geq z^2(t)$ for a.e. $t\in{(0,1)}$.
\end{itemize}

\noindent
To summarize, the change of variable formula can be rewritten in the following form:
\begin{equation}
\label{equ:rig}
\frac{\rho(s)}{\rho(t)}=\frac{h_s(t)}{1+(t-s)z(t)},\qquad \text{for all}\ t,s\in{(0,1)},
\end{equation}
where $z(t)$ coincides for all $t \in (0,1)$ with the second Peano derivative of $\tau \mapsto \varphi_{\tau}(\gamma_t)$ and of $\tau \mapsto \varphic_{\tau}(\gamma_t)$ at $\tau = t$. 
These second Peano derivatives exist for all $t \in (0,1)$ and are a continuous function. 
We are therefore in position to obtain the aforementioned 
factorization of the ``Jacobian''. 
It has been already proved in \cite{CMi} (see Theorem 12.3) that 
properties (A), (B), (C) together with the change of variable formula 
\eqref{equ:rig} are enough to obtain a factorization of the real function 
$1/\rho(t)$ into a product $L(t)Y(t)$,  
in which the first factor $L(t)$ is concave due to dilational and
dimensional effects (analogous to the Brunn-Minkowski inequality on 
$(\R^n,|\cdot|,\L^n)$), 
while the latter term $Y(t)$ captures the effects of the curvature of $(X,\sfd,\mm)$.
In the smooth case $\rho(t)^{-1/n}$ would be interpreted as the mean-free path between particles during transport.
\begin{theorem}[Isolating curvature effects in the volume distortion along the direction transported
{\cite[Theorem 12.3]{CMi}}]\label{T:decomposition}
If the change of variable formula \eqref{equ:rig} holds   and the properties (A), (B), (C) are satisfied, then
\[
\frac{1}{\rho_t(\gamma_t)}= L(t)Y(t) \quad \forall t\in{(0,1)},
\]
where $L$ is concave and $Y$ is a $\CD(K_0,N)$ density on $(0,1)$.
\end{theorem}
%
%%%%%%%%%%%%%%%%%%%%%%%%%%%%%%%%%%%%%%%%%%%%%%%%%%%%%%%%%%%%%%%%%%%%%%%%%%
%%%%%%%%%%%%%%%%%%%%%%%%%%%%%%%%%%%%%%%%%%%%%%%%%%%%%%%%%%%%%%%%%%%%%%%%%%
\subsection{Main Theorems}
Finally, putting together the result proved so far in 
Section \ref{S:firstimplication} and Section \ref{S:anyp} we close the circle by proving:
\begin{theorem}[Non-branching $\CD_p$ spaces are $\CD^1_{Lip}$ hence $\CD_q$]
\label{T:main2}
Let $(X,\sfd,\m)$ be a $p$-essentially non-branching 
m.m.s.\ verifying $\CD_p(K,N)$ for some $p >1$. 
If $(X,\sfd,\m)$ is also $q$-essentially non-branching for some $q > 1$, then 
 it verifies $\CD_q(K,N)$.
\end{theorem}
\begin{proof}
Consider $\mu_0,\mu_1\in{\mathcal{P}_q(X,\sfd,\mm)}$. 
 Recall that $\CD_p(K,N)$ implies $(X,\sfd)$ to be  a geodesic space,
hence the same is true for $(\mathcal{P}_{q}(X),W_{q})$.
Moreover, it implies  $(X,\sfd)$ is $\MCP(K,N)$, hence qualitatively 
non-degenerate.
Since $(X,\sfd,\mm)$ is assumed to be  $q$-essentially non-branching, Theorem~\ref{teo:kell} yields
a unique $\nu \in \OptGeo_{q}(\mu_0,\mu_1)$ and 
$$
[0,1] \ni t \mapsto \mu_t:=(\ee_t)_{\sharp}\nu \ll \mm.
$$
Let   $\rho_t:=d\mu_t/d\m$ be the versions of the densities guaranteed by Corollary \ref{cor:version}. 

Finally let $\varphi:X\to \mathbb{R}$ be a Kantorovich potential for the  optimal transport problem from $\mu_0$ to $\mu_1$, with cost $c:=\sfd^q/q$. Recall that $G_{\varphi}\subset \Geo(X)$ denote the set of ($\varphi, q)$-Kantorovich geodesics, i.e. all the geodesics $\gamma$ for which
\[
\varphi(\gamma_0)+\varphi^c(\gamma_1)=\frac{\sfd^q(\gamma_0,\gamma_1)}{q}.
\]
As already observed, $\nu$ will be concentrated on $G_{\varphi}=G_{\varphi}^{+}\cup G_{\varphi}^{0}$, where $G_{\varphi}^{+}$ and $G_{\varphi}^{0}$ denote the subsets of positive and zero length  ($\varphi, q)$-Kantorovich geodesics  respectively.

\noindent
By the change of variables formula obtained in
 Theorem \ref{teo:changeof}  (which relies on the $\CD^1_{Lip}(K,N)$ conclusion of Theorem \ref{teo:cm}),  
for $\nu$-a.e. geodesic $\gamma\in{G_{\varphi}^{+}}$:
\begin{equation}
\frac{\rho_s(\gamma_s)}{\rho_t(\gamma_t)}=\frac{h_{\gamma_s}^{\varphi_s(\gamma_s) ,s}(t)}{1+(t-s)\frac{\partial_{\tau}|_{\tau=t} \ell_{\tau}^p(\gamma_t)}{ p\ell(\gamma)^p}}=\frac{h_{\gamma_s}^{\varphi_s(\gamma_s) ,s}(t)}{1+(t-s)\frac{\partial_{\tau}|_{\tau=t} \bar{\ell}_{\tau}^p(\gamma_t)}{ p\ell(\gamma)^p}}, \quad \text{for a.e.}\,\, t,s \in{(0,1)}
\end{equation}
where for all $s\in{(0,1)}, h_s=h_{\gamma_s}^{\varphi_s(\gamma_s) ,s}$ is a $\CD(K_0,N)$ density, with $K_0=\ell(\gamma)^2K$ and $h_s(s)=1$.
Since Corollary \ref{cor:version} implies the Lipschitz regularity of $t\mapsto \rho_t(\gamma_t)$, assumptions (A) and 
(B) of the Theorem \ref{T:decomposition} are satisfied. Moreover, the third order information on the Kantorovich potential $\varphi$ guarantees also the validity of the assumption (C) of the Theorem \ref{T:decomposition}.
Hence for $\nu$-a.e. $\gamma\in{G^{+}_{\varphi}}$, it holds 
\[
\frac{1}{\rho_t(\gamma_t)}= L(t)Y(t),\,\, \forall t\in{(0,1)}
\]
 where $L$ is a concave function and $Y$ is a $\CD(K_0,N)$ density on $(0,1)$. 

It is now a standard application of H\"older's inequality 
that gives us the validity of the $\CD_{q}(K,N)$ inequality along 
the $W_{q}$-geodesic $\mu_{t}$:
fix $t_0,t_1\in{(0,1)}$ and set $t_{\alpha}=\alpha t_1+(1-\alpha)t_0$, where $\alpha\in{[0,1]}$. 
Using that $\sigma_{K_0,N}^{(\alpha)}(\theta)=\sigma_{K,N}^{(\alpha)}(\theta \ell(\gamma))$, it holds true:
\begin{align}
\label{equ:cd1dim}
\rho_{t_{\alpha}}^{-\frac{1}{N}}(\gamma_{t_{\alpha}})&= L^{\frac{1}{N}}(t_{\alpha})Y^{\frac{1}{N}}(t_{\alpha})\notag \\
&\geq \bigl(\alpha L(t_1)+(1-\alpha)L(t_0)\bigr)^{\frac{1}{N}}\cdot \bigl( \sigma_{K_0,N-1}^{(\alpha)}(|t_1-t_0|)Y^{\frac{1}{N-1}}(t_1)+ \sigma_{K_0,N-1}^{(1-\alpha)}(|t_1-t_0|)Y^{\frac{1}{N-1}}(t_0) \bigr)^{\frac{N-1}{N}}\notag\\
&\geq \alpha^{\frac{1}{N}}\sigma_{K_0,N-1}^{(\alpha)}(|t_1-t_0|)^{\frac{N-1}{N}}Y^{\frac{1}{N}}(t_1)L^{\frac{1}{N}}(t_1) +(1-\alpha)^{\frac{1}{N}}\sigma_{K_0,N-1}^{(1-\alpha)}(|t_1-t_0|)^{\frac{N-1}{N}}Y^{\frac{1}{N}}(t_0)L^{\frac{1}{N}}(t_0)\notag\\
&=\alpha^{\frac{1}{N}}\sigma^{(\alpha)}_{K,N-1}(|t_1-t_0|\ell(\gamma))^{\frac{N-1}{N}}\rho_{t_1}^{-\frac{1}{N}}(\gamma_{t_1})+(1-\alpha)^{\frac{1}{N}}\sigma_{K,N-1}^{(1-\alpha)}(|t_1-t_0|\ell(\gamma))^{\frac{N-1}{N}}{\rho}_{t_0}^{-\frac{1}{N}}(\gamma_{t_0})\notag\\
&=\tau_{K,N}^{(\alpha)}(\sfd(\gamma_{t_0},\gamma_{t_1}))\rho_{t_1}^{-\frac{1}{N}}(\gamma_{t_1})+\tau_{K,N}^{(1-\alpha)}(\sfd(\gamma_{t_0},\gamma_{t_1}))\rho_{t_0}^{-\frac{1}{N}}(\gamma_{t_0}).
\end{align}
Recall that, by Corollary \ref{cor:version}, the function $t\mapsto \rho_t(\gamma_t)$ is upper semi-continuous at the endpoints; so, it follows that  for $\nu$-a.e. $\gamma\in{G_{\varphi}^{+}}$ the inequality \eqref{equ:cd1dim} holds true for all $t_0,t_1\in{[0,1]}$.   In particular, setting $t_0=0,t_1=1$, we have that  for all $\alpha\in{[0,1]}$:
\begin{equation}
\label{equ:cd1dimfin}
\rho_{\alpha}^{-\frac{1}{N}}(\gamma_{\alpha})\geq \tau_{K,N}^{(\alpha)}(\sfd(\gamma_{0},\gamma_{1}))\rho_{1}^{-\frac{1}{N}}(\gamma_{1})+\tau_{K,N}^{(1-\alpha)}(\sfd(\gamma_{0},\gamma_{1}))\rho_{0}^{-\frac{1}{N}}(\gamma_{0});
\end{equation}
the latter inequality being satisfied for $\nu$-a.e.$\gamma\in{G^{+}_{\varphi}}$. We now claim that  \eqref{equ:cd1dimfin} is also satisfied for every $\gamma\in{G^{0}_{\varphi}}$, confirming in this way  the validity of the $\CD(K,N)$ condition. Indeed, in this case  the map $\alpha \mapsto \rho_{\alpha}(\gamma_\alpha)$  turns out to be constant by the Theorem \ref{teo:changeof}  and then \eqref{equ:cd1dimfin} is trivially satisfied  as an equality, since $\tau_{K,N}^{(\alpha)}(0)=\alpha$, for every $\alpha\in{[0,1]}$. Thus, the claim.
\end{proof}
\begin{corollary}[Local-to-Global]\label{C:loctoglob}
Fix any $p>1$ and $K,N \in \R$ with $N>1$. 
Let $(X,\sfd,\mm)$ be a $p$-essentially non-branching metric measure space verifying $\CD_{p,loc}(K,N)$ and such that 
$(X,\sfd)$ is a length space with $\spt (\mm) = X$. 
Then $(X,\sfd,\mm)$ verifies $\CD_{p}(K,N)$.
\end{corollary}

\end{document}

%% file: image.pdf_tex
%% Creator: Inkscape inkscape 0.92.5, www.inkscape.org
%% Accompanies image file 'image.pdf' (pdf, eps, ps)
\begingroup%
  \makeatletter%
  \providecommand\color[2][]{%
    \errmessage{(Inkscape) Color is used for the text in Inkscape, but the package 'color.sty' is not loaded}%
    \renewcommand\color[2][]{}%
  }%
  \providecommand\transparent[1]{%
    \errmessage{(Inkscape) Transparency is used (non-zero) for the text in Inkscape, but the package 'transparent.sty' is not loaded}%
    \renewcommand\transparent[1]{}%
  }%
  \providecommand\rotatebox[2]{#2}%
  \newcommand*\fsize{\dimexpr\f@size pt\relax}%
  \newcommand*\lineheight[1]{\fontsize{\fsize}{#1\fsize}\selectfont}%
  \ifx\svgwidth\undefined%
    \setlength{\unitlength}{504bp}%
    \ifx\svgscale\undefined%
      \relax%
    \else%
      \setlength{\unitlength}{\unitlength * \real{\svgscale}}%
    \fi%
  \else%
    \setlength{\unitlength}{\svgwidth}%
  \fi%
  \global\let\svgwidth\undefined%
  \global\let\svgscale\undefined%
  \makeatother%
  \begin{picture}(1,0.94940476)%
    \lineheight{1}%
    \setlength\tabcolsep{0pt}%
    \put(0,0){\includegraphics[width=\unitlength,page=1]{image.pdf}}%
  \end{picture}%
\endgroup%